\pdfoutput=1
\documentclass [12pt, reqno]{amsart}
\usepackage{ mathrsfs }

\usepackage{amssymb,latexsym,amsmath,amsfonts, eucal}

\newtheorem{thm}{Theorem}[section]
\newtheorem{cor}[thm]{Corollary}
\newtheorem{lem}[thm]{Lemma}
\newtheorem{prop}[thm]{Proposition}
\theoremstyle{definition}
\newtheorem{defn}[thm]{Definition}
\newtheorem{rem}[thm]{Remark}
\numberwithin{equation}{section}

\newcommand{\beas}{\begin{eqnarray*}}
\newcommand{\eeas}{\end{eqnarray*}}
\newcommand{\bes} {\begin{equation*}}
\newcommand{\ees} {\end{equation*}}
\newcommand{\be} {\begin{equation}}
\newcommand{\ee} {\end{equation}}
\newcommand{\bea} {\begin{eqnarray}}
\newcommand{\eea} {\end{eqnarray}}
\newcommand{\ra} {\rightarrow}
\newcommand{\txt} {\textmd}

\newcommand{\C} {\mathbb{C}^{n}}

\begin{document}
\title[Hecke-Bochner identity and eigenfunctions]{Hecke-Bochner identity and eigenfunctions associated to Gelfand pairs on the Heisenberg group}


\author[AMIT Samanta]{AMIT Samanta}
\address{Department of Mathematics, Indian Institute of Science, Bangalore 560012}

\email[Amit Samanta]{amit@math.iisc.ernet.in}

\thanks{This work is supported in 
part by grant from UGC Centre for Advanced Study and in part by research fellowship of the Indian Institute of Science.}

\begin{abstract}
Let $\mathbb{H}^{n}$ be the $(2n+1)$-dimensional Heisenberg group, and let $K$ be a compact subgroup of $U(n)$, such that $(K,\mathbb{H}^{n})$ is a Gelfand pair. Also assume that the $K$-action on $\C$ is polar. We prove a Hecke-Bochner identity associated to the Gelfand pair $(K,\mathbb{H}^{n})$. For the special case $K=U(n)$, this was proved by Geller \cite{G}, giving a formula for the Weyl transform of a function $f$ of the type $f=Pg$, where $g$ is a radial function, and $P$ a bigraded solid $U(n)$-harmonic polynomial. Using our general Hecke-Bochner identity we also characterize (under some conditions) joint eigenfunctions of all differential operators on $\mathbb{H}^{n}$ that are invariant under the action of $K$ and the left action of $\mathbb{H}^{n}$.  
\vspace*{0.1in}

\begin{flushleft}
\vspace*{0.1in}

\end{flushleft}

\end{abstract}

\maketitle

\section{Introduction}

This paper is concerned with two fundamental problems in Harmonic analysis on the Heisenberg group, $\mathbb{H}^n$. The first one is the Hecke-Bochner identity and the second one is a characterization of joint eigenfunctions for a certain family of invariant differential operators on $\mathbb{H}^n$. We first briefly recall the known results in this direction.

The Hecke-Bochner identity on $\mathbb{R}^n$ states that (see \cite{SW}, Theorem-3.10, page-158) the Fourier transform of a function $f=Pg$, where $P$ is a homogeneous solid $SO(n)$-harmonic polynomial (of degree $k$ say) and $g$ is radial, is given by $\widehat{Pg}=Ph$, where $h$ is a radial function given by $$h(r)=i^{-
k}\int_{s=0}^{\infty}g(s)\frac{J_{\frac{n}{2}+k-1}(rs)}{(rs)^{\frac{n}
{2}+k-1}}s^{n+k-1}ds,$$ where $J_{\frac{n}{2}+k-1}$ is the Bessel's function of order 
$\frac{n}{2}+k-1$.
Secondly, any eigenfunction $\varphi$ of $\triangle$, the Laplacian on $\mathbb{R}^n$, with eigenvalue $-\lambda^2$ is given by the integral representation $$\varphi(x)=\int_{S^{n-1}}e^{i\lambda x\cdot\omega}dT(\omega),$$ where $T$ is a certain analytic functional. See Helgason (\cite{H1}, Theorem 2.1, page-5) for $n=2$ and Hashizume et al \cite{HKMO} for general case. Both these results can be interpreted in terms of harmonic analysis on the Gelfand pair $\big(\mathbb{R}^n\ltimes SO(n),SO(n)\big)$. Note that a solid homogeneous harmonic polynomial of degree $k$ is an element which transforms according to a class one representation of $SO(n)$. Next, the Laplacian $\triangle$ is the generator of $\mathbb{R}^n\ltimes SO(n)$ invariant differential operators on $\mathbb{R}^n$. This point of views have a natural generalization to other homogeneous spaces.

In the context of Riemannian symmetric spaces $X=G/K$, Helgason (\cite{HS2}, Corollary 7.4) characterized all $K$-finite joint eigenfunctions for $D(G/K)$. The characterization of arbitrary joint eigenfunctions for $D(G/K)$ was done by Helgason (\cite{HS1}, Chapter IV, Corollary 1.6) when $\textup{rank} X=1$ and by Kashiwara et al \cite{KKMOOT} in the general case. A Hecke-Bochner type identity was established, when $X$ is of rank one, by Bray \cite{Br}. For general case see \cite{H2}, Chapter-III, Corollary 5.5. 
  
In this paper, we consider these two questions on the Heisenberg group associated to 
the Gelfand pair $(K,\mathbb{H}^{n})$, where $K\subset U(n)$ and the $K$-action on $\C$ is polar. We prove a Hecke-Bochner type identity (Theorem \textbf{\ref{eqn 
t7.4}}), giving a  formulae for the Weyl transform of a function which transforms 
according to a class one representation of $K$. We will see that the formulae involves 
generalized $K$-spherical functions, as in the case of Euclidean spaces and Riemannian 
symmetric spaces. For the special case $K=U(n)$ this was already proved by Geller 
(\cite{G}, Theorem 4.2). Let $\mathcal{L}_{K}(\textbf{h}_n)$ be the algebra of all 
differential operators on $\mathbb{H}^{n}$ that are invariant under the action of $K$ 
and the left action of $\mathbb{H}^{n}$. Any joint eigenfunction of all 
$D\in\mathcal{L}_{K}(\textbf{h}_n)$ has to be of the form $f(z,t)=e^{i\lambda t}g(z)$ 
for some complex number $\lambda$. Following the view point of Thangavelu in \cite{T}, under the assumptions that $\lambda$ is non-zero 
real and $e^{-(|\lambda|-\epsilon)|z|^{2}}|g(z)|\in L^{p}(\C)$ for some $\epsilon>0$ 
and $1\leq p\leq\infty$, we characterize all $K$-finite joint eigenfunctions $f(z,t)$ 
of all $D\in\mathcal{L}_{K}(\textbf{h}_{n})$, in terms of the representations of the 
Heisenberg group (Theorem \textbf{\ref{eqn t8.3}}). We extend this result for arbitrary (with the same growth condition) joint eigenfunctions, when dim$V^{M}_{\delta}=1$ for all class one representations $\delta$ of $K$; here $M$ is the stabilizer of a $K$-regular point, $V_{\delta}$ is the (finite dimensional) Hilbert space where the representation $\delta$ is realized and $V_{\delta}^{M}$ is the space of $M$-fixed vectors in $V_{\delta}$. This can be put in a different form, giving an integral representation of eigenfunctions, which for $K=U(n)$ is precisely Theorem 4.1 in \cite{T}. We also obtain a different integral representation with an explicit kernel.

The plan of the paper is as follows. In section 2., we recall the definition of polar action of 
$K\subset SO(n)$ on $\mathbb{R}^{n}$, develop a system of polar 
coordinates and state some results about polar actions. In section 3., we show that 
the Kostant-Rallis Theorem holds for polar actions i.e each $K$-harmonic polynomial is 
determined by its values on a regular $K$-orbit. We also discuss the class one 
representations of $K$ realized on the space of $K$-harmonic polynomials and on the 
space of their restriction to a regular $K$-orbit. In section 4., for a class one 
representation $\delta$ of $K$, we consider $\delta$ type 
Hom$(V_{\delta},V_{\delta})$-valued functions $G$ i.e $G:
\mathbb{R}^{n}\rightarrow$Hom$(V_{\delta},V_{\delta})$ such that $G(k\cdot 
x)=\delta(k)G(x)$. 
We show that such a $G$ can be written in a 
special form, which, for the case $K=SO(n)$, is equivalent to considering a 
function of type $Pg$, where $P$ is a solid homogeneous $SO(n)$-harmonic polynomial of 
certain degree and $g$ is radial. In section 5., we mainly recall some basic facts 
related to the Heisenberg group, its representations and Weyl transform. We also state some 
results about Gelfand pairs and bounded $K$- spherical functions from \cite{BJR}. 
Section 6., deals with the Weyl transform of $K$-invariant functions. In section 7., we 
prove the main results of this paper. We start with defining generalized $K$-
spherical functions, prove a Hecke-Bochner type identity for the Weyl transform. Using 
this we prove the uniqueness (upto a right multiplication by a constant matrix) of 
generalized $K$-spherical functions. We also give a formulae of generalized $K$-
spherical function in terms of the representations of Heisenberg group. This formulae 
together with the uniqueness of generalized $K$-spherical functions will imply  
characterizations of $K$-finite joint eigenfunctions (with the usual growth condition) of all $D\in\mathcal{L}_{K}
(\textbf{h}_{n})$, which we present in section 8. Section 9. deals with square integrable (modulo the center) joint eigenfunctions. In the final section, we discuss the special case when dim$V^{M}_{\delta}=1$ for all class one representations $\delta$ of $K$.

\section{Polar actions and coordinates}
In this section we recall polar actions and develop a system of polar coordinates on 
the spaces upon they act. References for this section are Conlon \cite{C}, Dadok 
\cite{D} and Lander \cite{L}. Let $K$ be a compact connected subgroup of $SO(n)$ which 
acts naturally on $\mathbb{R}^{n}$. Let $\mathfrak{k}$ be the Lie algebra of $K$. We 
denote the inner product on $\mathbb{R}^{n}$ by $(.,.)$. Let $N_{x}:=\{k\cdot x :k\in 
K\}$ be the $K$-orbit through $x$, and $K_{x}:=\{k:k\cdot x=x\}$ be the isotropy 
subgroup of $x$, hence $N_{x}\cong K/K_{x}$. A $K$-orbit of maximal dimension is 
called a $regular~~ orbit$, and any point on a regular orbit is called a $regular 
~~point.$ A $K$-orbit through a point $x$ is called a $principal ~~orbit$ if $K_x$ is a subgroup of a conjugate of any other isotropy subgroup. 
Clearly any principal orbit is also a regular orbit. The action of $K$ on 
$\mathbb{R}^{n}$ is called $polar ~~action$ if there is a linear subspace $T$ of 
$\mathbb{R}^{n}$ which meets every $K$-orbit and is orthogonal to the $K$-orbit at 
every point i.e $(\mathfrak{k}\cdot x,T)=0$ for all $x\in T$. Such a linear subspace 
$T$ is called a $K-transversal~~ domain.$ This is precisely the condition $(A)$ in the 
introduction of \cite{C}. Then dim $(T)=$ dim $(\mathbb{R}^{n})-$ dim. of a regular 
orbit (\cite{C}, Proposition 1.1). Therefore if we take a regular point $x\in T$ then 
clearly $A_{x}=T$ where $A_{x}=\{y\in\mathbb{R}^{n}:(y,\mathfrak{k}\cdot x)=0\}.$ 
Consequently $A_{x}$ meets all the orbits orthogonally. Hence the above definition of polar 
action is equivalent to that of Dadok \cite{D}. Also, for polar action any orbit of 
maximal dimension is principal (\cite{C}, Proposition 2.2). Therefore, regular orbits 
and principal orbits are equivalent for polar action.

From now on we always assume that $K$ is a compact connected subgroup of $SO(n)$ whose 
action on $\mathbb{R}^{n}$ is polar. We state some results from Conlon \cite{C} and 
derive some easy consequences. Since regular orbits and principal orbits are same, we 
only use the word ``regular orbit" instead of using both. As mentioned above we have,

\begin{prop}\label{eqn p2.1}
\textup{(Conlon \cite{C}, Proposition 1.1)} Let $N\subset \mathbb{R}^{n}$ be a K-orbit 
of maximal dimension. Then \textup{dim(}N\textup{)} $=$ 
\textup{dim(}$\mathbb{R}^{n}$\textup{)}$-$\textup{dim(}T\textup{)}.
\end{prop}

\begin{thm}\label{eqn t2.2}
\textup{(Conlon \cite{C}, Theorem II)} Let $T\subset \mathbb{R}^{n}$ be a K-
transversal domain. Then there is a finite collection ${P_{1}, P_{2},\cdots, P_{r}}$ 
of hyperplanes in T, together with positive integers m(i), $i=1,2,\cdots, r$, such 
that for each $x\in T$,
$$\textup{dim(}N_{x}\textup{)}=\textup{dim(}\mathbb{R}^{n}\textup{)}-\textup{dim(}T\textup{)}-\sum_{i\in 
I_{x}}m(i),$$
where $I_{x}=\{i:x\in P_{i}\}.$
\end{thm}

\begin{defn}\label{eqn d2.3}
Each $P_{i}$ as above is called a $singular~~ variety$ of multiplicity $m(i)$, and 
each connected component of $T\smallsetminus\cup P_{i}$ a $Weyl~~ domain$ in $T$. The 
$Weyl~~ group$ $W = W(K,T)$ is the group of transformations of $T$ consisting of those 
$k\in K$ such that $k\cdot T=T.$
\end{defn}

\begin{thm}\label{eqn t2.4}
\textup{(Conlon \cite{C}, Theorem III)} If $T$ is a K-transversal domain, then the 
orthogonal reflection of T in each singular variety $P_{i}$ exists, W is a finite 
group generated by all such reflections, and W permutes simply transitively the set of 
Weyl domains in T. If $x\in T$ lies on no singular variety, then W permutes simply 
transitively the set $N_{x}\cap T$.
\end{thm}

Fix a Weyl domain $T^{+}$ in $T$. As an easy consequence of the above three results we 
get the following corollary.

\begin{cor}\label{eqn c2.5}
All the points of $T^{+}$ are regular, and each regular $K-$orbit intersects $T^{+}$ 
exactly at one point.
\end{cor}

\begin{lem}\label{eqn l2.6}
If $x\in T$ is regular then $K_{x}=K_{T},$ where $K_{T}:=\{k\in K:k\cdot q=q,~~\forall 
q\in T\}$ is the stabilizer of $T$.
\end{lem}

\begin{proof}
See the proof of Proposition 1.1 in Conlon \cite{C}.
\end{proof}

Let $M=K_{T}$ be as defined in the above lemma. Define the ``polar coordinate 
mapping''
$$\phi:T^{+}\times K/M\longrightarrow \mathbb{R}^{n}~\textup{ by }~\phi(r,kM)=k\cdot 
r.$$
Clearly $\phi$ is well defined and by Corollary \textbf{\ref{eqn c2.5}}, its image is 
precisely the set of all regular points. If $k_{1}\cdot r_{1}=k_{2}\cdot r_{2}$ for 
$k_{1},k_{2}\in K$ and $r_{1},r_{2}\in T^{+}$, then $K$- orbits through $r_{1}$ and 
$r_{2}$ are same. By Corollary \textbf{\ref{eqn c2.5}}, $r_{1}=r_{2}=r$ (say). 
Consequently $k_{1}^{-1}k_{2}$ fixes $r$ and hence belongs to $K_{T}$ by Lemma 
\textbf{\ref{eqn l2.6}}. Therefore $(r_{1},k_{1}M)=(r_{2},k_{2}M)$. So, we have proved 
the following proposition.

\begin{prop}\label{eqn p2.7}
The polar coordinate mapping $\phi$, defined above, is a bijection of $T^{+}\times 
K/M$ onto the set of regular points in $\mathbb{R}^{n}$ whose complement has measure 
zero.
\end{prop}

For a regular point $x$, if $x=\phi(r,kM)$ for $r\in T^{+}$ and $k\in K$ then we 
simply write $x=(r,kM)$ and call this the polar coordinates of $x$. It is clear from 
the definition of $\phi$, that $k_{1}\cdot (r,k_{2}M)=(r,k_{1}k_{2}M)$, $r\in T^{+}$ 
and $k_{1},k_{2}\in K$.

\begin{rem}\label{eqn r2.70}
Let $K=SO(n)$.
Consider the $K$-regular point $e_1=(1,0,0,\cdots,0)\in\mathbb{R}^{n}$. A $K$-transversal domain $T$ can be chosen to be $T=A_{e_1}=\{(x,0,0,\cdots,0)\in\mathbb{R}^n:x\in\mathbb{R}\}$, and $T^{+}=\{(r,0,0,\cdots,0)\in\mathbb{R}^n:r>0\}$, which can be identified with $(0.\infty)$. If $M$ is the stabilizer of $e_1$, via the map $kM\rightarrow k\cdot e_1$, we have the identification $K/M=K\cdot e_1=S^{n-1}$. Therefore, by the above proposition, it follows that each regular point $x\in\mathbb{R}^{n}$ can be written uniquely as $x=(r,\omega)=r\omega$, $r>0,\omega\in S^{n-1}$, which gives the usual polar coordinate system on $\mathbb{R}^{n}$.
\end{rem}

We conclude this section by relating polar actions and symmetric space 
actions, due to Dadok \cite{D}. First we define a symmetric space action.

\begin{defn}\label{eqn d2.8}
The action of a connected subgroup $G$ of $SO(n)$ with Lie algebra $\mathfrak{g}$ on 
$\mathbb{R}^{n}$ is called a symmetric space action if there is a real semisimple Lie 
algebra $\mathfrak{u}$ with Cartan decomposition 
$\mathfrak{u}=\mathfrak{k}^{\prime}+\mathfrak{p}$, a Lie algebra isomorphism $A:
\mathfrak{g}\longrightarrow\mathfrak{k}^{\prime}$, and a real vector space isomorphism 
$L:\mathbb{R}^{n}\longrightarrow \mathfrak{p}$ such that $L(X\cdot y)=[A(X),L(y)]$ for 
all $X\in \mathfrak{g}$ and $y\in \mathbb{R}^{n}$. Here $[.,.]$ denotes the Lie 
algebra bracket on $\mathfrak{g}$.
\end{defn}

\begin{rem}\label{eqn r2.9}
Let the action of $G$ be a symmetric space action. If $U$ is a connected Lie group 
with Lie algebra $\mathfrak{u}$, and $K^{\prime}$ is a connected subgroup of $U$ with 
Lie algebra $\mathfrak{k}^{\prime}$, then the action of $G$ on $\mathbb{R}^{n}$ is 
isomorphic to that of $\textup{Ad}(K^{\prime})$ on $\mathfrak{p}$, i.e if we identify 
$\mathbb{R}^{n}$ and $\mathfrak{p}$ via the map $L$, then $G$-orbits and $\textup{Ad}
(K^{\prime})$-orbits coincide.
\end{rem}

The relation between a polar and a symmetric space action is provided by the following 
proposition.

\begin{prop}\label{eqn p2.10}
\textup{(Dadok \cite{D}, Proposition 6)} Let $K$ be a connected, compact subgroup of 
$SO(n)$ whose action on $\mathbb{R}^{n}$ is polar. Then there exists a connected 
subgroup $G$ of $SO(n)$ whose action on $\mathbb{R}^{n}$ is a symmetric space action 
and whose orbits coincide with those of $K$.
\end{prop}

\section{$K$-Harmonic polynomials}
Throughout this section we assume that $K$ is a connected compact subgroup of $SO(n)$ 
whose action on $\mathbb{R}^{n}$ is polar, $T$ a $K$-transversal domain, and 
$M=K_{T}$, the centralizer of $T$. Let $S$ denote the space of polynomials on 
$\mathbb{R}^{n}$, $I\subset S$ the set of $K$- invariants in $S$ and $I_{+}$ the set 
of polynomials in $I$ without the constant term. Let $H\subset S$ denote the set of 
$K$-harmonic polynomials, that is, polynomials annihilated by the constant coefficient 
differential operators on $\mathbb{R}^{n}$ defined by elements in $I_{+}$. For more 
details about $K$- harmonic polynomials see Helgason \cite{H1}, Chapter III. The 
following result is proved there (Theorem 1.1).

\begin{thm}\label{eqn t3.1}
S=IH, that is, each polynomial $p$ on $\mathbb{R}^{n}$ has the form 
$p=\sum_{k}i_{k}h_{k}$ where $i_{k}$ is $K$-invariant and $h_{k}$ is $K$-harmonic.
\end{thm}

Since the action of $K$ is polar by Proposition \textbf{\ref{eqn p2.10}}, there is a 
connected subgroup $G$ of $SO(n)$ whose action on $\mathbb{R}^{n}$ is a symmetric 
space action and whose orbits coincide with those of $K$. Let $L$, $K^{\prime}$ and 
$\mathfrak{p}$  are as in Definition \textbf{\ref{eqn d2.8}}. Therefore, by Remark 
\textbf{\ref{eqn r2.9}}, if we identify $\mathbb{R}^{n}$ and $\mathfrak{p}$ via the 
map $L$, then $K$ and $\textup{Ad}(K^{\prime})$orbits coincide. Hence $I$ and $I_{+}$ 
are same for both actions and consequently so is $H$. So, for polar actions we have 
the following version of Kostant-Rallis Theorem (see Helgason \cite{H2}, Chapter III, 
Theorem 2.4).

\begin{thm}\label{eqn t3.2}
Each $K$-harmonic polynomial is determined by its values on a regular $K$-orbit.
\end{thm}

Now we briefly describe the class one representations of $K$ realized on the space $H$ and on the space of their restriction to a regular $K$-orbit. This is similar to the symmetric space theory (see Helgason \cite{H2}, 
page-236,237 and 298,299; \cite{H1}, page-533). For $x\in T$ regular, consider the embedding $K/M=N_{x}\subset\mathbb{R}^{n}$ via the 
map $kM\longrightarrow k\cdot x$. Then as is well known (Helgason \cite{H1}, Exercise 
A1 (iv), page-73) each $K$-finite function on $K/M$ is the restriction of a polynomial 
$p\in S$ which by Theorem \textbf{\ref{eqn t3.1}} can be taken to be harmonic. Thus, 
by Theorem \textbf{\ref{eqn t3.2}} we see that the restriction mapping 
$h\longrightarrow h\mid _{N_{x}}$ is a bijection of $H$ onto the space of $K$-finite 
functions in $\mathcal{E}(K/M)$ (the space of smooth functions on $K/M$). Let 
$\widehat{K}_{M}$ be the set of all inequivalent unitary irreducible representation of 
$K$ having $M$ fixed vector. If $\delta\in\widehat{K}_{M}$, let $H_{\delta}$ 
(respectively $\mathcal{E}_{\delta}(K/M)$) denote the space of $K$-finite functions in 
$H$ (respectively $\mathcal{E}(K/M)$) of type $\delta$. Then the restriction mapping 
maps $H_{\delta}$ onto $\mathcal{E}_{\delta}(K/M)$. Let $V_{\delta}$ be the (finite 
dimensional) Hilbert space on which $\delta$ is realized and let 
$V_{\delta}^{M}\subset V_{\delta}$ be the space of $M$-fixed vectors. Let $v_{1}, 
v_{2}, \cdots , v_{d(\delta)}$ be an orthonormal basis of $V_{\delta}$ such that 
$v_{1}, v_{2}, \cdots , v_{l(\delta)}$ span $V_{\delta}^{M}$. Then the functions
$$kM\longrightarrow \langle v_{j},\delta(k)v_{i}\rangle~ 1\leq j\leq d(\delta), ~1\leq 
i\leq l(\delta)$$ form a basis of $\mathcal{E}_{\delta}(K/M)$ (Theorem 3.5, chapter V, 
Helgason \cite{H1}), and
\bea\label{eqn20} \mathcal{E}_{\delta}
(K/M)=\bigoplus_{i=1}^{l(\delta)}\mathcal{E}_{\delta,i}(K/M),\eea
where $\mathcal{E}_{\delta,i}(K/M)$ is the space of functions
$$F_{v,i}(K/M)=\langle v,\delta(k)v_{i}\rangle,~v\in V_{\delta}.$$
The map $v\longrightarrow F_{v,i}$ is an isomorphism of $V_{\delta}$ onto 
$\mathcal{E}_{\delta,i}(K/M)$ commuting with the action of $K$.
Consequently $H_{\delta}$ decomposes into $l(\delta)$ copies of $\delta$. Thus we 
write
\bea\label{eqn21}H_{\delta}=\bigoplus_{i=1}^{l(\delta)}H_{\delta,i},\eea
where the action of $K$ on each $H_{\delta,i}$ is
equivalent to $\delta$ (by decomposing $H_{\delta}$ first into homogeneous components 
we can assume that the $H_{\delta,i}$ consists of homogeneous polynomials of degree 
say $d_{i}(\delta)$), and the vector space
$F_{\delta}=\textup{Hom}_{K}(V_{\delta},H_{\delta})$ of linear maps $\eta$ of 
$V_{\delta}$ into $H_{\delta}$ satisfying \bea\label{eqn22} \eta(\delta(k)v)=k\cdot 
(\eta(v))~~~k\in K,~v\in V_{\delta}\eea
has dimension $l(\delta)$. 

\begin{rem}\label{eqn r3.3}
Let $K=SO(n)$. Let $e_1$, $T$, $T^{+}$, $M$ be as in the Remark \textbf{\ref{eqn r2.70}}. Also we have the identification $K/M=S^{n-1}$. In this special case, note that, the space $H$ consists of all polynomials $P$ such that $\triangle P=0$, where $\triangle=\Sigma\partial^2/\partial x^2_i$ is the usual Laplacian on $\mathbb{R}^{n}$. Let $\mathcal{H}_m$ denotes the space of all $m$th degree homogeneous polynomials in $H$ and $\mathcal{S}_{m}$ denotes the space of restrictions of elements of $\mathcal{H}_m$ to $S^{n-1}$. The elements of $\mathcal{H}_m$ are called solid harmonics of degree $m$, and those of $\mathcal{S}_m$ are called spherical harmonics of degree $m$. The $K$-action on $K/M=S^{n-1}$ defines a unitary representation on $L^{2}(S^{2n-1})$. Clearly each $\mathcal{S}_m$ is a $K$-invariant subspace. Let $\delta_m$ denotes the restriction of $\delta$ to $\mathcal{S}_m$. In fact these describe all inequivalent, irreducible, unitary representations in $\widehat{K}_M$. Note that according to our general notation, $H_{\delta_m}=\mathcal{H}_m$, $\mathcal{E}_{\delta_m}(K/M)=\mathcal{S}_m$, and $l(\delta_m)=$dim$V^{M}_{\delta_m}=1$. 
Let $v^m$ be the unique (upto constant multiple) unit $M$-fixed vector in $V_{\delta_m}$. Then the one dimensional vector space $F_{\delta_m}=$Hom$(V_{\delta_m},\mathcal{H}_m)$ is spanned by the linear map $\eta_{\delta_m}:V_{\delta_m}\rightarrow\mathcal{H}_m$, where for $v\in V_{\delta_m}$, $\eta_{\delta_m}(v)$ is the unique element in $\mathcal{H}_m$ whose restriction to $S^{n-1}$ is $Y_{v}(kM):=\langle v,\delta_m(k)v^m\rangle\in\mathcal{S}_m$ i.e $\eta_{\delta_m}(v)(x)=|x|^mY_v(x/|x|)$.
\end{rem}

\section{K-Type functions in matrix form}
We assume that $K$ is a connected, compact subgroup of $SO(n)$ whose action on 
$\mathbb{R}^{n}$ is polar. We use all the notation from the previous two sections. For 
two finite dimensional vector spaces $V$ and $W$ denote the space of all linear maps 
from $V$ into $W$, by Hom$(V,W)$. For two positive integers $p$ and $q$ denote the 
space of all $p\times q$ matrices with complex entries by $\mathcal{M}_{p\times q}$. 
If $A$ is a set and $f:A\longrightarrow\mathcal{M}_{p\times q}$ a function, then we 
define $f_{ij}:A\longrightarrow \mathbb{R}^{n}$ by $f_{ij}(a)=(i,j)$th entry of 
$f(a)$, for $a\in A$. For $\delta\in \widehat{K}_{M}$ define $\mathcal{X}^{\delta}
(\mathbb{R}^{n})$  to be the set of all functions $$F:
\mathbb{R}^{n}\longrightarrow\textup{Hom}(V_{\delta}^{M},V_{\delta})$$ satisfying the 
condition
\bea\label{eqn31}F(k\cdot x)=\delta(k)F(x)~\forall~ x\in \mathbb{R}^{n},~ k\in K,\eea 
and $\mathcal{Y}^{\delta}(\mathbb{R}^{n})$ to be the set of all functions $$G:
\mathbb{R}^{n}\longrightarrow\textup{Hom}(V_{\delta},V_{\delta})$$ satisfying the 
conditions \bea\label{eqn32}G(k\cdot x)=\delta(k)G(x),~~ 
G(x)\delta(m)=G(x)~\forall~x\in\mathbb{R}^{n},~ k\in K,~m\in M.\eea Here the 
multiplications are the compositions of linear maps. Proposition \textbf{3.1} below 
says that the sets $\mathcal{X}^{\delta}(\mathbb{R}^{n})$ and $\mathcal{Y}^{\delta}
(\mathbb{R}^{n})$ can be identified. Also, define $\mathcal{E}^{\delta}
(\mathbb{R}^{n})$ to be the space of all smooth functions in $\mathcal{X}^{\delta}(\mathbb{R}^{n})$. 
Choose an orthonormal ordered basis $\textbf{b}=\{v_{1},v_{2},\cdots, v_{d(\delta)}\}$ 
for $V_{\delta}$, so that $\textbf{b}^{M}=\{v_{1},v_{2},\cdots, v_{l(\delta)}\}$ form 
an ordered basis for $V_{\delta}^{M}$. Identify $\delta$ with its matrix 
representation with respect to the basis $\textbf{b}$. Then we can identify 
$\mathcal{X}^{\delta}(\mathbb{R}^{n})$ with the space of all functions $$F:
\mathbb{R}^{n}\longrightarrow \mathcal{M}_{d(\delta)\times  l(\delta)}$$ satisfying 
(\ref{eqn31}) (but now, the multiplications are simply matrix multiplications), via 
the matrix representation with respect to bases $\textbf{b}$ for $V_{\delta}$ and 
$\textbf{b}^{M}$ for $V_{\delta}^{M}$. Similarly identify $\mathcal{Y}^{\delta}(\mathbb{R}^{n})$ and $\mathcal{E}^{\delta}(\mathbb{R}^{n})$ 
with their corresponding matrix representations with respect to bases $\textbf{b}$ and 
$\textbf{b}^{M}$. Through out this paper, we use these identifications with respect to 
the basis $\textbf{b}$ and $\textbf{b}^{M}$. Define
$$Y^{\delta}:K/M\longrightarrow\mathcal{M}_{d(\delta)\times l(\delta)}$$
by 
$$Y^{\delta}_{ij}(kM)= \delta_{ij}(k)=\langle\delta(k)v_{j},v_{i}\rangle, ~ 1 \leq i 
\leq d(\delta),~ 1 \leq j \leq l(\delta).$$ If $\stackrel{\vee}{\delta}$ denote the 
contragredient representation, choose $V_{\stackrel{\vee}{\delta}}=V_{\delta}^{*}$ (the dual 
vector space of $V_{\delta}$) with inner product $\langle,\rangle$ defined by $\langle 
v^{*},w^{*}\rangle=\langle w,v\rangle$, $v,w\in V_{\delta}$. Take the orthonormal 
ordered basis $\textbf{b}^{*}$ of $V_{\stackrel{\vee}{\delta}}$ to be the dual basis 
$\{v_{1}^{*},v_{2}^{*},\cdots,v_{d(\delta)}^{*}\}$. Then 
$\textbf{b}^{*M}=\{v_{1}^{*},v_{2}^{*},\cdots,v_{l(\delta)}^{*}\}$ will be a basis for 
$V_{\stackrel{\vee}{\delta}}^{M}$. Identify $\stackrel{\vee}{\delta}$ with its matrix representation 
with respect to the basis $\textbf{b}^{*}$. Then $\stackrel{\vee}{\delta}_{ij}
(k)=\overline{\delta_{ij}(k)}$. Therefore $\{Y^{\delta}_{ij}(kM):1\leq i\leq 
d(\delta),1\leq j\leq l(\delta) \}$ form a basis for $\mathcal{E}_{\stackrel{\vee}{\delta}}
(K/M)$. For more details about contragredient representation see Helgason \cite{H1}, 
page-393,533. Now, take an ordered basis $\textbf{e}=\{\eta_{1},\eta_{2},\cdots,
\eta_{l(\delta)}\}$ for $F_{\stackrel{\vee}{\delta}}=\textup{Hom}_{K}
(V_{\stackrel{\vee}{\delta}},H_{\stackrel{\vee}{\delta}})$. Define $$P^{\delta}:
\mathbb{R}^{n}\longrightarrow\mathcal{M}_{d(\delta)\times l(\delta)}$$ by 
$$P^{\delta}_{ij}(x)=\eta_{j}(v^{*}_{i})(x),~1\leq i\leq d(\delta),~1\leq j\leq 
l(\delta).$$ Since $\eta_{j}(\stackrel{\vee}{\delta}(k)v_{i}^{*})=k\cdot(\eta_{j}(v_{i}^{*}))$, 
using the fact that $\stackrel{\vee}{\delta}_{ij}(k)=\overline{\delta_{ij}(k)}$ and 
$\stackrel{\vee}{\delta}$ is unitary, one can show that $P^{\delta}(k\cdot 
x)=\delta(k)P^{\delta}(x).$ Hence $P^{\delta}\in \mathcal{E}^{\delta}(\mathbb{R}^{n}).
$ Define $$\Upsilon_{\delta}:\mathbb{R}^{n}\longrightarrow\mathcal{M}_{l(\delta)\times 
l(\delta)}$$ by
\bea\label{eqn 4.000}\Upsilon_{\delta}(x)=[P^{\delta}(x)]^{\star}[P^{\delta}(x)].\eea 
Here $\star$ denotes the matrix adjoint. Clearly $\Upsilon_{\delta}$ is $K$-invariant.

\begin{rem}\label{eqn r4.01}
Let $K=SO(n)$. We describe $Y^{\delta}$, $P^{\delta}$ and $\Upsilon_{\delta}$ in this special case. Let $e_1$, $T$, $T^{+}$, $M$ be as in the Remark \textbf{\ref{eqn r2.70}}, and $\mathcal{H}_m,\mathcal{S}_m,\delta_m,v^m,\eta_{\delta_m}$ be as in Remark \textbf{\ref{eqn r3.3}}. Choose an ordered orthonormal basis $\{v_1,v_2,\cdots v_{d(m)}\}$ for $V_{\delta_m}$, such that $\{v_1=v^m\}$ is the orthonormal basis for $V_{\delta_m}^{M}$. Then $\{Y^{\delta_m}_{i1}(kM)=\langle\delta_m(k)v_1,v_i\rangle:1\leq i\leq d(m)\}$ forms an orthogonal basis for $\mathcal{E}_{\stackrel{\vee}{\delta}_m}(K/M)=\overline{\mathcal{S}}_{m}=\mathcal{S}_m$, and $\Sigma_{i=1}^{d(\delta)}\big|Y^{\delta_m}_{i1}(kM)\big|^2=1$. 
Take $\{\eta_{\stackrel{\vee}{\delta}_m}\}$ as a basis for $F_{\stackrel{\vee}{\delta}_m}$ Then, by Remark \textbf{\ref{eqn r3.3}}, for $x=(r,kM)=(r,\omega)$, \beas P^{\delta_m}_{i1}(x)=\eta_{\stackrel{\vee}{\delta}_m}(v_i^{*})(x)=r^m\langle v_i^{*},\stackrel{\vee}{\delta}_m(k)v_1^{*}\rangle=r^m\langle\delta_m(k)v_1,v_i\rangle=r^{m}Y^{\delta_m}_{i1}(kM)=|x|^{m}Y^{\delta_m}_{i1}(\omega)\eeas i.e $P^{\delta_m}_{i1}$ is the unique element in $\mathcal{H}_m$ whose restriction to $S^{n-1}$ is $Y^{\delta_m}_{ij}$.

From the above discussion we can prove the following : 
Take $P_m^i\in\mathcal{H}_m$, and $Y_m^i\in\mathcal{S}_m$ to be their restrictions to $S^{n-1}$ so that $\{Y_{m}^{i}:i=1,2,\cdots,d(m)\}$ forms an orthonormal basis for $\mathcal{S}_m$.  Then it is possible to choose orthonormal ordered bases  $\textbf{b}=\{v_1,v_2,\cdots,v_{d(m)}\}$ for $V_{\delta_m}$ and $\textbf{b}^M=\{{v_1\}}$ for $V_{\delta_m}^{M}$, so that, with respect to these bases, $Y^{\delta_m}:S^{n-1}\rightarrow\mathcal{M}_{d(m)\times 1}$ is given by \bea\label{eqn 4.001}Y^{\delta_m}(\omega)=\sqrt{\frac{|S^{n-1}|}{d(m)}}\bigg[Y^1_m(\omega),Y^2_m(\omega),\cdots, Y^{d(m)}_m(\omega)\bigg]^t,\omega\in S^{n-1}.\eea We can choose a basis $\textbf{e}$ for $F_{\check{\delta}_m}$ so that, $P^{\delta_m}:\mathbb{R}^{n}\rightarrow\mathcal{M}_{d(m)\times 1}$ is given by   \bea\label{eqn 4.002}P^{\delta_m}(x)=\sqrt{\frac{|S^{n-1}|}{d(m)}}\bigg[P^1_m(x),P^2_m(x),\cdots, P^{d(m)}_m(x)\bigg]^t,x\in \mathbb{R}^n.\eea In particular, \bea\label{eqn 4.003}P^{\delta_m}(x)=|x|^mY^{\delta_m}(x/|x|).\eea Also, we have $$\Upsilon_{\delta_m}(x)=\frac{|S^{n-1}|}{d(m)}\sum_{i=1}^{d(m)}\overline{P^i_m(x)}P^i_m(x)=\frac{|S^{n-1}|}{d(m)}|x|^{2m}\sum_{i=1}^{d(m)}\big|Y^i_m(\omega)\big|^2=|x|^{2m}.$$ 


\end{rem}

\begin{prop}\label{eqn p4.1}
Each $G\in \mathcal{Y}^{\delta}(\mathbb{R}^{n})$ is determined by its restriction on 
$V_{\delta}^{M}$.
\end{prop}

\begin{proof}
If $G\in\mathcal{Y}^{\delta}(\mathbb{R}^{n})$, then $G$ is identified with its 
$\big(d(\delta)\times d(\delta)\big)$ matrix with respect to the fixed basis 
$\textbf{b}$. Hence it is enough to show that all the entries in last $\big(d(\delta)-
l(\delta)\big)$ columns of $G$ are zero. Since $G(x)\delta(m)=G(x)$ for all $m\in M$, 
equating the matrix entries on both sides we get, for $1\leq i,j\leq d(\delta)$,
\bea\label{eqn33}G_{ij}(x)=\sum_{p=1}^{d(\delta)}G_{ip}(x)\delta_{pj}(m),~\forall m\in 
M.\eea Since $v_{j}\in (V_{\delta}^{M})^{\bot}$ for $j\geq l(\delta)$, 
$\int_{M}\delta(m)v_jdm=0$ if $j\geq l(\delta)$. So,
\bea\label{eqn34} \int_{M}\delta_{pj}(m)dm=\int_{M}\langle\delta(m)v_{j},v_{p}\rangle 
dm=0,~ 1\leq p\leq d(\delta),~j\geq l(\delta).\eea Therefore for $j\geq l(\delta)$, 
integrating both side of (\ref{eqn33}) over $M$ we get  the desired result.
\end{proof}

\begin{lem}\label{eqn l4.2}
Suppose F is in $\mathcal{E}^{\delta}(\mathbb{R}^{n})$. Then there is a unique 
function $G_{0}:T^{+}\longrightarrow\mathcal{M}_{l(\delta)\times l(\delta)}$ such that 
for all regular points $x=(r,kM)$,
$$F(x)=Y^{\delta}(kM)G_{0}(r).$$
\end{lem}

\begin{proof}
First note that the uniqueness follows from the fact that $Y^{\delta}(kM)$ has a left 
inverse namely $[Y^{\delta}(kM)]^{\star}$. Since $F(\sigma\cdot x)=\delta(\sigma)F(x)$ 
for all $\sigma\in K$, we can write (for $x=(r,kM)$ regular)
\beas F(x)&=&\int_{K}\delta(\sigma)^{-1}F(\sigma\cdot 
x)d\sigma\\&=&\int_{K}\delta(\sigma)^{-1}F\left(\sigma\cdot(r,kM)\right)d\sigma\\&=&\int_{K}\delta(\sigma)^{-1}F(r,
\sigma kM)d\sigma\\&=& \int_{K}\delta(\sigma k^{-1})^{-1}F(r,\sigma M)d\sigma\\&=& 
\delta(k)\int_{K}\delta(\sigma)^{-1}F(r,\sigma M )d\sigma \\&=& 
\delta(k)\int_{K}\int_{M}\delta(\sigma m)^{-1}F(r,\sigma M )d\sigma dm 
=\delta(k)G_{0}^{\prime}(r),\eeas where $$G_{0}^{\prime}
(r)=\int_{K}\int_{M}\delta(\sigma m)^{-1}F(r,\sigma M )d\sigma dm.$$ 
Now, $$\delta_{ij}(\sigma m)=\sum_{p=1}^{d(\delta)}\delta_{ip}(\sigma)\delta_{pj}(m).
$$ Integrating both sides over $M$ and using (\ref{eqn34}) we get (for each $\sigma\in 
K$), $$\int_{M}\delta_{ij}(\sigma m)dm=0,~ 1\leq i\leq d(\delta),~j\geq l(\delta).$$ 
Since $\delta(\sigma m)^{-1}=\overline{\delta(\sigma m)}^{t}$, all the entries in last 
$\big(d(\delta)-l(\delta)\big)$ rows of the matrix $\int_{M}\delta(\sigma m)^{-1}dm$ 
are zero for all $\sigma\in K$, and consequently so is for the $d(\delta)\times 
l(\delta)$ matrix
$$G_{0}^{\prime}(r)=\int_{K}\int_{M}\delta(\sigma m)^{-1}F(r,\sigma M )d\sigma dm$$ ( 
note that $F$ is a $\big(d(\delta)\times l(\delta)\big)$ matrix). Therefore, if we 
define $G_{0}:T^{+}\longrightarrow\mathcal{M}_{l(\delta)\times l(\delta)}$ by 
$$(G_{0})_{ij}(r)=(G^{\prime}_{0})_{ij}(r),~1\leq i,j\leq l(\delta),$$ then 
$$\delta(k)G_{0}^{\prime}(r)=Y^{\delta}(kM)G_{0}(r),$$ since first $l(\delta)$ columns 
in the matrix $\delta(k)$ are precisely the columns in $Y^{\delta}(kM).$ Hence the 
proof.
\end{proof}

\begin{cor}\label{eqn c4.3}
Let $F\in \mathcal{E}^{\delta}(\mathbb{R}^{n})$. Then the $j$th column of $F$ is 
determined by $F_{1j}$. In particular $F$ is determined by its first row.
\end{cor}

\begin{proof}
Let $F\in \mathcal{E}^{\delta}(\mathbb{R}^{n})$ be such that all the entries in first 
row are identically zero. We have to show that $F\equiv 0$. Let $G_{0}$ be as in the 
previous lemma. We have
$$\sum_{p=1}^{l(\delta)}Y^{\delta}_{1p}(kM)(G_{0})_{pj}(r)=0~\textup{for all regular 
points}~(r,kM).$$ Since $Y^{\delta}_{1p}$s are linearly independent, for each $r\in 
T^{+}$, $(G_{0})_{pj}(r)=0,~p=1,2,\cdots l(\delta)$; and consequently the $j$th column 
of $F$ is zero on the set of regular points. The set of regular points being dense in 
$\mathbb{R}^{n}$, we are done.
\end{proof}

\begin{rem}\label{eqn r4.4}
Using the above arguments, one can show the following : For $F\in \mathcal{E}^{\delta}
(\mathbb{R}^{n})$, let $V_{F}$ denote the finite dimensional vector space spanned by 
$F_{ij}$'s, and $V^{i}_{F}$ denote the space spanned by the entries of $i$th row in 
$F$. Let $m(\delta)$ be the number of linearly independent columns in $F$. Also assume 
that first $m(\delta)$ columns are linearly independent. Then $\{F_{ij}:j=1,2,
\cdots,m(\delta)\}$ form a basis for $V^{i}_{F}$; and $\{F_{ij}:i=1,2,\cdots 
d(\delta);j=1,2,\cdots,m(\delta)\}$ form a basis for $V_{F}$. In particular, 
dim$V_{F}=m(\delta)l(\delta)$.
\end{rem}

\begin{lem}\label{eqn l4.5}
There is a unique function 
$J_{\delta}:T^{+}\longrightarrow\mathcal{M}_{l(\delta)\times l(\delta)}$ such that for 
all regular point $x=(r,kM)$ in $\mathbb{R}^{n}$,
\bea\label{eqn35}P^{\delta}(x)=Y^{\delta}(kM)J_{\delta}(r).\eea
Also for each $r\in T^{+}$, $J^{\delta}(r)$ is invertible, and consequently for all 
regular point $x=(r,kM)$
\bea\label{eqn36}Y^{\delta}(kM)=P^{\delta}(x)[J_{\delta}(r)]^{-1}.\eea
\end{lem}

\begin{proof}
First part follows from Lemma \textbf{\ref{eqn l4.2}}, since $P^{\delta}\in 
\mathcal{E}^{\delta}(\mathbb{R}^{n})$. Now, if possible, let for some $r_{0}$ in 
$T^{+}$, $J_{\delta}(r_{0})$ be not invertible i.e $\textup{det}(J_{\delta}(r_{0}))=0$ 
(where det stands for the determinant). This implies that the columns of $J_{\delta}
(r_{0})$ namely $\big[J_{1j}(r_{0}),J_{2j}(r_{0}),\cdots J_{l(\delta)j}(r_{0})\big]^t$, $1\leq 
j\leq l(\delta)$ are linearly dependent as vectors in $\mathbb{C}^{l(\delta)}$. 
Equating the entries of first row in (\ref{eqn35}) for $x=(r_{0},kM)$ we get ($1\leq 
j\leq l(\delta)$),
$$P^{\delta}_{1j}(x)=Y^{\delta}_{11}(kM)J_{1j}(r_{0})+ Y^{\delta}_{12}(kM)J_{2j}
(r_{0})+\cdots +Y^{\delta}_{1l(\delta)}(kM)J_{l(\delta)j}(r_{0}).$$ Therefore 
$P^{\delta}_{1j}$s are linearly dependent when restricted to the orbit through $r_{0}$ 
and hence by Kostant-Rallis Theorem (Theorem \textbf{\ref{eqn t3.2}}) 
$P^{\delta}_{1j}$s are linearly dependent which is a contradiction. Therefore 
$J^{\delta}(r)$ is invertible for all $r\in T^{+}$.
\end{proof}

\begin{rem}\label{eqn r4.6}

\indent {\rm{\textbf{(i)}}} Let $x=(r,kM)$ be a regular point. Since $Y^{\delta}(kM)$ 
has a left inverse, and $J_{\delta}(r)$ is invertible $P^{\delta}(x)$ also has a left 
inverse.

\indent {\rm{\textbf{(ii)}}} The function $J_{\delta}$ is related to 
$\Upsilon_{\delta}$ (see (\ref{eqn 4.000})) by \bea\label{eqn37}[J_{\delta}(r)]^{\star}[J_{\delta}
(r)]=\Upsilon_{\delta}(r),~\forall~r\in T^{+}.\eea

\indent {\rm{\textbf{(iii)}}} Let $K=SO(n)$. 
Let $Y^{\delta_m}$, $P^{\delta_m}$ be as in (\ref{eqn 4.001}) and (\ref{eqn 4.002}). Then we have seen that $$P^{\delta_m}(x)=r^mY^{\delta_m}(\omega),~x=r\omega,~r>0,\omega\in S^{n-1}.$$ Therefore $J_{\delta_m}:T^{+}=(0,\infty)\rightarrow\mathcal{M}_{1\times1}$ is given by $J_{\delta_m}(r)=r^m$. 
\end{rem}

The next proposition follows by using (\ref{eqn36}) in Lemma \textbf{\ref{eqn l4.2}}, 
and by \textbf{(i)} in the previous remark.

\begin{prop}\label{eqn p4.7}
Suppose F is in $\mathcal{E}^{\delta}(\mathbb{R}^{n})$. Then there is a unique (on the 
set of regular points) $K$-invariant function $G:
\mathbb{R}^{n}\longrightarrow\mathcal{M}_{l(\delta)\times l(\delta)}$ such that for 
all regular points $x$ (hence for almost every $x$),
$$F(x)=P^{\delta}(x)G(x).$$
\end{prop}


Throughout this paper we use the following convention : when we say that a matrix-valued function is a polynomial we mean that each entry of the function is a polynomial. 

\begin{cor}\label{eqn c4.8}
Let $F\in \mathcal{E}^{\delta}(\mathbb{R}^{n})$ be a polynomial.
Then there is a unique $K$-invariant polynomial $G:
\mathbb{R}^{n}\longrightarrow\mathcal{M}_{l(\delta)\times l(\delta)}$ such that for 
all $x\in\mathbb{R}^{n}$,
$$F(x)=P^{\delta}(x)G(x).$$
\end{cor}

\begin{proof}
Since the set of regular points is dense in $\mathbb{R}^{n}$, uniqueness follows from 
Remark \textbf{\ref{eqn r4.6}} \textbf{(i)}. Now let $G$ be as in the previous 
proposition. It is enough to show that each entry of $G$ is equal to a polynomial on 
the set of regular points. Consider $F_{11}$. For all regular points $x$ we have 
\bea\label{eqn38}F_{11}(x)=\sum_{p=1}^{l(\delta)}P^{\delta}_{1p}(x)G_{p1}(x).\eea
By Theorem \textbf{\ref{eqn t3.1}} we have $$S_{\stackrel{\vee}{\delta}}=IH_{\stackrel{\vee}{\delta}},$$
where $S_{\stackrel{\vee}{\delta}}\subset S$ denotes the space of all polynomials of type $\stackrel{\vee}{\delta}$. 
Clearly $F_{11}\in S_{\stackrel{\vee}{\delta}}$. Therefore there exists $K$-invariant polynomials 
$I_{ij}$ such that for all $x\in\mathbb{R}^{n}$ \bea\label{eqn39}F_{11}
(x)=\sum_{i=1}^{d(\delta)}\sum_{j=1}^{l(\delta)}I_{ij}(x)P^{\delta}_{ij}(x).\eea
Since $P^{\delta}_{ij}$s are linearly independent, by Kostant-Rallis Theorem (Theorem 
\textbf{\ref{eqn t3.2}}) so are their restrictions to any regular orbit. Comparing  
equations (\ref{eqn38}) and (\ref{eqn39}), restricted to a orbit passing through a 
regular point $x$, we get $G_{p1}(x)=I_{p1}(x)$ for all $p=1,2,\cdots l(\delta)$. 
Similar proof works for other entries of $G$. 
\end{proof}

\section{Heisenberg group : representations, Weyl transform, spherical functions}
The Heisenberg group $\mathbb{H}^{n}$ is the Lie group with underlying manifold 
$\mathbb{C}^{n}\times\mathbb{R}$ and group operation $$(z,t)
(z',t')=(z+z',t+t'+\textup{Im} (z.\bar{z'})).$$ For the following see Geller \cite{G}. 
For real non-zero $\lambda$, let $$\mathcal{H}^{\lambda}=\{u ~\textup{holomorphic on} 
~\mathbb{C}^{n}: \int_{\mathbb{C}^{n}}|u(w)|^{2}d\widetilde{w}^{\lambda}=||u||
^{2}<\infty\},$$ where the measure $d\widetilde{w}^{\lambda}$ is given by
$$d\widetilde{w}^{\lambda}=(2|\lambda|/\pi)^{n}e^{-2|\lambda||w|^{2}}dwd\bar{w}.$$
The space $\mathcal{H}^{\lambda}$ is a Hilbert space and an orthonormal basis is given 
by $\{u_{\nu}^{\lambda}:\nu\in \mathbb{Z}_{+}^{n}\}$, where $\mathbb{Z}_{+}^{n}$ is 
the set of non-negative $n$-tuple, and $$u_{\nu}^{\lambda}(w)=[(2|
\lambda|)^{1/2}w]^{\nu}
(\nu !)^{1/2}.$$
(Here $\nu 
!=\Pi_{j=1}^{n}\nu_{j}!
$ and 
$w^{\nu}=\Pi_{j=1}^{n}w_{j}^{\nu_{j}}.)$ 
Let $\mathcal{O}
(\mathcal{H}^{\lambda})$ 
denote the set of all linear operators in $\mathcal{H}^{\lambda}$ whose domain of 
definition contains $\mathcal{P}(\mathbb{C}^{n})$, the space of holomorphic polynomial 
on $\mathbb{C}^{n}$. For $\lambda>0$, define $\overline{W}^{\lambda}_{j},$ 
$W^{\lambda}_{j}\in\mathcal{O}(\mathcal{H}^{\lambda})$ as follows: if $P\in\mathcal{P}
(\mathbb{C}^{n})$, $$\overline{W}^{\lambda}_{j}P(w)=2|\lambda|
w_{j}P(w)~\textup{and}~W^{\lambda}_{j}P(w)=\frac{\partial P}{\partial w_{j}}(w),$$ 
while if $\lambda<0$ the situation is reversed (In Geller \cite{G}, the notation 
$W^{+}_{j\lambda}$, $W_{j\lambda}$ are used for $\overline{W}^{\lambda}_{j}$, 
$W^{\lambda}_{j}$ respectively). We have the commutation relations \bea\label{eqn41}
[\overline{W}^{\lambda}_{j},W^{\lambda}_{k}]=-2\delta_{jk}\lambda I,~ 
[W^{\lambda}_{j},W^{\lambda}_{k}]=0,~[\overline{W}^{\lambda}_{j},
\overline{W}^{\lambda}_{k}]=0,\eea where $I$ denote the identity operator. Let 
$\overline{W}^{\lambda}=(\overline{W}^{\lambda}_{1},\overline{W}^{\lambda}_{1},\cdots 
\overline{W}^{\lambda}_{n})$,\\ $W^{\lambda}=(W^{\lambda}_{1},W^{\lambda}_{1},\cdots 
W^{\lambda}_{n})$; and for $z\in\mathbb{C}^{n}$, let $z\cdot \overline{W}^{\lambda}$, 
$z\cdot W^{\lambda}$ denote the operators 
$z_{1}\overline{W}^{\lambda}_{1}+z_{2}\overline{W}^{\lambda}_{2}+\cdots 
z_{n}\overline{W}^{\lambda}_{n}$ and $z_{1}W^{\lambda}_{1}+z_{2}W^{\lambda}_{2}+\cdots 
z_{n}W^{\lambda}_{n}$ respectively. Then $i(-z\cdot \overline{W}^{\lambda}+\bar{z}\cdot 
W^{\lambda})$  being self-adjoint, $$V_{z}^{\lambda}=\textup{exp}(-z\cdot 
\overline{W}^{\lambda}+\bar{z}\cdot W^{\lambda})$$ extends to a unitary operator on 
$\mathcal{H}^{\lambda}$ which satisfy 
\bea\label{eqn42}V^{\lambda}_{z}V^{\lambda}_{w}=\textup{exp}(2i\lambda ~\textup{Im}
(z\cdot\bar{w}))V^{\lambda}_{z+w},\eea and has an explicit formulae given as follows: 
if $u\in\mathcal{H}_{\lambda}$,
\beas(V^{\lambda}_{z}u)(w)&=& u(w+\bar{z})\textup{exp}[-2\lambda(w\cdot z+|z|
^{2}/2)]~\textup{for}~\lambda >0 \\&=& u(w-z)\textup{exp}[2\lambda(-w\cdot\bar{z}+|z|
^{2}/2)]~\textup{for}~\lambda<0 .\eeas
In view of (\ref{eqn42}) we have a representation $\Pi^{\lambda}$ of $\mathbb{H}^{n}$ 
on $\mathcal{H}^{\lambda}$, given by $\Pi^{\lambda}(z,t)=e^{i\lambda 
t}V^{\lambda}_{z}.$ Explicitly $\Pi^{\lambda}$ is given as follows: if 
$u\in\mathcal{H}^{\lambda}$, \beas(\Pi^{\lambda}(z,t)u)(w)&=& u(w+\bar{z})\textup{exp}
[-2\lambda(w\cdot z+|z|^{2}/2)]e^{i\lambda t}~\textup{for}~\lambda >0 \\&=& u(w-
z)\textup{exp}[2\lambda(-w\cdot\bar{z}+|z|^{2}/2)]e^{i\lambda 
t}~\textup{for}~\lambda<0 .\eeas
In fact, these are all the unitary irreducible representations of $\mathbb{H}^{n}$ 
which are non-trivial on the center. Note that $\Pi^{\lambda}(z,0)=V^{\lambda}_{z}$. 
We will write $\Pi^{\lambda}(z)$ instead of $\Pi^{\lambda}(z,0)$.
Since $V^{\lambda}_{z}$ is unitary, we can define a map $\mathcal{G}^{\lambda}:
\mathscr{S}(\mathbb{C}^{n})\longrightarrow\mathcal{O}(\mathcal{H}^{\lambda})$ by
$$\mathcal{G}^{\lambda}f=\int_{\mathbb{C}^{n}}f(z)V_{z}^{\lambda}dzd\bar{z}=\int_{\mathbb{C}^{n}} 
f(z)\Pi^{\lambda}(z)dzd\bar{z}.$$
The operator $\mathcal{G}^{\lambda}f$ is called the Weyl transform of $f$. Let $\mathcal{S}_{2}
(\mathcal{H}^{\lambda})$ stand for the Hilbert space of Hilbert-Schmidt operators 
on $\mathcal{H}^{\lambda}$ with the inner product $\langle T,S\rangle=\textup{tr}
(TS^{*})$. Let $||.||_{\textup{HS}}$ denote the Hilbert-Schmidt norm.
Now we state the Plancherel theorem for Weyl transform. 

\begin{thm}\label{eqn t5.1}
\textup{(Geller \cite{G}, Theorem 1.2)} If $f\in\mathscr{S}(\mathbb{C}^{n})$, then 
$\mathcal{G}^{\lambda}f\in\mathcal{S}_{2}(\mathcal{H}^{\lambda})$ and
$$||f||_{2}^{2}=\pi^{-n}(2|\lambda|)^{n}||\mathcal{G}^{\lambda}f||_{\textup{HS}}^{2}.
$$
The map $\mathcal{G}^{\lambda}$ may then be extended as a constant multiple of a unitary map 
from $L^{2}(\mathbb{C}^{n})$ onto $\mathcal{S}_{2}(\mathcal{H}^{\lambda})$. A polarization of the above formula gives $$\langle f,g\rangle=\pi^{-n}(2|\lambda|)^n\big\langle \mathcal{G}^{\lambda}f,\mathcal{G}^{\lambda}g\big\rangle,$$ where $f,g\in L^2(\C)$.
\end{thm}

For $f,g\in L^{2}(\mathbb{C}^{n})$, define the twisted convolution 
$$f\times^{\lambda}g(z)=\int_{\mathbb{C}^{n}}f(z-w)g(w)e^{2i\lambda~\textup{Im}
(z\cdot\bar{w})}dw.$$ Then it is well-known that, $f\times^{\lambda}g\in L^{2}
(\mathbb{C}^{n})$ and $$\mathcal{G}^{\lambda}
(f\times^{\lambda}g)=\mathcal{G}^{\lambda}(f)\mathcal{G}^{\lambda}(g).$$
Next, we extend this definition to a suitable subset of $\mathscr{S}^{\prime}
(\mathbb{C}^{n})$, the space of tempered distributions on $\mathbb{C}^{n}$ (see Geller 
\cite{G}, page 624-625). We say that $T\in\mathscr{S}^{\prime}(\mathbb{C}^{n})$ is 
Weyl transformable if there exist $R\in\mathcal{O}(\mathcal{H^{\lambda}})$ such that 
$$T(f)=\pi^{-n}(2|\lambda|)^{n}\sum_{\nu\in\mathbb{Z}_{+}^{n}}
\left(Ru_{\nu}^{\lambda},
(\mathcal{G^{\lambda}}f)u_{\nu}^{\lambda}\right)~\forall~f\in\mathscr{S}
(\mathbb{C}^{n}),$$ where the series converges absolutely. It is shown in \cite{G} 
that if such an $R$ exists then it is unique. In this case we call $R$ to be the Weyl 
transform of $T$ and write $\mathcal{G}^{\lambda}(T)=R$. It is clear from the 
polarization of Plancherel Theorem (Theorem \textbf{\ref{eqn t5.1}}) that this 
definition agrees with the previous definition of Weyl transform if $T$ is given by an 
$L^{2}$-function. In the course of proving the uniqueness of $R$, Geller proved that, 
if we fix a $\gamma\in\mathbb{Z}^{n}_{+}$, then for each $\alpha,
\beta\in\mathbb{Z}^{n}_{+}$ there exist $f_{\alpha\beta}\in\mathscr{S}(\C)$ such that 
$\mathcal{G}^{\lambda}
(f_{\alpha\beta})u^{\lambda}_{\gamma}=\delta_{\alpha\gamma}u^{\lambda}_{\beta}$.  
Taking $\beta=\alpha$, in particular we have the following: Fix 
$\gamma\in\mathbb{Z}^{n}_{+}$. Then for each $\alpha\in\mathbb{Z}^{n}_{+}$, there 
exists $f_{\alpha}\in\mathscr{S}(\C)$ such that $\mathcal{G}^{\lambda}
(f_{\alpha})u^{\lambda}_{\gamma}=\delta_{\alpha\gamma}u^{\lambda}_{\alpha}$. From this 
fact, the next proposition follows easily. 

\begin{prop}\label{eqn p5.2}
Let $\{T_{j}\}$ be a sequence of tempered distributions which converge to a tempered 
distribution $T$ in the topology of $\mathscr{S}^{\prime}(\C)$, i.e $T_{j}(f)\ra T(f)$ 
for all $f\in\mathscr{S}(\C)$. Assume that all $T_{j}$'s  and $T$ are Weyl 
transformable. Then for any $u,v\in\mathcal{P}(\mathbb{C}^{n})$, 
$\langle\mathcal{G}^{\lambda}(T_{j})u,v\rangle \ra \langle\mathcal{G}^{\lambda}
(T)u,v\rangle$.
\end{prop}

Define $\mathcal{F}^{\prime}:\mathscr{S}(\mathbb{C}^{n})\longrightarrow\mathscr{S}
(\mathbb{C}^{n})$ by $$(\mathcal{F}^{\prime}f)(\zeta)=\int_{\mathbb{C}^{n}} \exp(-
z\cdot \bar{\zeta}+\bar{z}\cdot\zeta)f(z)dzd\bar{z}.$$
This is a modification of the usual Euclidean Fourier transform $\mathcal{F}$, the 
relation being that $(\mathcal{F}^{\prime}f)(\zeta)=(\mathcal{F}f)(-2i\zeta)$. So we 
can extend $\mathcal{F}^{\prime}$ as a continuous, linear, one-to-one mapping of 
$\mathscr{S}^{\prime}(\mathbb{C}^{n})$ onto $\mathscr{S}^{\prime}(\mathbb{C}^{n})$. 
Let $T\in\mathscr{S}^{\prime}(\mathbb{C}^{n})$ be such that $\mathcal{F}^{\prime-1}T$ 
is Weyl transformable. Then we define the Weyl correspondence $\mathcal{W}^{\lambda}$ 
of $T$ by $$\mathcal{W}^{\lambda}(T)=\mathcal{G}^{\lambda}(\mathcal{F}^{\prime-1}T).$$ 
On $\mathbb{H}^{n}$, the differential operators $$T=\frac{\partial}{\partial t} ,~ 
Z_{j}=\frac{\partial}{\partial\bar{z}_{j}}-iz_{j}\frac{\partial}{\partial t},~ 
\overline{Z}_{j}=\frac{\partial}{\partial z_{j}}+i\bar{z}_{j}\frac{\partial}{\partial t}$$ 
are the left invariant vector fields corresponding to the one parameter family of 
subgroups $\Gamma_{0}=\{(0,s):s\in\mathbb{R}\}$, 
$\Gamma_{j}=\{se_{j},0:s\in\mathbb{R}\}$ and 
$\overline{\Gamma}_{j}=\{s\bar{e}_{j},0:s\in\mathbb{R}\}$ respectively, where 
$\{e_{1},e_{2},\cdots ,e_{n}\}$ be the usual basis for $\mathbb{C}^{n}$. In \cite{G}, 
page-651, the notation differ slightly. Geller uses $\overline{Z}_{j}$ for our operator 
$Z_{j}$ (and $Z_{j}$ for $\overline{Z}_{j}$). These  form a basis for $\mathcal{L}
(\mathfrak{h}_{n})$, the set of all left invariant differential operators on 
$\mathbb{H}^{n}$. Here $\mathfrak{h}_{n}$ is the Lie algebra of $\mathbb{H}^{n}$. For 
each $D\in\mathcal{L}(\mathfrak{h}_{n})$, let $D^{\lambda}$ denote the operator on 
$\mathbb{C}^{n}$ obtained by replacing each copy of $\partial/\partial t$ in $D$ by $-
i\lambda$. Define $$\mathcal{L}^{\lambda}
(\mathbb{C}^{n})=\{D^{\lambda}:D\in\mathcal{L}(\mathfrak{h}_{n})\},
~\textup{and}~\mathcal{R}^{\lambda}(\mathbb{C}^{n})
=\{D^{-\lambda}:D\in\mathcal{L}(\mathfrak{h}_{n})\}.$$ Then 
$$L^{\lambda}_{j}=\frac{\partial}{\partial\bar{z}_{j}}-\lambda z_{j},~~ 
\overline{L}^{\lambda}_{j}=\frac{\partial}{\partial z_{j}}+\lambda \bar{z}_{j}$$ form a 
basis for $\mathcal{L}^{\lambda}(\mathbb{C}^{n})$, and 
$$R^{\lambda}_{j}=\frac{\partial}{\partial\bar{z}_{j}}+\lambda z_{j},~~ 
\overline{R}^{\lambda}_{j}=\frac{\partial}{\partial z_{j}}-\lambda \bar{z}_{j}$$ form a 
basis for $\mathcal{R}^{\lambda}(\mathbb{C}^{n})$. In \cite{G}, page-619, these are 
denoted by $\widetilde{\mathscr{Z}}_{j\lambda}$, $\mathscr{Z}_{j\lambda}$, 
$\widetilde{\mathscr{Z}}_{j\lambda}^{R}$, $\mathscr{Z}_{j\lambda}^{R}$ respectively. Note 
that the action of $Z_{j}$ and $\overline{Z}_{j}$ on a function of the form $e^{-i\lambda 
t}f(z)$ are given by $$Z_{j}(e^{-i\lambda t}f)=e^{-i\lambda t}L_{j}^{\lambda}
(f),~\overline{Z}_{j}(e^{-i\lambda t}f)=e^{-i\lambda t}\overline{L}_{j}^{\lambda}(f).$$
We also have the commutation relations \bea\label{eqn43}
[\overline{L}^{\lambda}_{j},L^{\lambda}_{k}]=-2\delta_{jk}\lambda 
I,~[L^{\lambda}_{j},L^{\lambda}_{k}]=0,~[\overline{L}^{\lambda}_{j},
\overline{L}^{\lambda}_{k}]=0.\eea \bea\label{eqn43.1}
[\overline{R}^{\lambda}_{j},R^{\lambda}_{k}]=2\delta_{jk}\lambda 
I,~[R^{\lambda}_{j},R^{\lambda}_{k}]=0,~[\overline{R}^{\lambda}_{j},
\overline{R}^{\lambda}_{k}]=0.\eea  The following proposition tells how the operators 
$L^{\lambda}_{j}$, $\overline{L}^{\lambda}_{j}$, $R^{\lambda}_{j}$, $\overline{R}^{\lambda}_{j}$ 
behave under $\mathcal{G}^{\lambda}$. The proof can be found in \cite{G}, page 
624-625.

\begin{prop}\label{eqn p5.3}
If $T\in\mathscr{S}^{\prime}(\mathbb{C}^{n})$ is Weyl transformable, then so are 
$L^{\lambda}_{j}T$, $\overline{L}^{\lambda}_{j}T$, $R^{\lambda}_{j}T$, 
$\overline{R}^{\lambda}_{j}T$. Also $$\mathcal{G}^{\lambda}
(L^{\lambda}_{j}T)=-\mathcal{G}^{\lambda}(T)W^{\lambda}_{j},~\mathcal{G}^{\lambda}
(\overline{L}^{\lambda}_{j}T)=\mathcal{G}^{\lambda}(T)\overline{W}^{\lambda}_{j},$$ 
$$\mathcal{G^{\lambda}}(R^{\lambda}_{j}T)=-W^{\lambda}_{j}\mathcal{G}^{\lambda}
(T),~\mathcal{G}^{\lambda}
(\overline{R}^{\lambda}_{j}T)=\overline{W}^{\lambda}_{j}\mathcal{G}^{\lambda}(T).$$
\end{prop}

Let $\delta_{0}$ denote the Dirac delta distribution at origin. Since 
$\mathcal{G}^{\lambda}(\delta_{0})$ is the identity operator, from the above 
proposition we get the following corollary. We write $\mathcal{G}^{\lambda}(D)$ for 
$\mathcal{G}^{\lambda}(D\delta_{0})$, if $D\in \mathcal{L}^{\lambda}(\mathbb{C}^{n})$ 
or $\mathcal{R}^{\lambda}(\mathbb{C}^{n})$.

\begin{cor}\label{eqn c5.4}
Let $T$ be a Weyl transformable tempered distribution. Then $\mathcal{G}^{\lambda}
(DT)=\mathcal{G}^{\lambda}(T)\mathcal{G}^{\lambda}(D)$ if $D\in \mathcal{L}^{\lambda}
(\mathbb{C}^{n})$; and $\mathcal{G}^{\lambda}(DT)=\mathcal{G}^{\lambda}
(D)\mathcal{G}^{\lambda}(T)$ if $D\in \mathcal{R}^{\lambda}(\mathbb{C}^{n})$.
\end{cor}

Let $\mathcal{P}(\mathbb{C}_{\mathbb{R}}^{n})$ denote the space of all polynomials on 
the underlying real vector space $\mathbb{C}^{n}_{\mathbb{R}}$ of $\mathbb{C}^{n}$. 
Clearly $\mathbb{C}^{n}_{\mathbb{R}}$ can be identified with $\mathbb{R}^{2n}$. In 
other words the elements of $\mathcal{P}(\mathbb{C}^{n}_{\mathbb{R}})$ are polynomials 
in $z$ and $\bar{z}$ with complex coefficients. From now on we use the following 
convention : when we write ``\textbf{polynomial}", we mean a polynomial in $z$ 
and $\bar{z}$, i.e. an element in $\mathcal{P}(\mathbb{C}_{\mathbb{R}}^{n})$, and 
elements of $\mathcal{P}(\mathbb{C}^{n})$ are called ``\textbf{holomorphic 
polynomials}" i.e. polynomials in $z$ only. For a monomial 
$p(\zeta)=\zeta^{\rho}\bar{\zeta}^{\gamma}$ ($\rho,\gamma$ multi-indices), we set 
$$\theta_{1}^{\lambda}(p)=(\overline{R}^{\lambda})^{\gamma}(-R^{\lambda})^{\rho},~ 
\theta_{2}^{\lambda}(p)=(-R^{\lambda})^{\rho}(\overline{R}^{^{\lambda}})^{\gamma}$$ and 
$$\tau_{1}^{\lambda}(p)=(\overline{W}^{\lambda})^{\gamma}(W^{\lambda})^{\rho},~ 
\tau_{2}^{\lambda}(p)=(W^{\lambda})^{\rho}(\overline{W}^{^{\lambda}})^{\gamma}.$$ In the above $(\overline{R}^{\lambda})^{\gamma}=(\overline{R}_1^{\lambda})^{\gamma_1}\cdots(\overline{R}_n^{\lambda})^{\gamma_n}$ (order does not matter because of commutation relations (\ref{eqn43.1})), where $\gamma=(\gamma_1,\cdots,\gamma_n)$. The other expressions are similarly defined. Define 
$$\theta^{\lambda}(p)=\frac{1}{2}(\theta^{\lambda}_{1}(p)+\theta^{\lambda}_{2}(p)),~ \tau^{\lambda}(p)=\frac{1}{2}(\tau^{\lambda}_{1}(p)+\tau^{\lambda}_{2}(p)).$$ 
We extend them to all polynomials by linearity. Note that by Proposition 
\textbf{\ref{eqn p5.3}},
\bea\label{eqn44}\mathcal{G}^{\lambda}(\theta^{\lambda}_{1}(p))=\tau^{\lambda}_{1}
(p),~ \mathcal{G}^{\lambda}(\theta^{\lambda}_{2}(p))=\tau^{\lambda}_{2}
(p),~\mathcal{G}^{\lambda}(\theta^{\lambda}(p))=\tau^{\lambda}(p),\eea for any 
polynomial $p$.

\begin{prop}\label{eqn p5.5}
\textup{(Geller \cite{G}, Proposition 2.1 (a), 2.7)}
\begin{itemize}
\item [\textbf{(a)}] If $p$ is a polynomial, then $\mathcal{F}^{\prime-1}p$ is Weyl 
transformable and hence $\mathcal{W}^{\lambda}(p)$ is well defined.
\item [\textbf{(b)}] If $p$ is a $U(n)$-harmonic polynomial, then $\mathcal{W}^{\lambda}
(p)=\tau^{\lambda}_{1}(p)=\tau^{\lambda}_{2}(p)=\tau^{\lambda}(p).$
\end{itemize}
\end{prop}

\begin{rem}\label{eqn r5.6}
In fact one can prove that for any polynomial $p$, $\mathcal{W}^{\lambda}
(p)=\tau^{\lambda}(p)$. Since we will be dealing with only harmonic polynomials we 
don't need this general result.
\end{rem}

We conclude this section with a short discussion about Gelfand pairs and $K$-
spherical functions on $\mathbb{H}^{n}$. For details see Benson et al. \cite{BJR}. 
Let $K$ be a compact Lie subgroup of Aut($\mathbb{H}^{n}$), the group automorphisms of 
$\mathbb{H}^{n}$. Each $k\in U(n)$, the group of $n\times n$ unitary 
matrices on $\mathbb{C}^{n}$, gives rise to an automorphism of $\mathbb{H}^{n}$, via 
$k\cdot (z,t)=(k\cdot z,t)$. So we can consider $U(n)$ as a subgroup of 
Aut($\mathbb{H}^{n}$). In fact $U(n)$ is a maximal connected, compact subgroup of 
Aut($\mathbb{H}^{n}$), and thus any connected, compact subgroup of 
Aut($\mathbb{H}^{n}$) is the conjugate of a subgroup $K$ of $U(n)$. The pair $(K,\mathbb{H}^{n})$ is called a Gelfand pair if $L^{1}_{K}
(\mathbb{H}^{n})$, the convolution subalgebra of $K$-invariant $L^{1}$ functions on 
$\mathbb{H}^{n}$, is commutative. Since conjugates 
of $K$ form Gelfand pairs with $\mathbb{H}^{n}$ if and only if $K$ does, and produce 
the same joint eigenfunctions for all $D\in \mathcal{L}_{K}(\mathfrak{h}_{n})$ (the set 
of all differential operators on $\mathbb{H}^{n}$ that are invariant under the action 
of $K$ and the left action of $\mathbb{H}^{n}$), which is our main interest in this 
paper, we will always assume that we are dealing with a connected, compact subgroup 
$K$ of $U(n)$. The $K$-action on $\mathbb{C}^{n}$ gives rise to a natural action on a 
function $f$ on $\mathbb{C}^{n}$ given by $k\cdot f(z)=f(k^{-1}\cdot z)$. Under this 
action we have the decomposition of $\mathcal{P}(\mathbb{C}^{n})$ into $K$-irreducible 
subspaces as  $$\mathcal{P}
(\mathbb{C}^{n})=\bigoplus_{\alpha\in\Lambda}V_{\alpha}~\textup{(algebraic direct 
sum)}.$$ Here $\Lambda$ denotes a countably infinite index set. Since $\mathcal{P}_{m}
(\mathbb{C}^{n})$, the space of    homogeneous holomorphic polynomials of degree $m$, 
is invariant under the $K$-action (as $K\subset U(n)$), we can take each $V_{\alpha}$ 
to be contained in some $\mathcal{P}_{m}(\mathbb{C}^{n})$. Define the unitary 
representation $U^{\lambda}$ of $K$ on the Hilbert space $\mathcal{H}^{\lambda}$ as 
follows: if $u\in\mathcal{H}^{\lambda}$,
\begin{eqnarray*} U^{\lambda}(k)u=\begin{cases}\bar{k}\cdot u ~\textup{if} 
~\lambda>0\\ k\cdot u ~\textup{if}~ \lambda<0\,.\end{cases}
\end{eqnarray*}
Since $(K,\mathbb{H}^{n})$ is a Gelfand pair, $U^{\lambda}$ is multiplicity free (see 
Benson et al. \cite{BJR}, Theorem 1.7). $\mathcal{P}(\mathbb{C}^{n})$ being dense in 
$\mathcal{H}^{\lambda}$, we get the
same decomposition of $\mathcal{H}^{\lambda}$ into $U^{\lambda}$-irreducible subspaces 
: 
$$\mathcal{H}^{\lambda}=\bigoplus^{\bot}_{\alpha\in\Lambda}V_{\alpha}~\textup{(orthogonal 
Hilbert space decomposition)}.$$ Choose a basis $\{e^{\lambda}_{\alpha\nu}:\nu=1,2,
\cdots d(\alpha)\}$ for each $V_{\alpha}$ so that $\{e^{\lambda}_{\alpha\nu}:
\alpha\in\Lambda,~\nu=1,2,\cdots d(\alpha)\}$ is an orthonormal basis for 
$\mathcal{H}^{\lambda}$. We will use this basis in the later sections. The behaviour 
of $K$-action on a function under Weyl transform is given by the following 
proposition.

\begin{prop}\label{eqn p5.7}
\textup{(Geller \cite{G}, Proposition 1.3)}
\begin{itemize}
\item [\textbf{(a)}] $\Pi^{\lambda}(k\cdot z)=\big(U^{\lambda}
(k)\big)^{-1}\Pi^{\lambda}(z)\big(U^{\lambda}(k)\big)$.
\item [\textbf{(b)}] If $f\in L^{2}(\mathbb{C}^{n})$,~~ $\mathcal{G}^{\lambda}(k\cdot 
f)=\left(U^{\lambda}(k)\right)\mathcal{G}^{\lambda}f\left(U^{\lambda}(k)\right)^{-1}.$
\item [\textbf{(c)}] For any polynomial $p$, $\mathcal{W}^{\lambda}(k\cdot 
p)=\left(U^{\lambda}(k)\right)\mathcal{W}^{\lambda}(p)\left(U^{\lambda}
(k)\right)^{-1}.$
\end{itemize}
\end{prop}

In fact, \textbf{(c)} is not proved in \cite{G}. But using the definition of Weyl 
transform of a tempered distribution, one can show that \textbf{(b)} is true for any 
Weyl transformable tempered distribution. Since Euclidean Fourier transform commutes 
with the action of $K$, \textbf{(c)} follows.

A smooth $K$-invariant function $\phi:\mathbb{H}^{n}\longrightarrow\mathbb{C}$ is 
called $K-spherical$ if $\phi(0,0)=1$ and $\phi$ is a joint eigenfunction for all 
$D\in\mathcal{L}_{K}(\mathfrak{h}_{n})$. In \cite{BJR}, the authors describe all bounded 
$K$-spherical functions, their forms, and the corresponding eigenvalues. We summarise 
these in the following theorem. Assume that $dk$ is the normalized Haar measure on 
$K$.

\begin{thm}\label{eqn t5.8}
There are two distinct classes of bounded $K$-spherical functions.\\
\textup{(a)} The first type is parametrized  by $(\lambda,
\alpha)\in\mathbb{R}^{*}\times\Lambda$ \textup{(}$\mathbb{R}^{*}$ denotes the set of 
all non-zero real numbers\textup{)}, and given by $$\phi^{\lambda}_{\alpha}
(z,t)=\int_{K}\langle\Pi^{\lambda}\left(k\cdot(z,t)\right)v,v\rangle dk,$$ for any 
unit vector $v\in V_{\alpha}$. Each $\phi^{\lambda}_{\alpha}$ has the form 
$$\phi^{\lambda}_{\alpha}(z,t)=e^{i\lambda t}q^{\lambda}_{\alpha}(z)e^{-|\lambda||z|
^{2}},$$ where $q_{\alpha}^{\lambda}(z)$ is a polynomial. The corresponding eigenvalue 
$\widetilde{\mu}^{\lambda}_{\alpha}$s are distinct, and can be obtained from the equation 
\textup{(}for any non-zero $v\in V_{\alpha}\textup{)}$, \bea\label{eqn45}\Pi^{\lambda}
(D)v=\widetilde{\mu}^{\lambda}_{\alpha}(D)v~\forall D\in\mathcal{L}_{K}(\mathfrak{h}_{n}).
\eea
\textup{(b)} The second type is parametrized by  $\mathbb{C}^{n}/K$, the space of $K$-orbits in $\mathbb{C}^{n}$. For $\omega\in\mathbb{C}^{n}$ we write $\eta_{\omega}$ 
for the associated $K$-spherical function. One has 
$\eta_{\omega}=\eta_{\omega^{\prime}}$, if $K\cdot\omega=K\cdot\omega^{\prime}$. 
$\eta_{\omega}(z,t)$ is independent of $t$, and is given by $$\eta_{\omega}
(z,t)=\int_{K}e^{i\textup{Re}\langle\omega,k\cdot z\rangle}dk.$$
\end{thm}

Let $\mathcal{L}_{K}^{\lambda}(\mathbb{C}^{n})=\{D^{\lambda} : D\in\mathcal{L}_{K}
(\mathfrak{h}_{n})\}$. Clearly $\mathcal{L}_{K}^{\lambda}
(\mathbb{C}^{n})\subset\mathcal{L}^{\lambda}(\mathbb{C}^{n})$.
Define $$\psi_{\alpha}^{\lambda}(z)=\frac{1}
{\kappa_{\alpha}^{\lambda}}\int_{K}\left\langle\Pi^{\lambda}\left(k\cdot 
z)\right)v,v\right\rangle dk,$$ where $v$ is any unit vector in $V_{\alpha},$ and 
$\kappa_{\alpha}^{\lambda}$ is the square of $L^{2}$ norm of 
$\int_{K}\left\langle\Pi^{\lambda}\left(k\cdot z)\right)v,v\right\rangle dk$. The 
functions $\psi_{\alpha}^{\lambda}$ are real valued and 
$\psi_{\alpha}^{\lambda}=\psi_{\alpha}^{-\lambda}$ (see \cite{BJR}, Remark, page-428). 
Therefore we can write $\phi_{\alpha}^{-\lambda}(z,t)=\kappa_{\alpha}^{\lambda}e^{-
i\lambda t}\psi_{\alpha}^{\lambda}$. Then with the property $||
\psi^{\lambda}_{\alpha}||^{2}_{2}=\frac{1}{\kappa_{\alpha}^{\lambda}}$, 
$\psi_{\alpha}^{\lambda}(z)$ is the unique (upto constant multiple) bounded joint eigenfunction for all 
$D^{\lambda}\in\mathcal{L}_{K}^{\lambda}(\mathbb{C}^{n})$ with the eigenvalue 
$\mu^{\lambda}_{\alpha}$, where $\mu_{\alpha}^{\lambda}
(D^{\lambda})=\widetilde{\mu}_{\alpha}^{-\lambda}(D)$ for all $D\in\mathcal{L}_{K}
(\mathfrak{h}_{n}).$

\begin{rem}\label{eqn r5.9}
Equation (\ref{eqn45}) can be restated in terms of Weyl transform as (for any non-zero 
$v\in V_{\alpha}$) $$\mathcal{G}^{\lambda}(D)v=\mu_{\alpha}^{\lambda}(D)v~\forall 
D\in\mathcal{L}^{\lambda}_{K}(\mathbb{C}^{n}).$$
\end{rem}

\begin{rem}\label{eqn r5.10}
Let $K=U(n)$. $\mathcal{L}^{\lambda}_{K}(\C)$ is generated by the special Hermite operator $$\mathcal{L}^{\lambda}:=\sum_{j=1}^{n}L^{\lambda}_j 
\overline{L}^{\lambda}_j+\overline{L}^{\lambda}_jL^{\lambda}_j=\sum^{n}_{j=1}\frac{\partial}
{\partial\bar{z}_j}\frac{\partial}{\partial z_j}-\lambda|z|
^{2}+\sum_{j=1}^{n}\bigg(\bar{z}_j\frac{\partial}{\partial\bar{z}_j}-
z_j\frac{\partial}{\partial z_j}\bigg).$$ The decomposition of $\mathcal{P}(\C)$ into $K$-irreducible subspaces is given by $\mathcal{P}(\C)=\oplus_{k\in\mathbb{Z}^{+}}\mathcal{P}_{k}(\C)$. Recall that $\mathcal{P}_{k}(\C)$ is the space of all homogeneous holomorphic polynomial on $\C$ of degree $k$. The bounded $K$-spherical functions are parametrized by $\mathbb{Z}^{+}$, the set of non negative integers. The corresponding $\psi_{k}^{\lambda}$'s are given by $$\psi^{\lambda}_{k}(z)=\pi^{-n}(2|\lambda|)^nL^{n-1}_k\big(2|\lambda||z|^2\big)e^{-|\lambda||z|^2},$$ where $L^{n-1}_k$ is the Laguerre polynomial of type $n-1$, and the corresponding eigenvalues are given by $\mu^{\lambda}_{k}(\mathcal{L}^{\lambda})=-2|\lambda|(2k+n)$. It is easy to see that (or by Corollary 2.3 in \cite{BJR}), $\psi^{\lambda}_k=\Sigma_{V_{\alpha}\subset\mathcal{P}_{k}(\C)}\psi^{\lambda}_{\alpha}$.  
\end{rem}

\section{Weyl transform of $K$-invariant functions}
Through out this section we assume that $(K,\mathbb{H}^{n})$ is a Gelfand pair. Let 
$\lambda\in\mathbb{R}^{*}$ be fixed.

\begin{prop}\label{eqn p6.1}
Let $T\in\mathscr{S}^{\prime}(\mathbb{C}^{n})$ be $K$-invariant and Weyl 
transformable. Then $\mathcal{G}^{\lambda}T$ is a constant multiple of the identity 
operator on each $V_{\alpha}$.
\end{prop}

\begin{proof}
For simplicity of notation we suppress the superscript $\lambda$ from the notation 
introduced in the previous section. Since $T$ is $K$-invariant, there exists a sequence 
$\{f_{j}\}$ of smooth, compactly supported, $K$-invariant functions on 
$\mathbb{C}^{n}$ such that $f_{j}$ converge  to $T$ in the topology of 
$\mathscr{S}^{\prime}(\C)$. By Proposition \textbf{\ref{eqn p5.7}} \textbf{(b)}, each 
$\mathcal{G}f_{j}$ commutes with all $U(k)$. Since the representation $U$ of $K$ on 
the various $V_{\alpha}$'s are irreducible and inequivalent, $\mathcal{G}f_{j}$ preserves each 
$V_{\alpha}$. Thus, by Schur's Lemma, $\mathcal{G}f_{j}$ is constant on each 
$V_{\alpha}$. Hence by Proposition \textbf{\ref{eqn p5.2}} we are done.                                                                                                                                                                                                                                                                                                                                                                                           
\end{proof}

Since the Euclidean Fourier transform commutes with the action of $K$, an easy consequence 
of the above proposition and Proposition \textbf{4.5 (a)} is the following corollary.

\begin{cor}\label{eqn c6.2}
Weyl correspondence of a $K$-invariant polynomial is constant on each $V_{\alpha}$.
\end{cor}

\begin{prop}\label{eqn p6.3}
$\mathcal{G}^{\lambda}(\psi^{\lambda}_{\alpha})=\mathcal{P}_{\alpha}$, where 
$\mathcal{P}_{\alpha}$ denotes the projection operator onto $V_{\alpha}$.
\end{prop}

\begin{proof}
As usual we suppress the superscript $\lambda$. By the previous proposition, 
$\mathcal{G}(\psi_{\alpha})$ is constant on each $V_{\beta}$, say $c_{\beta}I$. Let 
$v\in V_{\beta}$ be non-zero. Then $$\mathcal{G}(\psi_{\alpha})v=c_{\beta}v.$$ Let 
$D\in\mathcal{L}_{K}(\mathbb{C}^{n})$. By Corollary \textbf{\ref{eqn c5.4}} we have 
$$\mathcal{G}(\psi_{\alpha})\mathcal{G}(D)v=\mathcal{G}(D\psi_{\alpha})v=\mu_{\alpha}
(D)\mathcal{G}(\psi_{\alpha})v
=c_{\beta}\mu_{\alpha}(D)v.$$ Again, by Remark \textbf{\ref{eqn r5.9}}, $$\mathcal{G}
(\psi_{\alpha})\mathcal{G}(D)v=\mu_{\beta}(D)\mathcal{G}
(\psi_{\alpha})v=c_{\beta}\mu_{\beta}(D)v.$$ Therefore we have $$c_{\beta}\mu_{\alpha}
(D)v=c_{\beta}\mu_{\beta}(D)v.$$ This is true for all $D\in\mathcal{L}_{K}
(\mathbb{C}^{n})$.
Since $\mu_{\beta}\neq\mu_{\alpha}$ for $\beta\neq\alpha$, we get $c_{\beta}=0$. 
Therefore $\mathcal{G}(\psi_{\alpha})$ is zero on $V_{\beta}$, if $\beta\neq\alpha$. 
Now take a unit vector $u$ from $V_{\alpha}$. Then \beas c_{\alpha}&=& \langle 
\mathcal{G}(\psi_{\alpha})u,u\rangle \\&=& \int_{\mathbb{C}^{n}}\psi_{\alpha}
(z)\langle\Pi(z)u,u\rangle dz\\&=& \int_{K}\int_{\mathbb{C}^{n}}\psi_{\alpha}
(z)\langle\Pi(k\cdot z)u,u\rangle dz dk.\eeas Since $\Pi$ is unitary, $|
\langle\Pi(k\cdot z)u,u\rangle|\leq 1$. Therefore we can use Fubini's theorem to 
interchange the integrals and the fact that $\psi_{\alpha}$ are real valued to get \beas 
c_{\alpha}=\kappa_{\alpha}\int_{\mathbb{C}^{n}}\psi_{\alpha}(z)\psi_{\alpha}(z)dz= 
\kappa_{\alpha}||\psi_{\alpha}||_{2}^{2}=1.\eeas Hence the proof.
\end{proof}

\begin{prop}\label{eqn p6.4}
Let $f\in L^{2}(\mathbb{C}^{n})$. Then $f = 
\Sigma_{\alpha\in\Lambda}f\times^{\lambda}\psi_{\alpha}^{\lambda}$, where the series converges in $L^{2}(\C)$.
\end{prop}

\begin{proof}
Since the index set $\Lambda$ is countable, we can identify $\Lambda$ with the set of 
natural numbers $\mathbb{N}$. For $j\in \mathbb{N}$, $$\mathcal{G}(f)|
_{V_{j}}=\mathcal{G}(f)\mathcal{P}_{j}=\mathcal{G}(f)\mathcal{G}(\psi_{j})=
\mathcal{G}(f\times\psi_{j}).$$ Therefore, \beas\big|\big|\mathcal{G}
(f)-\mathcal{G}\bigg(\sum_{j=1}^{N}f\times\psi_{j}\bigg)\big|\big|
_{\textup{HS}}^{2}&=& \sum_{j>N}\sum_{\nu=1}^{d(j)}||\mathcal{G}(f)e_{j\nu}||
^{2}_{2}\longrightarrow 0\eeas as $N\longrightarrow\infty$, since $\mathcal{G}(f)$ is 
a Hilbert-Schmidt operator. Hence by the Plancherel theorem (Theorem \textbf{\ref{eqn 
t5.1}}) we are done.
\end{proof}

The above proposition was also proved in \cite{RS}.

\section{generalized spherical functions and Weyl transform of K-type functions}
From now on we assume that $(K,\mathbb{H}^{n})$ is a Gelfand pair, where $K$ is a 
connected, compact subgroup of $U(n)$, whose action on $\mathbb{C}^{n}$ is polar. 
More precisely, if we identify $U(n)$ as a subgroup of $SO(2n)$ and $\mathbb{C}^{n}$ 
with $\mathbb{R}^{2n}$, then the action of $K\subset SO(2n)$ on $\mathbb{R}^{2n}$ is 
polar, so that we can use all the results about polar actions from the first three 
sections. Our main aims are to find all generalized $K$-spherical functions (Theorem 
\textbf{\ref{eqn t7.13}}) and  give a formulae for Weyl transform of a function $F\in 
\mathscr{S}^{\delta}(\mathbb{C}^{n})$ (Theorem \textbf{\ref{eqn t7.4}}). Here $\mathscr{S}^{\delta}
(\C):=\{F\in\mathcal{E}^{\delta}(\C):F_{ij}\in\mathscr{S}(\C)\}.$ Theorem 
\textbf{\ref{eqn t7.4}} can be thought of as a generalization of the Theorem 4.2 in 
\cite{G}, which is a Hecke-Bochner type identity. To prove his theorem, Geller 
introduced certain Hilbert spaces of linear operators which turned out to be analogous 
to $L^{2}(S^{2n-1})$ and showed that the Weyl correspondence of $U(n)$-harmonic 
polynomials are dense in these Hilbert spaces. We show that a similar result holds for 
any $K$ (Proposition \textbf{\ref{eqn p7.7}}), and use this to prove our theorems.

For two positive integers $p$ and $q$, let $\mathcal{R}_{p\times q}^{\lambda}(\C)$ 
denote the set of all $p\times q $ matrices whose entries belong to 
$\mathcal{R}^{\lambda}(\C)$; $\mathcal{O}_{p\times q}(\mathcal{H}^{\lambda})$ denote 
the set of all $p\times q $ matrices whose entries belong to $\mathcal{O}
(\mathcal{H}^{\lambda})$; and $\mathcal{H}^{\lambda}_{p\times q}$ denote the same 
whose entries belong to $\mathcal{H}^{\lambda}$. For 
$\mathcal{T}\in\mathcal{O}_{p\times q}(\mathcal{H}^{\lambda})$, define its action 
as follows: 
if $u\in\mathcal{P}(\C)$, the $(i,j)$th entry of $\mathcal{T}u$ is equal to 
$\mathcal{T}_{ij}u$. Let $\mathscr{S}_{p\times q}^{\prime}(\C)$ denote the set of all 
$p\times q$ matrices with entries from $\mathscr{S}^{\prime}(\C)$. For 
$g\in\mathscr{S}^{\prime}(\C)$ and $F\in\mathscr{S}^{\prime}_{p\times q}(\C)$, we 
define the following whenever they make sense. For a differential operator $D$ on 
$\C$, define $DF:\C\ra\mathcal{M}_{p\times q}$, by $(DF)_{ij}(z)=DF_{ij}(z).$ If 
$\mathcal{D}\in\mathcal{R}^{\lambda}_{p\times q}(\C)$, define $\mathcal{D}g:
\C\ra\mathcal{M}_{p\times q}$, by $(\mathcal{D}g)_{ij}(z)=\mathcal{D}_{ij}g(z).$ 
Define $\mathcal{G}^{\lambda}F, \mathcal{W}^{\lambda}F\in\mathcal{O}_{p\times q}
(\mathcal{H}^{\lambda}); \tau^{\lambda}(P^{\delta})\in\mathcal{O}_{d(\delta)\times 
l(\delta)}(\mathcal{H}^{\lambda});$ and $\theta^{\lambda}
(P^{\delta})\in\mathcal{R}^{\lambda}_{p\times q}(\C)$, 
$\mathcal{F}^{\prime-1}P^{\delta}\in\mathscr{S}^{\prime}_{d(\delta)\times l(\delta)}
(\C)$ 
similarly. For $S\in\mathcal{O}({H}^{\lambda})$,  
$\mathcal{T}\in\mathcal{O}_{p\times q}(\mathcal{H}^{\lambda})$, define 
$\mathcal{T}S\in\mathcal{O}_{p\times q}(\mathcal{H}^{\lambda})$ by 
$(\mathcal{T}S)_{ij}=\mathcal{T}_{ij}\circ S$. Similarly define 
$S\mathcal{T}\in\mathcal{O}_{p\times q}(\mathcal{H}^{\lambda})$. For a $r\times p$ 
constant matrix $C$, define $C\mathcal{T}\in\mathcal{O}_{r\times q}
(\mathcal{H}^{\lambda})$, by 
$(C\mathcal{T})_{ij}=\Sigma_{k=1}^{p}C_{ik}\mathcal{T}_{kj}$.

If $f$ is a joint eigendistribution of all $D\in\mathcal{L}^{\lambda}_{K}(\C)$, then 
$K$ being a subgroup of $U(n)$, it is also a joint eigendistribution of all 
$D\in\mathcal{L}^{\lambda}_{U(n)}(\C)$. But 
$\mathcal{L}_{U(n)}^{\lambda}(\C)$ is generated by the special Hermite operator $\mathcal{L}^{\lambda}$, 
which is elliptic (see Remark \textbf{\ref{eqn r5.10}}). So we can assume that $f$ 
is smooth. Therefore, we will consider only smooth joint eigenfunctions for 
$\mathcal{L}^{\lambda}_{K}(\C)$. 

\begin{defn}\label{eqn d7.1}     
A function $\Psi\in\mathcal{E}^{\delta}(\mathbb{C}^{n})$ is said to be a generalized 
$K$-spherical function of type $\delta$ corresponding to $\mu^{\lambda}_{\alpha}$, if 
it is a joint eigenfunction for all $D\in\mathcal{L}^{\lambda}_{K}(\mathbb{C}^{n})$ 
with eigenvalue $\mu^{\lambda}_{\alpha}$, i.e $D\Psi=\mu^{\lambda}_{\alpha}(D)\Psi$ 
for all $D\in\mathcal{L}^{\lambda}_{K}(\mathbb{C}^{n})$.
\end{defn}

By Proposition \textbf{\ref{eqn p5.5}} \textbf{(b)} it follows that, for any $K$-
harmonic (hence $U(n)$- harmonic) polynomial $p$, $\theta^{\lambda}_{1}
(p)=\theta^{\lambda}_{2}(p)$. Hence $ \theta^{\lambda}
(P^{\delta})=\theta^{\lambda}_{1}(P^{\delta})=\theta^{\lambda}_{2}(P^{\delta}).$ 
Define $\Psi_{\alpha}^{\delta,\lambda}\in\mathcal{E}^{\delta}(\mathbb{C}^{n})$, by 
$$\Psi^{\delta,\lambda}_{\alpha}=\theta^{\lambda}(P^{\delta})\psi^{\lambda}_{\alpha}.
$$

\begin{prop}\label{eqn p7.2}
$\Psi_{\alpha}^{\delta,\lambda}$ is a generalized $K$-spherical function of type 
$\delta$ corresponding to $\mu^{\lambda}_{\alpha}$.
\end{prop}

\begin{proof}
We suppress the superscript $\lambda$. Since 
$\theta(P^{\delta})\in\mathcal{R}_{d(\delta)\times l(\delta)}(\C)$, by Corollary 
\textbf{\ref{eqn c5.4}}, $$\mathcal{G}(\Psi_{\alpha}^{\delta})=\mathcal{G}
(\theta(P^{\delta}))\mathcal{G}(\psi_{\alpha})=
\tau(P^{\delta})\mathcal{G}(\psi_{\alpha})=\mathcal{W}(P^{\delta})\mathcal{G}
(\psi_{\alpha}).$$ Therefore, by Proposition \textbf{\ref{eqn p5.7}}, for $k\in K$, \beas &&\mathcal{G}
(k^{-1}\cdot\Psi_{\alpha}^{\delta})=U(k)^{-1}\mathcal{G}(\Psi_{\alpha}^{\delta})U(k)=
U(k)^{-1}\mathcal{W}(P^{\delta})\mathcal{G}(\psi_{\alpha})U(k)\\&=&
U(k)^{-1}\mathcal{W}(P^{\delta})U(k)U(k)^{-1}\mathcal{G}
(\psi_{\alpha})U(k)=\mathcal{W}(k^{-1}\cdot P^{\delta})\mathcal{G}
(k^{-1}\cdot\psi_{\alpha})\\&=&\delta(k)\mathcal{W}(P^{\delta})\mathcal{G}
(\psi_{\alpha})=
\delta(k)\mathcal{G}(\Psi_{\alpha}^{\delta}).\eeas So we get $\Psi_{\alpha}^{\delta}
(k\cdot z)=\delta(k)\Psi_{\alpha}^{\delta}(z)$, and hence 
$\Psi_{\alpha}^{\delta}\in\mathcal{E}^{\delta}(\mathbb{C}^{n}).$
Again, $D\Psi_{\alpha}^{\delta}=\mu_{\alpha}(D)\Psi_{\alpha}^{\delta}$ for all 
$D\in\mathcal{L}_{K}(\mathbb{C}^{n})$, because \beas && \mathcal{G}
(D\Psi_{\alpha}^{\delta})=\mathcal{G}(\Psi_{\alpha}^{\delta})\mathcal{G}
(D)=\mathcal{G}
(\theta(P^{\delta}))\mathcal{G}(\psi_{\alpha})\mathcal{G}(D)\\&=&\mathcal{G}
(\theta(P^{\delta}))\mathcal{G}(D\psi_{\alpha})=\mu_{\alpha}(D)\mathcal{G}
(\theta(P^{\delta}))\mathcal{G}(\psi_{\alpha})=\mu_{\alpha}(D)
\mathcal{G}(\Psi_{\alpha}^{\delta}).\eeas
Therefore $\Psi_{\alpha}^{\delta}$ is a generalized $K$-spherical function of type $\delta$ corresponding to $\mu_{\alpha}$.
\end{proof}

Note that $\Psi_{\alpha}^{\delta,\lambda}(z)$ is equal to $e^{-|\lambda||z|^{2}}$ 
times a polynomial in $\mathcal{E}^{\delta}(\mathbb{C}^{n})$, and hence by Corollary 
\textbf{\ref{eqn c4.8}}, there is a unique $K$-invariant polynomial $L^{\delta,
\lambda}_{\alpha}:\mathbb{C}^{n}\longrightarrow\mathcal{M}_{l(\delta)\times 
l(\delta)}$ such that \bea\label{eqn61}\Psi_{\alpha}^{\delta,\lambda}(z)=P^{\delta}
(z)L^{\delta,\lambda}_{\alpha}(z)e^{-|\lambda||z|^{2}}.\eea
Define the $l(\delta)\times l(\delta)$ constant matrix $A_{\alpha}^{\delta,\lambda}$ 
by \beas A_{\alpha}^{\delta,\lambda}&=&\int_{\mathbb{C}^{n}}[\Psi_{\alpha}^{\delta,
\lambda}(z)]^{\star}\Psi_{\alpha}^{\delta,\lambda}(z)dz\\&=& \int_{\mathbb{C}^{n}}
[L_{\alpha}^{\delta,\lambda}(z)]^{\star}\Upsilon_{\delta}(z)L_{\alpha}^{\delta,
\lambda}(z)e^{-2|\lambda||z|^{2}}dz.
\eeas Clearly $A_{\alpha}^{\delta,\lambda}$ is positive definite. Let $\alpha(\delta)$ 
denote the number of linearly independent columns in $\Psi_{\alpha}^{\delta,\lambda}$. Let $C_i$ denote the $i$th column. Choose $C_{l(1)},C_{l(2)},\cdots,C_{l\big(\alpha(\delta)\big)}$ linearly independent and $l(1)<l(2),\cdots<l\big(\alpha(\delta)\big)$. Let the remaining columns be $C_{m(1)},C_{m(2)},\cdots,C_{m(l(\delta)-\alpha(\delta))}$, with $m(1)<m(2)<\cdots <m(l(\delta)-\alpha(\delta))$ Let $$I_1=\big\{l(1),l(2),\cdots,l\big(\alpha(\delta)\big)\big\},$$ $$I_{2}=\big\{m(1),m(2),\cdots,m\big(l(\delta)-\alpha(\delta)\big)\big\}.$$ Then $I_1$ and $I_2$ are disjoint, and $$I_1\cup I_2=\{1,2,\cdots,l(\delta)\}.$$ Let $\widetilde{\Psi}^{\delta,\lambda}_{\alpha}$ be the $d(\delta)\times\alpha(\delta)$ matrix whose $r$th column is $C_{l(r)}$, where $l(r)\in I_1$; and $\widetilde{L}^{\delta,\lambda}_{\alpha}$ be the $l(\delta)\times\alpha(\delta)$ matrix whose $r$th column is $C_{l(r)}$, where $l(r)\in I_1$. 
Then clearly we have 
$$\widetilde{\Psi}_{\alpha}^{\delta,\lambda}(z)=P^{\delta}(z)\widetilde{L}_{\alpha}^{\delta,
\lambda}(z)e^{-|\lambda||z|^{2}}.$$ Define \beas \widetilde{A}_{\alpha}^{\delta,
\lambda}&=&\int_{\mathbb{C}^{n}}[\widetilde{\Psi}_{\alpha}^{\delta,\lambda}
(z)]^{\star}\widetilde{\Psi}_
{\alpha}^{\delta,\lambda}(z)dz\\&=& \int_{\mathbb{C}^{n}}[\widetilde{L}_{\alpha}^{\delta,
\lambda}(z)]^{\star}\Upsilon_{\delta}(z)\widetilde{L}_{\alpha}^{\delta,\lambda}(z)
e^{-2|\lambda||z|^{2}}dz.
\eeas
Note that $\widetilde{A}_{\alpha}^{\delta,\lambda}$ is precisely the $\alpha(\delta)\times 
\alpha(\delta)$
matrix, obtained by deleting the $m(r)$th rows and columns from $A_{\alpha}^{\delta,\lambda}$, where $m(r)\in I_2$.

\begin{lem}\label{eqn l7.3}
If $\alpha(\delta)>0$, $\widetilde{A}_{\alpha}^{\delta,\lambda}$ is invertible.
\end{lem}

\begin{proof}
As usual we suppress the superscript $\lambda$. If $\widetilde{A}_{\alpha}^{\delta}$ is 
not invertible, then there exist a non-zero vector $\underline{e}
\in\mathbb{C}^{\alpha(\delta)}$, such that 
$\langle(\widetilde{A}_{\alpha}^{\delta}\underline{e}),\underline{e}\rangle=0$ (here 
$\langle,\rangle$ denotes the usual hermitian inner product on 
$\mathbb{C}^{\alpha(\delta)}$). Since 
$$\widetilde{A}^{\delta}_{\alpha}=\int_{\mathbb{C}^{n}}
[\widetilde{\Psi}_{\alpha}^{\delta}(z)]^{\star}\widetilde{\Psi}_{\alpha}^{\delta}(z)dz,$$  we 
get $\widetilde{\Psi}_{\alpha}^{\delta}(z)\underline{e}=0$ for all $z$, implying that the 
columns of $\widetilde{\Psi}_{\alpha}^{\delta}$ are linearly dependent, which is a 
contradiction. Hence the proof.
\end{proof}

Before we state one of our main theorems we introduce some more notation. For $F\in\mathcal{E}^{\delta}(\C)$, define $l(\delta)\times l(\delta)$ matrix $C^{\delta,\lambda}_{\alpha}(F)$ as follows : Let $\widetilde{C}^{\delta,\lambda}_{\alpha}(F)$ denotes the $\alpha(\delta)\times l(\delta)$ matrix whose $r$th row is the $l(r)$th row of $C^{\delta,\lambda}_{\alpha}$, where $l(r)\in I_1$; and $\widetilde{\widetilde{C}}^{\delta,\lambda}_{\alpha}(F)$ denotes the $\big(l(\delta)-\alpha(\delta)\big)\times l(\delta)$ matrix whose $r$th row is the $m(r)$th row of $C^{\delta,\lambda}_{\alpha}$, where $m(r)\in I_2$. 
\begin{equation}\label{eqn63}
 \left.\begin{aligned}
        \widetilde{C}^{\delta,\lambda}_{\alpha}(F)&=\big(\widetilde{A}_{\alpha}^{\delta,
\lambda}\big)^{-1}\int_{\C}[\widetilde{\Psi}_{\alpha}^{\delta,\lambda}
(z)]^{\star}F(z)dz\\&=\big(\widetilde{A}_{\alpha}^{\delta,
\lambda}\big)^{-1}\int_{\mathbb{C}^{n}}[\widetilde{L}^{\delta,\lambda}_{\alpha}(z)]
^{\star}
\Upsilon_{\delta}(z)G(z)e^{-|\lambda||z|^{2}}dz,\\
 \widetilde{\widetilde{C}}^{\delta,\lambda}_{\alpha}(F)&=0.
\end{aligned}
\right\}
\qquad 
\end{equation} whenever the integrals exist.


\begin{thm} \label{eqn t7.4}
\textup{(}Hecke-Bochner identity\textup{)} Suppose $F=P^{\delta}G\in \mathscr{S}^{\delta}(\mathbb{C}^{n})$, where $G$ is $K$-invariant. 
Then $\mathcal{G}^{\lambda}(F)=\mathcal{W}^{\lambda}(P^{\delta})S,$ where 
$S\in\mathcal{O}_{l(\delta)\times l(\delta)}(\mathcal{H}^{\lambda})$ whose action on 
each $V_{\alpha}$ is the $l(\delta)\times l(\delta)$ constant matrix 
$C^{\delta,\lambda}_{\alpha}(F)$; equivalently if $F=P^{\delta}G\in\mathscr{S}^{\delta}(\C)$, where $G$ is $K$-invariant, then $F\times^{\lambda}\psi^{\lambda}_{\alpha}= 
\Psi_{\alpha}^{\delta,\lambda}C^{\delta,\lambda}_{\alpha}(F)$, where $C^{\delta,\lambda}_{\alpha}(F)$ is defined by  (\ref{eqn63}).



\end{thm}

Before proving this general theorem let us consider the special case $K=U(n)$, and describe $\Psi^{\delta,\lambda}_{\alpha},\widetilde{\Psi}^{\delta,\lambda}_{\alpha},L^{\delta,\lambda}_{\alpha}\cdots$. 
We also show that, in this special case, the above theorem is precisely Theorem 4.2 in \cite{G}. We put these in the following remark.    

\begin{rem}\label{eqn r7.31}
For this remark $K$ always stands for $U(n)$.
Let $M$ be the stabilizer of the $K$-regular point $e_1=(1,0,\cdots,0)\in\C$; then $M$ can be identified with $U(n-1)$. Via the map $kM\rightarrow k\cdot e_1$, we have the identification $K/M=K\cdot e_1=S^{2n-1}$. Note that, in this special case, the space $H$ consists of all polynomials $P$ such that $\triangle_{\C}P=0$, where $\triangle_{\C}=\Sigma_{j=1}^{n}\frac{\partial^2}{\partial z_j\partial\bar{z}_j}$. for each pair of non-negative integers $(p,q)$, let $\mathcal{P}_{pq}$ be the space of all polynomials $P$ in $z$ and $\bar{z}$ of the form $$P(z)=\sum_{|\alpha|=p}\sum_{|\beta|=q}c_{\alpha\beta}z^{\alpha}\bar{z}^{\beta}.$$ Let $\mathcal{H}_{pq}=H\cap\mathcal{P}_{pq}$     , and $\mathcal{S}_{pq}$ denote the space of restrictions of elements of $\mathcal{H}_{pq}$ to $S^{2n-1}$. The elements of $\mathcal{H}_{pq}$ are called bigraded solid harmonics of degree $(p,q)$, and those of $\mathcal{S}_{pq}$ are called bigraded spherical harmonics of degree $(p,q)$. The $K$-action on $S^{2n-1}$ defines an unitary representation on $L^{2}(S^{2n-1})$. Clearly each $\mathcal{S}_{pq}$ is a $K$-invariant subspace. Let $\delta_{pq}$ denotes the restriction of this representation to $\mathcal{S}_{pq}$. It is well known that these describe all inequivalent, irreducible, unitary representations in $\widehat{K}_M$. Note that, according to our general notation (in Section 3), $H_{\delta_{pq}}=\mathcal{H}_{pq},\mathcal{E}_{\delta_{pq}}(K/M)=\mathcal{S}_{pq}$ and $l(\delta_{pq})=$dim$V^{M}_{\delta_{pq}}=1$. 

Using the similar arguments as in Remark \textbf{\ref{eqn r4.01}}, we can prove the 
following : first note that, in this case, $\stackrel{\vee}{\delta}_{pq}$ is 
equivalent to $\delta_{qp}$. Take $P^{i}_{pq}\in\mathcal{H}_{pq}$, and $Y^{i}_{pq}$ to 
be their restrictions to $S^{2n-1}$ so that $\{Y^{i}_{pq}:i=1,2,\cdots,d(p,q)\}$ forms 
an orthonormal basis for $\mathcal{S}_{pq}$. Then it is possible to choose orthonormal 
ordered basis $\textbf{b}=\{v_1,v_2,\cdots,v_{d(p,q)}\}$ for $V_{\delta_{pq}}$ and 
$\textbf{b}^{M}=\{v_1\}$ for $V^{M}_{\delta_{pq}}$, and a basis $\textbf{e}$ for 
$F_{\delta_{pq}}=$Hom$({V_{\delta_{pq}},\mathcal{H}_{pq}})$, so that, with respect to 
these bases, $Y^{\check{\delta}_{pq}}:S^{2n-1}\rightarrow \mathcal{M}_{d(p,q)\times 
1}$ and $P^{\check{\delta}_{pq}}:\C\rightarrow \mathcal{M}_{d(p,q)\times 1}$ are given 
by $$Y^{\stackrel{\vee}{\delta}_{pq}}(\omega)=\sqrt{\frac{|S^{2n-1}|}
{d(p,q)}}\bigg[Y^1_{pq}(\omega),Y^2_{pq}(\omega),\cdots, Y^{d(p,q)}_{pq}
(\omega)\bigg]^t,\omega\in S^{2n-1},$$     $$P^{\stackrel{\vee}{\delta}_{pq}}
(z)=\sqrt{\frac{|S^{2n-1}|}{d(p,q)}}\bigg[P^1_{pq}(z),P^2_{pq}(z),\cdots, 
P^{d(p,q)}_{pq}(z)\bigg]^t,z\in \C.$$ In particular, $$P^{\stackrel{\vee}
{\delta}_{pq}}(z)=|z|^{p+q}Y^{\stackrel{\vee}{\delta}_{pq}}(z/|z|).$$ Also, we have 
$$\Upsilon_{\stackrel{\vee}{\delta}_{pq}}(z)=|z|^{2(p+q)}.$$

\indent {\rm{\textbf{(i)}}}
Recall the $U(n)$-spherical functions $\psi^{\lambda}_k$'s from Remark \textbf{\ref{eqn r5.10}}. The corresponding generalized $K$-spherical functions are given by $\Psi^{\delta_{pq},\lambda}_{k}=\theta^{\lambda}(P^{\delta_{pq}})\psi_{k}^{\lambda}$. 
Note that, $\Psi^{\delta_{pq},\lambda}_{k}:\C\rightarrow\mathcal{M}_{d(p,q)\times 1}$, $L^{\delta_{pq},\lambda}_{k}:\C\rightarrow\mathcal{M}_{1\times 1}$, and $A^{\delta_{pq},\lambda}_k$ is a $1\times 1$ matrix. Let $L^{\gamma}_{k}$ denotes the $k$ th degree Laguerre polynomial of type $\gamma$. We will show the following : For $\lambda>0$,
{\small\begin{align}\label{eqn r7.5.1}
          L^{\stackrel{\vee}{\delta}_{pq},\lambda}_k(z)&=\begin{cases}(-1)^q\pi^{-n}(2|\lambda|)^{n+p+q}L^{n+p+q-1}_{k-p}(2|\lambda||z|^2),~\textup{if}~p\leq k\\0,~\textup{if}~p>k,
\end{cases}\\
\label{eqn r7.5.2}
\Psi^{\stackrel{\vee}{\delta}_{pq},\lambda}_k(z)&=\begin{cases}(-1)^q\pi^{-n}(2|\lambda|)^{n+p+q}P^{\stackrel{\vee}{\delta}_{pq}}(z)L^{n+p+q-1}_{k-p}(2|\lambda||z|^2)e^{-|\lambda||z|^2},~\textup{if}~p\leq k\\ 0_{d(p,q)\times 1},~\textup{if}~p>k,
\end{cases}\\
\label{eqn r7.5.3}
 A^{\stackrel{\vee}{\delta}_{pq},\lambda}_k&=\begin{cases}\pi^{-n}(2|\lambda|)^{n+p+q}\frac{\Gamma(k+n+q)}{\Gamma(n)\Gamma(k-p+1)},~\textup{if}~p\leq k\\0,~\textup{if}~p>k;
\end{cases}
  \end{align}}
consequently $\widetilde{\Psi}^{\stackrel{\vee}{\delta}_{pq},\lambda}_k=\Psi^{\stackrel{\vee}{\delta}_{pq},\lambda}_k$, $\widetilde{L}^{\stackrel{\vee}{\delta}_{pq},\lambda}_k=L^{\stackrel{\vee}{\delta}_{pq},\lambda}_k$ and $\widetilde{A}^{\stackrel{\vee}{\delta}_{pq},\lambda}_k=A^{\stackrel{\vee}{\delta}_{pq},\lambda}_k$ if $p\leq k$. When $\lambda<0$, the role of $p$ and $q$ will be interchanged in the above formulae. We give a proof assuming $\lambda>0$. The proof for $\lambda<0$ will be similar. Since $$\Psi^{\stackrel{\vee}{\delta}_{pq},\lambda}_{k}=\theta^{\lambda}(P^{\stackrel{\vee}{\delta}_{pq}})\psi_{k}^{\lambda}=P^{\stackrel{\vee}{\delta}_{pq}}(z)L^{\stackrel{\vee}{\delta}_{pq},\lambda}_k(z)e^{-|\lambda||z|^2},$$ and $z_1^{p}\bar{z}_2^q\in\mathcal{H}_{pq}$, it follows that $$\theta^{\lambda}(z^p_1\bar{z}^q_2)\psi^{\lambda}_k=z^p_1\bar{z}_2^q L^{\stackrel{\vee}{\delta}_{pq},\lambda}_k(z)e^{-|\lambda||z|^2}.$$ Therefore, to prove (\ref{eqn r7.5.1}) it is enough to show that
\begin{equation} 
\label{eqn r7.5.4}
\theta^{\lambda}(z^p_1\bar{z}^q_2)\psi^{\lambda}_k=\begin{cases}(-1)^q\pi^{-n}(2|\lambda|)^{n+p+q}z^p_1\bar{z}^q_2L^{n+p+q-1}_{k-p}(2|\lambda||z|^2)e^{-|\lambda||z|^2},~\textup{if}~p\leq k\\ 0,~\textup{if}~p>k.
\end{cases}
\end{equation}
Since 
$\theta^{\lambda}(\bar{z}_{2})=\overline{R}^{\lambda}_{2}=\partial/\partial z_{2}-\lambda \bar{z}_{2}$, and 
$$\frac{\partial}{\partial z_{2}}\big[L_{k}^{n+q-1}\big(2|\lambda||z|
^{2}\big)\big]=2|\lambda|\bar{z}_{2}\big[L_{k}^{n+q-1}\big]^{\prime}\big(2|
\lambda||z|^{2}\big),$$ an easy calculation shows that 
$$\theta^{\lambda}(\bar{z}_{2})\psi^{\lambda}_{k}(z)=\pi^{-n}(-1)(2|
\lambda|)^{n+1}\bar{z}_{2}\big[\big(L_{k}^{n-1}\big)^{\prime}-
L_{k}^{n-1}\big](2|\lambda||z|^{2})e^{-|\lambda||z|^2}.$$ Using the 
well-known relations $$(L_{k}^{\alpha})^{\prime}=-
L_{k-1}^{\alpha+1},~L_{k}^{\alpha+1}=L_{k-1}^{\alpha+1}+L_{k}^{\alpha},$$ we get 
$$\theta^{\lambda}(\bar{z}_{2})\psi^{\lambda}_{k}(z)=\pi^{-n}(-1)(2|
\lambda|)^{n+1}\bar{z}_{2}L_{k}^{n+1-1}(2|\lambda||z|^2)e^{-|\lambda||z|^2}.$$ Since 
$\theta^{\lambda}(\bar{z}_{2}^{m+1})=(\overline{R}^{\lambda}_{2})^{m+1}=\overline{R}^{\lambda}_{2}\theta^{\lambda}(\bar{z}_{2}^{m})$, by 
an induction argument we can prove that $$\theta^{\lambda}(\bar{z}_{2}^{q})\psi^{\lambda}_{k}
(z)=\pi^{-n}(-1)^{q}(2|
\lambda|)^{n+q}\bar{z}_{2}^{q}L_{k}^{n+q-1}(2|\lambda||z|^2)e^{-|\lambda||z|^2},$$ for all non negative integer $q.$ Now fix a $q$. Then again by 
induction on $p$ and using a similar argument we can prove that 
$$\theta^{\lambda}(z_{1}^{p}\bar{z}_{2}^{q})\psi^{\lambda}_{k}(z)=\pi^{-n}(-1)^{q}
(2|\lambda|)^{n+p+q}(z_{1}^{p}\bar{z}_{2}^{q})L_{k-p}^{n+p+q-1}(2|\lambda||z|^2)e^{-|\lambda||z|^2},$$ whenever $p\leq k.$ In particular we 
get for any fixed $q$, $\theta^{\lambda}(z_{1}^{k}\bar{z}_{2}^{q})\psi_{k}
(z)$ is equal to a constant times $z_{1}^{k}\bar{z}_{2}^{q_{1}}e^{-|\lambda||
z|^{2}}$. Since $\theta^{\lambda}(z_{1}^{k+1}\bar{z}_{2}^{q})=(-
R^{\lambda}_{1})\theta^{\lambda}(z_{1}^{k}\bar{z}_{2}^{q})$ and 
$$R^{\lambda}_{1}\big(z_{1}^{k}\bar{z}_{2}^{q}e^{-|\lambda||z|
^{2}}\big)=(\partial/\partial \bar{z}_{1}+\lambda 
z_{1})\big(z_{1}^{k}\bar{z}_{2}^{q}e^{-|\lambda||z|^{2}}\big)=0~ 
\textup{(as} \lambda>0\textup{)},$$ we get 
$\theta(z_{1}^{k+1}\bar{z}_{2}^{q})\psi^{\lambda}_{k}(z)=0$, and consequently 
$\theta^{\lambda}(z_{1}^{p}\bar{z}_{2}^{q})\psi^{\lambda}_{k}(z)=0$ for all $p> k$. This finishes the proof of (\ref{eqn r7.5.1}). (\ref{eqn r7.5.2}) follows immediately from (\ref{eqn r7.5.1}). $A^{\stackrel{\vee}{\delta}_{pq},\lambda}_k$ has the formulae $$A^{\stackrel{\vee}{\delta}_{pq},\lambda}_k=\int_{\C}[L^{\stackrel{\vee}{\delta}_{pq},\lambda}_k(z)]^{\star}\Upsilon_{\check{\delta}_{pq}}(z)L^{\stackrel{\vee}{\delta}_{pq},\lambda}_k(z)e^{-2|\lambda||z|^2}dz.$$ Therefore, (\ref{eqn r7.5.3}) follows by (\ref{eqn r7.5.1}) and the fact that $$\int_{0}^{\infty}[L^{\gamma}_k(r)]^2e^{-r}r^{\gamma}dr=\frac{\Gamma(k+\gamma+1)}{\Gamma(k+1)}.$$ 


\indent {\rm{\textbf{(ii)}}} 
From the above discussion it is immediate that, for the special case $K=U(n)$, Theorem \textbf{\ref{eqn t7.4}} can be restated as follows (which is precisely Theorem 4.2 in \cite{G}) : Suppose $Pg\in\mathscr{S}(\C)$, where $P\in\mathcal{H}_{pq}$ and $g$ is a radial function. For $\lambda>0$, $\mathcal{G}^{\lambda}(Pg)=\mathcal{W}^{\lambda}(P)S$, $S\in\mathcal{O}(\mathcal{H}^{\lambda})$, whose action on each $\mathcal{P}_{k}(\C)$ is a constant $c_k$, where $c_k=0$ if $p>k$, and for $p\leq k$, it is given by $$c_k=(-1)^{q}\frac{\Gamma(n)\Gamma(k-p+1)}{\Gamma(k+n+q)}\int_{\C}g(z)L^{n+p+q-1}_{k-p}(2|\lambda||z|^2)|z|^{2(p+q)}e^{-|\lambda||z|^2}.$$ when $\lambda<0$, the role of $p$ and $q$ will be interchanged in the definition of $c_k$. 
\end{rem}


To prove Theorem \textbf{\ref{eqn t7.4}}, we need several steps. For a finite dimensional subspace $V$ of 
$\mathcal{H}^{\lambda}$, let $\mathcal{O}(V)$ stand for the vector space of all 
bounded linear operators $R:V\longrightarrow \mathcal{H}^{\lambda}$. Define an inner 
product $\langle,\rangle^{\lambda}$ on $\mathcal{O}(V)$ by $$\langle 
R,R^{\prime}\rangle^{\lambda}=\sum_{j=1}^{d}\langle Rv_{j},R^{\prime}v_{j}\rangle,$$ 
for an orthonormal basis $\{v_{1}, v_{2}, \cdots v_{d}\}$. Clearly the definition is 
independent of the orthonormal basis. One can see that the norm defined by the above 
inner product is equivalent to the operator norm. Since $\mathcal{O}(V)$ is a Banach 
space with respect to the operator norm, we conclude that $\mathcal{O}(V)$ is a 
Hilbert space with the above inner product. If $V\subset \mathcal{P}(\mathbb{C}^{n})$ 
we can view $\mathcal{O}(\mathcal{H}^{\lambda})$ as a subset of $\mathcal{O}(V)$ by 
restricting the elements in $\mathcal{O}(\mathcal{H}^{\lambda})$ to $V$. If $V=V_{\alpha}$, we denote the inner product by $\langle,\rangle_{\alpha}^{\lambda}$. 
In this case, by Schur's orthogonality relation we have another formula for $\langle,
\rangle_{\alpha}^{\lambda}$, given by $$\langle 
R,R^{\prime}\rangle_{\alpha}^{\lambda}=\int_{K}\langle R(k\cdot v), R^{\prime}(k\cdot 
v)\rangle dk~,$$ for any unit vector $v\in V_{\alpha}$.

\begin{lem}\label{eqn l7.5}
\begin{itemize}
\item [\textbf{(a)}] Suppose $f\in \mathscr{S}(\mathbb{C}^{n})$, and 
$\delta\in\hat{K}_{M}$. Then $$\big\langle\mathcal{G}^{\lambda}(f),
\mathcal{W}^{\lambda}(P^{\delta})\big\rangle_{\alpha}^{\lambda}=\pi^{n}(2|\lambda|)^{-
n}\langle f,\Psi_{\alpha}^{\delta,\lambda}\rangle.$$ Here the equality is entry wise.
\item [\textbf{(b)}] For two $K$-harmonic polynomials $p,q$ 
$$\langle\mathcal{W}^{\lambda}(p),\mathcal{W}^{\lambda}
(q)\rangle_{\alpha}^{\lambda}=\pi^{n}(2|\lambda|)^{-n}\langle\theta^{\lambda}
(p)\psi^{\lambda}_{\alpha},\theta^{\lambda}(q)
\psi^{\lambda}_{\alpha}\rangle.$$ 
\end{itemize}
\end{lem}

\begin{proof}
The proof is similar to that of Lemma 2.1 in \cite{T}.
\beas\big\langle\mathcal{G}(f),\mathcal{W}(P^{\delta})\big\rangle_{\alpha}&=
&\sum_{\nu=1}^{d(\alpha)}\big\langle\mathcal{G}(f)e_{\alpha\nu},\mathcal{W}
(P^{\delta})e_{\alpha\nu}\big\rangle
\\&=&\sum_{\beta\in\Lambda}\sum_{\nu=1}^{d(\beta)}\big\langle\mathcal{G}
(f)e_{\beta\nu},\mathcal{W}
(P^{\delta})\mathcal{P}_{\alpha}e_{\beta\nu}\big\rangle\\&=&\sum_{\beta\in\Lambda}
\sum_{\nu=1}^{d(\beta)}\big\langle\mathcal{G}(f)e_{\beta\nu},\mathcal{G}
\big(\theta(P^{\delta})\big)\mathcal{G}(\psi_{\alpha})e_{\beta\nu}\big\rangle.\eeas By 
Corollary \textbf{\ref{eqn c5.4}}, $$\mathcal{G}
\big(\theta(P^{\delta})\big)\mathcal{G}
(\psi_{\alpha})=\mathcal{G}\big(\theta(P^{\delta})\psi_{\alpha}\big)=
\mathcal{G}(\Psi_{\alpha}^{\delta}).$$ Hence by the Plancherel 
Theorem (Theorem \textbf{\ref{eqn t5.1}}), \textbf{(a)} follows. The proof of \textbf{(b)} is similar.
\end{proof}

\begin{lem}\label{eqn l7.6}
$\{\mathcal{W}^{\lambda}(p):p\in\mathcal{P}(\mathbb{C}_{\mathbb{R}}^{n})\}$ is dense 
in $\mathcal{O}(V_{\alpha})$.
\end{lem}

\begin{proof}
It is shown in \cite{G} (See Proposition 2.10 (b)) that $\{\mathcal{W}
(p):p\in\mathcal{P}(\mathbb{C}_{\mathbb{R}}^{n})\}$ is dense in $\mathcal{O}
(\mathcal{P}_{m}(\mathbb{C}^{n})).$ Since $V_{\alpha}$ is contained in some 
$\mathcal{P}_{m}(\mathbb{C}^{n})$, we can extend any operator $T\in\mathcal{O}(V_{\alpha})$ to $T^{\prime}\in\mathcal{O}(\mathcal{P}_m(\C))$ by defining $T^{\prime}$ to be zero on the complement of $V_{\alpha}$ in $\mathcal{P}_m(\C)$. From this, it is easy to see that, $\{\mathcal{W}
(p):p\in\mathcal{P}(\mathbb{C}_{\mathbb{R}}^{n})\}$ is dense in $\mathcal{O}
(V_{\alpha}).$
\end{proof}

Let $H_{\stackrel{\vee}{\delta} }^{i}$ be the subspace of $H_{\stackrel{\vee}{\delta}}$, spanned by the 
entries of $i$th row in $P^{\delta}$. Using Schur's orthogonality it can be shown 
that, if $p\in H_{\stackrel{\vee}{\delta} }^{i}$, $q\in H_{\stackrel{\vee}{\delta^{\prime}} 
}^{i^{\prime}}$ with $(\delta,i)\neq(\delta^{\prime},i^{\prime})$, then 
\bea\label{eqn64}\int_{K}p(k\cdot z)\overline{q(k\cdot z)}dk=0.\eea

\begin{prop}\label{eqn p7.7}
$\mathcal{O}
(V_{\alpha})=\bigoplus_{\delta\in\widehat{K}_{M}}\bigoplus_{i=1}^{d(\delta)}\mathcal{W}
^{\lambda}(H_{\stackrel{\vee}{\delta} }^{i})\big|_{V_{\alpha}}$~(orthogonal~ Hilbert~ space~ 
decomposition).
\end{prop}

\begin{proof}
In Lemma \textbf{\ref{eqn l7.5}} \textbf{(b)}, if we take $p=P_{ij}^{\delta}$ and 
$q=P_{i^{\prime}j^{\prime}}^{\delta^{\prime}}$, then $\theta(p)\psi_{\alpha}=(i,j)$th 
entry of $\theta(P^{\delta})\psi_{\alpha}$ = $(i,j)$th entry of 
$\Psi_{\alpha}^{\delta}$, and similarly 
$\theta(q)\psi_{\alpha}=(i^{\prime},j^{\prime})$th entry of 
$\Psi_{\alpha}^{\delta^{\prime}}$. Note that $\Psi_{\alpha}^{\delta}$ has the form 
(\ref{eqn61}). Therefore if $(\delta,i)\neq(\delta^{\prime},i^{\prime})$, then by 
(\ref{eqn64}) we have, $$\int_{K}[\theta(p)\psi_{\alpha}](k\cdot 
z)\overline{[\theta(q)\psi_{\alpha}](k\cdot z)}dk=0~\forall z.$$ Integrating both 
sides over $\C$, and then making a change of variable, namely $z\rightarrow  
k^{-1}\cdot z$, we get $\langle\theta(p)\psi_{\alpha},\theta(q)\psi_{\alpha}\rangle=0$ 
if $(\delta,i)\neq(\delta^{\prime},i^{\prime}).$ Hence the orthogonality is proved.

By Lemma \textbf{\ref{eqn l7.6}}, $\{\mathcal{W}(p):p\in\mathcal{P}(\mathbb{C}_{\mathbb{R}}^{n})\}$ 
is dense in $\mathcal{O}(V_{\alpha})$. Therefore to complete the proof of the 
proposition it is enough to show that for any $\delta\in\widehat{K}_{M}$, 
\bea\label{eqn65}\mathcal{W}(IH_{\delta})\big|_{V_{\alpha}}=\mathcal{W}
(H_{\delta})\big|_{V_{\alpha}}.\eea For a polynomial $p$, let $p(D)$ denote the 
constant coefficient differential operator obtained by replacing $z_{j}$ by 
$-\partial/\partial\bar{z}_{j}$ and $\bar{z}_{j}$ by $\partial/\partial z_{j}$. Then 
it is an easy consequence of the Euclidean Fourier transform that 
$\mathcal{F}^{\prime-1}p=p(D)\delta_{0}$. Also, we can write 
$$\theta(p)=p(D)+\varepsilon(p),$$ where $\varepsilon(p)$ is a polynomial coefficient 
differential operator of order strictly less than the degree of $p$. Let $P_{m}$ be 
the space of all polynomials in $z$, $\bar{z}$ whose degree is less than or equal to 
$m$. We prove (\ref{eqn65}) by showing that \bea\label{eqn66}\mathcal{W}
(IH_{\delta}\cap P_{m} )\big|_{V_{\alpha}}\subset\mathcal{W}(H_{\delta})\big|
_{V_{\alpha}},\eea for all non negative integers $m$. We do it by induction on $m$. 
Since $\mathcal{W}(1)=\mathcal{G}(\delta_{0})=$ identity operator, (\ref{eqn66}) is 
true for $m=0$. Now suppose (\ref{eqn66}) is true for $m=k$. It is enough to show that 
$\mathcal{W}(p)\big|_{V_{\alpha}}\in \mathcal{W}(H_{\delta})\big|_{V_{\alpha}},$ for 
any polynomial $p$ of the type $p=jh$, where $j\in I$, $h\in H_{\delta}$ and degree 
$p=(k+1)$. \bea\label{eqn67} \mathcal{W}(p)=\nonumber\mathcal{G}
(\mathcal{F}^{\prime-1}p)&=&\mathcal{G}\big(h(D)j(D)\delta_{0}\big)\\&=&\mathcal{G}
\big(\theta(h)j(D)\delta_{0}\big)-\mathcal{G}\big(\varepsilon(h)j(D)\delta_{0}\big)\\
&=&\nonumber
\mathcal{W}(h)\mathcal{W}
(j)-\mathcal{W}\big(\mathcal{F}^{\prime}\big(\varepsilon(h)j(D)\delta_{0}\big)\big).
\eea By Corollary \textbf{\ref{eqn c6.2}}, $\mathcal{W}(j)$ is a scalar on 
$V_{\alpha}$. Hence $$[\mathcal{W}(h)\mathcal{W}(j)]\big|_{V_{\alpha}}\in\mathcal{W}
(H_{\delta})\big|_{V_{\alpha}}.$$ Now note that $\varepsilon(h)j(D)\delta_{0}$ is a 
distribution supported at the origin, whose order is less than or equal to $k$. 
Therefore $\mathcal{F}^{\prime}\big(\varepsilon(h)j(D)\delta_{0}\big)$ is a polynomial 
of degree at most $k$. Again from (\ref{eqn67}) we have $$\mathcal{F}^{\prime-1}p=
\theta(h)j(D)\delta_{0}-\varepsilon(h)j(D)\delta_{0},$$ which implies that 
$$\mathcal{F}^{\prime}\big(\varepsilon(h)j(D)\delta_{0}\big)\in 
S_{\delta}=IH_{\delta}.$$ Therefore, by the induction hypothesis, $$
\mathcal{W}\big(\mathcal{F}^{\prime}\big(\varepsilon(h)j(D)\delta_{0}\big)\big)\big|
_{V_{\alpha}}
\in\mathcal{W}(H_{\delta})\big|_{V_{\alpha}}.$$ So ultimately we get that 
$$\mathcal{W}(p)\big|_{V_{\alpha}}\in \mathcal{W}(H_{\delta})\big|_{V_{\alpha}},$$ as 
desired. Hence the proof of the proposition is complete.
\end{proof}

\begin{proof}
$(Proof~ of~ Theorem$ \textbf{\ref{eqn t7.4}})~ By Lemma \textbf{\ref{eqn l7.5}} 
\textbf{(a)},  $$\big\langle\mathcal{G}(f),\mathcal{W}
(P^{\delta^{\prime}})\big\rangle_{\alpha}=\pi^{n}(2|\lambda|)^{-n}\langle f,
\Psi_{\alpha}^{\delta^{\prime}}\rangle,$$ for $f\in \mathscr{S}(\mathbb{C}^{n})$ and 
$\delta^{\prime}\in\widehat{K}_{M}$.
Again, $\Psi_{\alpha}^{\delta^{\prime}}$ has the form (\ref{eqn61}). Therefore if we 
take $$f(z)=F_{ij}(z)=\sum_{k=1}^{l(\delta)}P^{\delta}_{ik}(z)G_{kj}(z),$$ then by 
(\ref{eqn64}), we get $$\big\langle\mathcal{G}(F_{ij}),\mathcal{W}
(P_{i^{\prime}j^{\prime}}^{\delta^{\prime}})\big\rangle_{\alpha}=0,$$ if 
$(\delta^{\prime},i^{\prime})\neq(\delta,i).$ By Proposition \textbf{\ref{eqn p7.7}}, 
in particular for $i=1$, we get $$\mathcal{G}(F_{1j})\big|_{V_{\alpha}}\in\mathcal{W}
(H_{\stackrel{\vee}{\delta}}^{1})\big|_{V_{\alpha}},$$ for all $j=1,2,\cdots,l(\delta)$; and 
consequently there are constants $c_{kj}$ such that \bea\label{eqn68}\mathcal{G}
(F_{1j})\big|_{V_{\alpha}}=\sum_{k=1}^{l(\delta)}c_{kj}
\mathcal{W}(P^{\delta}_{1k})\big|_{V_{\alpha}},\eea which implies \beas 
F_{1j}\times\psi_{\alpha}&=&\sum_{k=1}^{l(\delta)}c_{kj}\theta(P^{\delta}_{1k})\psi_{\alpha}\\&=& 
(1,j)\textup{th~entry ~of~}\big[\theta(P^{\delta})\psi_{\alpha}\big]C_{\alpha}\\&=& 
(1,j)\textup{th~entry ~of~} \Psi_{\alpha}^{\delta}C_{\alpha},\eeas
where $C_{\alpha}$ is the $l(\delta)\times l(\delta)$ constant matrix whose $(k,j)$th 
entry is $c_{kj}$. Therefore by Corollary \textbf{\ref{eqn c4.3}}, 
$F\times\psi_{\alpha}=\Psi_{\alpha}^{\delta}C_{\alpha}.$ This is true for all 
$\alpha\in\Lambda$. Hence we get $\mathcal{G}(F)=\mathcal{W}(P^{\delta})S$, if we 
define the $l(\delta)\times l(\delta)$ linear operator $S$ by $S\big|
_{V_{\alpha}}=C_{\alpha}$.

Since $\Pi(z)^{\star}=\Pi(-z)$, a direct calculation shows that $\mathcal{G}
(\bar{f})=\mathcal{G}(f^{-})^{\star}$, for all $f\in\mathscr{S}(\mathbb{C}^{n})$, 
where $f^{-}(z)=f(-z)$. Also, note that $(L_{j}f)^{-}=(-L_{j})f^{-}$ and 
$(\overline{L}_{j}f)^{-}=(-\bar{L}_{j})f^{-}$. Therefore if $p,q$ be two $K$-harmonic 
polynomials then 
\beas\mathcal{G}\big(\theta(p)\psi_{\alpha}\times\overline{\theta(q)\psi_{\beta}}\big)&=&
\mathcal{G}\big(\theta(p)\psi_{\alpha}\big)\mathcal{G}\big(\theta(q^{-})
\psi_{\beta}^{-}\big)^{\star}\\&=& \mathcal{W}(p)\mathcal{G}
(\psi_{\alpha})\big[\mathcal{W}(q^{-})\mathcal{G}(\psi_{\beta}^{-})\big]^{\star} 
\\&=&\tau(p)\mathcal{G}(\psi_{\alpha})\mathcal{G}
(\psi_{\beta})\tau(q^{-})^{\star}\\&=& 0,\eeas if $\beta\neq \alpha$. The last two 
equality holds on the domain of $\tau(q^{-})^{\star}$. Since $W_{j}$ and $\bar{W}_{j}$ 
are adjoint to each other we see that $\tau(q^{-})^{\star}=\tau(\bar{q}^{-})$ whose 
domain contains $\mathcal{P}(\mathbb{C}^{n})$. Hence we get 
$\theta(p)\psi_{\alpha}\times\overline{\theta(q)\psi_{\beta}}=0$. In particular 
$\int_{\mathbb{C}^{n}}\overline{\theta(q)\psi_{\beta}(z)}\theta(p)\psi_{\alpha}
(z)dz=0$ if $\beta\neq\alpha$. Applying this, we have for $\beta\neq\alpha$, 
\bea\label{eqn69}\int_{\mathbb{C}^{n}}\big[\Psi_{\beta}^{\delta}
(z)\big]^{\star}\Psi_{\alpha}^{\delta}(z)dz=0.\eea Since (by Proposition 
\textbf{\ref{eqn p6.4}}) 
$$F=\sum_{\beta\in\Lambda}F\times\psi_{\beta}=\sum_{\beta\in\Lambda}\Psi_{\beta}^{\delta}C_{\beta},
$$ we get \bea\label{eqn610} A_{\alpha}^{\delta}C_{\alpha}&=& 
\nonumber\int_{\mathbb{C}^{n}}[\Psi_{\alpha}^{\delta}(z)]^{\star}F(z)dz\\&=& 
\int_{\mathbb{C}^{n}}[L^{\delta}_{\alpha}(z)]^{\star}
\Upsilon_{\delta}(z)G(z)e^{-|\lambda||z|^{2}}dz.\eea
Now we show that it is possible to chose $C_{\alpha}=C_{\alpha}^{\delta,\lambda}(F)$, as defined in (\ref{eqn63}).
Without loss of generality assume that $l(r)=r$, i.e. first $\alpha(\delta)$ 
columns in $\Psi_{\alpha}^{\delta}$ are linearly independent. By Remark 
\textbf{\ref{eqn r4.4}}, $\{\Psi^{\delta}_{\alpha 1j}=\theta 
\big(P^{\delta}_{1j}\big)\psi_{\alpha}:j=1,2,\cdots \alpha(\delta)\}$ form a basis for 
$V^{1}_{\Psi_{\alpha}^{\delta }}=\textup{span}\{\Psi^{\delta}_{\alpha 1j}=\theta 
\big(P^{\delta}_{1j}\big)\psi_{\alpha}:j=1,2,\cdots l(\delta)\}$; which is equivalent 
to saying that $\{\mathcal{W}(P^{\delta}_{1j})\big|_{V_{\alpha}}:j=1,2,\cdots 
\alpha(\delta)\}$ form a basis for $\mathcal{W}(H_{\stackrel{\vee}{\delta}}^{1})\big|
_{V_{\alpha}}$. Therefore in (\ref{eqn68}) we can take $c_{kj}=0$ for $k> 
\alpha(\delta)$. Consequently from (\ref{eqn610}) 
we get $$\widetilde{A}_{\alpha}^{\delta}\widetilde{C}_{\alpha}= \int_{\mathbb{C}^{n}}
[\widetilde{L}^{\delta}_{\alpha}(z)]^{\star}
\Upsilon_{\delta}(z)G(z)e^{-|\lambda||z|^{2}}dz,$$ where $\widetilde{C}_{\alpha}$ denotes the $\alpha(\delta)\times l(\delta)$ matrix whose rows are precisely the first $\alpha(\delta)$ rows of $C_{\alpha}$.  But by Lemma \textbf{\ref{eqn 
l7.3}}, $\widetilde{A}^{\delta,\lambda}_{\alpha}$ is invertible. Therefore we can write $$\widetilde{C}_{\alpha}=\big(\widetilde{A}_{\alpha}^{\delta}\big)^{-1} \int_{\mathbb{C}^{n}}
[\widetilde{L}^{\delta}_{\alpha}(z)]^{\star}
\Upsilon_{\delta}(z)G(z)e^{-|\lambda||z|^{2}}dz.$$ Hence by the definition of $C^{\delta,\lambda}_{\alpha}(F)$, $C_{\alpha}=C^{\delta,\lambda}_{\alpha}(F)$ as desired.
\end{proof}

Now we extend Theorem \textbf{\ref{eqn t7.4}} to a larger class of 
functions. Let $$\mathcal{E}^{\lambda}(\C)=\{f\in\mathcal{E}(\C):e^{-(|
\lambda|-\epsilon)|z|^{2}}|f(z)|\in L^{p}(\C),~\textup{for~some~$\epsilon>0, 1\leq 
p\leq\infty$}\},$$ and for $\delta\in\hat{K}_{M}$, $$\mathcal{E}^{\delta,\lambda}
(\C)=\{F\in\mathcal{E}^{\delta}(\C):\textup{each~}F_{ij}\in\mathcal{E}^{\lambda}
(\C)\}.$$ Since $\psi_{\alpha}^{\lambda}(z)$ is equal to $e^{-|\lambda||z|^{2}}$ times 
a polynomial, clearly(by Holder's inequality) 
$$f\times^{\lambda}\psi^{\lambda}_{\alpha}(z)=\int_{\C}f(z-w)\psi^{\lambda}_{\alpha}
(w)e^{2i\lambda\textup{Im}(z\cdot \bar{w})}dw$$ is well defined, whenever 
$f\in\mathcal{E}^{\lambda}(\C).$ For $\epsilon>0$ and $z\in\C$, define 
$$\tau^{\epsilon}_{z}\psi^{\lambda}_{\alpha}(w)=e^{(|\lambda|-\epsilon)|w|^{2}}
[\psi^{\lambda}_{\alpha}(z-w)e^{-2i\lambda \textup{Im}(z.\bar{w})}],$$ which clearly 
belongs to $\mathscr{S}(\C)$. Note that if $f\in\mathcal{E}^{\lambda}(\C)$, then for 
some $\epsilon>0$, we can think of $e^{-(|\lambda|-\epsilon)|z|^{2}}f(z)$ as a tempered 
distribution and then clearly $$f\times^{\lambda}\psi_{\alpha}^{\lambda}(z)=e^{-(|
\lambda|-\epsilon)|\cdot|^{2}}f\big(\tau^{\epsilon}_{z}\psi^{\lambda}_{\alpha}\big).$$

\begin{lem}\label{eqn l7.8}
Let $f$ be a distribution on $\C$, such that for some $\epsilon>0$, $e^{-(|
\lambda|-\epsilon)|\cdot|^{2}}f$ is a tempered distribution. Let $D$ be a polynomial 
coefficient differential operator on $\C$. Then  
\begin{itemize}
\item [\textbf{(a)}] $e^{-(|\lambda|-\epsilon)|\cdot|^{2}}Df$ is also a tempered 
distribution.
\item [\textbf{(b)}] Let $f_{j}\in\mathscr{S}(\C)$ be such that $e^{-(|
\lambda|-\epsilon)|\cdot|^{2}}f_{j}\rightarrow e^{-(|\lambda|-\epsilon)|\cdot|^{2}}f$ 
in $\mathscr{S}^{\prime}(\C)$. Then $e^{-(|\lambda|-\epsilon)|\cdot|
^{2}}Df_{j}\rightarrow e^{-(|\lambda|-\epsilon)|\cdot|^{2}}Df$ in 
$\mathscr{S}^{\prime}(\C)$. Consequently for each $z\in\C$, 
$Df_{j}\times^{\lambda}\psi_{\alpha}^{\lambda}\rightarrow e^{-(|\lambda|-\epsilon)|
\cdot|^{2}}Df\big(\tau^{\epsilon}_{z}\psi_{\alpha}^{\lambda}\big)$.
\item [\textbf{(c)}] In particular, if $f\in\mathcal{E}^{\lambda}(\C)$, 
$f_{j}\times^{\lambda}\psi^{\lambda}_{\alpha}(z)\rightarrow 
f\times^{\lambda}\psi^{\lambda}_{\alpha}(z)$ for each $z\in\C$.
\end{itemize} 
\end{lem}

\begin{proof}
When the action of $D$ on $f$ is multiplication by a polynomial, clearly 
$\textbf{(a)}$ and $\textbf{(b)}$ are true. Note that $$e^{-(|\lambda|-\epsilon)|z|
^{2}}\frac{\partial f}{\partial z_{j}}=\frac{\partial}{\partial z_j}\big(e^{-(|
\lambda|-\epsilon)|z|^{2}}f\big)+(|\lambda|-\epsilon)\bar{z_j}f,$$ which immediately 
proves $\textbf{(a)}$, as well as $\textbf{(b)}$ when $D=\partial/\partial z_{j}.$ General case follows by an induction argument. Assertion \textbf{(c)} is immediate from \textbf{(b)}.  
\end{proof}

\begin{thm}\label{eqn t7.9}
Suppose $F=P^{\delta}G\in\mathcal{E}^{\delta,\lambda}(\C)$, $G$ is $K$-invariant. Then 
$F\times^{\lambda}\psi_{\alpha}^{\lambda}=\Psi^{\delta,\lambda}_{\alpha}C^{\delta,\lambda}_{\alpha}(F)$, 
where $C^{\delta,\lambda}_{\alpha}(F)$ is defined by (\ref{eqn63}). 
\end{thm}

\begin{proof} 
Each entry of $F$ belongs to $\mathcal{E}^{\lambda}
(\C)$. Take 
$F_{j}\in\mathscr{S}^{\delta}(\C)$ such that $e^{-(|\lambda|-\epsilon)|.|
^{2}}F_{j}\rightarrow e^{-(|\lambda|-\epsilon)|.|^{2}}F$ entry wise in 
$\mathscr{S}^{\prime}(\C)$. For each $F_{j}$ we can apply Theorem \textbf{\ref{eqn 
t7.4}} to get \bea\label{eqn e1}F_{j}\times\psi_{\alpha}(z)=\Psi_{\alpha}^{\delta}
(z)C^{\delta,\lambda}_{\alpha}(F_j),\eea where $C^{\delta,\lambda}_{\alpha}(F_j)$ is defined by equation (\ref{eqn63}). A similar argument used in the proof of the previous lemma shows that $\lim_{j\rightarrow\infty}C^{\delta,\lambda}_{\alpha}
(F_{j})=C^{\delta,\lambda}_{\alpha}(F)$. On the other hand,  by $\textbf{(c)}$ of the previous lemma, for each $z\in\C$, 
$$\lim_{j\rightarrow\infty}[F_{j}\times\psi_{\alpha}(z)]=F\times\psi_{\alpha}(z).$$ 
Hence for each $z\in\C$, taking limit, as 
$j\rightarrow\infty$, in (\ref{eqn e1}), the proof follows.
\end{proof}

Now we proceed to prove the uniqueness (upto right multiplication by a constant 
matrix) of generalized $K$-spherical function when it belongs to $\mathcal{E}^{\delta,
\lambda}(\C).$ 

\begin{lem}\label{eqn l7.10}
Let $f\in\mathcal{E}^{\lambda}(\C)$ be a joint eigenfunction for all 
$D\in\mathcal{L}_{k}^{\lambda}(\C)$ with eigenvalue $\mu_{\alpha}^{\lambda}.$ Then 
$f\times^{\lambda}\psi_{\beta}^{\lambda}=0$ if $\beta\neq\alpha$. 
\end{lem}

\begin{proof}
By definition of $\mathcal{E}^{\delta}(\C)$, $e^{-(|\lambda|-\epsilon)|.|^{2}}f$ is a 
tempered distribution for some $\epsilon>0$. Take $f_{j}\in\mathscr{S}(\C)$ such that 
$e^{-(|\lambda|-\epsilon)|.|^{2}}f_{j}\rightarrow e^{-(|\lambda|-\epsilon)|.|^{2}}f$ 
in $\mathscr{S}^{\prime}(\C)$. Let $D\in\mathcal{L}_{K}^{\lambda}(\C)$. Since 
$Df=\mu_{\alpha}(D)f$, by Lemma \textbf{\ref{eqn l7.8}} \textbf{(b)}, 
\bea\label{eqne3}\lim_{j\rightarrow\infty}Df_{j}\times\psi_{\beta}(z)=e^{-(|
\lambda|-\epsilon)|z|^{2}}\mu_{\alpha}
(D)f(\tau_{z}^{\epsilon}\psi_{\beta})=\mu_{\alpha}(D)f\times\psi_{\beta}.\eea Again by 
Lemma \textbf{\ref{eqn l7.8}} \textbf{(c)}, 
\bea\label{eqne4}\lim_{j\rightarrow\infty}f_{j}\times\psi_{\beta}
(z)=f\times\psi_{\beta}(z),~\forall z\in\C.\eea Now we will show that 
$Df_{j}\times\psi_{\beta}=\mu_{\beta}(D)f_{j}\times\psi_{\beta} $ for all $j$. By 
Proposition \textbf{\ref{eqn p5.3}} it follows that range$[\mathcal{G}
(D)\mathcal{P}_{\beta}]\subset \textbf{P}_{N}$ for some natural number $N$. Here 
$\textbf{P}_{N}$ denotes the space of all holomorphic polynomials of degree less than 
or equal to $N$. Hence $\mathcal{G}(D)\mathcal{P}_{\beta} 
=\bigg(\sum_{V_{\gamma}\subset\textbf{P}_{N}}
\mathcal{P}_{\gamma}\bigg)\mathcal{G}(D)\mathcal{P}_{\beta}$. Enlarging 
$\textbf{P}_{N}$ if necessary we may assume that $V_{\beta}\subset\textbf{P}_{N}$. Therefore we have $$\mathcal{G}
(Df_{j}\times\psi_{\beta})=\mathcal{G}(Df_{j})\mathcal{P}_{\beta}=\mathcal{G}
(f_{j})\mathcal{G}(D)
\mathcal{P}_{\beta}=\mathcal{G}(f_{j})\bigg(\sum_{V_{\gamma}\subset\textbf{P}_{N}}
\mathcal{P}_{\gamma}\bigg)\mathcal{G}(D)\mathcal{P}_{\beta}.$$  But, 
$$\mathcal{P}_{\gamma}\mathcal{G}(D)=\mathcal{G}(\psi_{\gamma})\mathcal{G}
(D)=\mathcal{G}(D\psi_{\gamma})
=\mu_{\gamma}(D)\mathcal{P}_{\gamma}.$$ Hence $$\mathcal{G}
(Df_{j}\times\psi_{\beta})=\mu_{\beta}(D)\mathcal{G}(f_{j})\mathcal{P}_{\beta}=
\mu_{\beta}(D)\mathcal{G}(f_{j}\times\psi_{\beta}).$$ Therefore 
$Df_{j}\times\psi_{\beta}(z)=\mu_{\beta}(D)f_{j}\times\psi_{\beta}(z)$ for all 
$z\in\C$. Now taking limit as $j\rightarrow\infty$ and using (\ref{eqne3}), 
(\ref{eqne4}) we get $\mu_{\alpha}(D)f\times\psi_{\beta}(z)=\mu_{\beta}
(D)f\times\psi_{\beta}(z).$ This is true for all $D\in\mathcal{L}^{\lambda}_{K}(\C)$. 
Since $\mu_{\beta}\neq\mu_{\alpha}$ for $\beta\neq\alpha$, we get $f\times\psi_{\beta}
(z)=0$ if $\beta\neq\alpha$. Hence the proof. 
\end{proof}


\begin{lem}\label{eqn l7.11}
Let $\mathcal{L}^{\lambda}$ be the special Hermite operator and $\psi^{\lambda}_k$'s be the $U(n)$- spherical functions (see Remark \textbf{\ref{eqn r5.10}}).
Let $f\in\mathcal{E}^{\lambda}(\C)$ be an eigenfunction of $\mathcal{L}^{\lambda}$ with 
eigenvalue $-2|\lambda|(2k+n)$. Then $f=f\times^{\lambda}\psi_{k}^{\lambda}.$ 
\end{lem}

\begin{proof}
For this proof let $K=U(n)$ and $M$ be the subgroup of $U(n)$ that fixes the 
coordinate vector $e_{1}=(1,0,\cdots,0)$ in $\C$. For $\delta\in\widehat{K}_{M}$, let 
$\chi_{\delta}(k)=tr(\delta(k))$. Define $f_{\delta}(z)=\int_{K}f(k^{-1}\cdot 
z)\chi_{\delta}(k)dk$.  Clearly each $f_{\delta}$ is an eigenfunction of $\mathcal{L}$ with 
eigenvalue $-2|\lambda|(2k+n)$. Applying the previous lemma for $K=U(n)$ to each 
$f_{\delta}$ we get $f_{\delta}\times\psi_{m}=0$ if $m\neq k$. Again Proposition 
4.5 \cite{T}, in particular, implies that each $f_{\delta}\in\mathscr{S}(\C)$. Hence by 
Proposition \textbf{\ref{eqn p6.4}}, 
$f_{\delta}=\Sigma_{m\in\mathbb{N}}f_{\delta}\times\psi_{m}.$ Consequently we get 
$f_{\delta}=f_{\delta}\times\psi_{k}$ for all $\delta\in\widehat{K}_{M}$. Since $\psi_{k}$ 
is radial, an easy calculation shows that 
$(f\times\psi_{k})_{\delta}=f_{\delta}\times\psi_{k}.$ Therefore 
$f_{\delta}=(f\times\psi_{k})_{\delta}$ for all $\delta\in\widehat{K}_{M}$. But for any 
smooth function $g$ it is well known that 
$g(z)=\Sigma_{\delta\in\hat{K}_{M}}g_{\delta}(z)$, where the right hand side converges 
uniformly over compact set. Hence we conclude that $f=f\times\psi_{k}.$
\end{proof}


\begin{prop}\label{eqn p7.12}
Let $f\in\mathcal{E}^{\lambda}(\C)$ be a joint eigenfunction for all 
$D\in\mathcal{L}_{K}^{\lambda}(\C)$ with eigenvalue $\mu_{\alpha}^{\lambda}.$ Then 
$f=f\times^{\lambda}\psi_{\alpha}^{\lambda}.$ 
\end{prop}

\begin{proof}
$V_{\alpha}\subset\mathcal{P}_k(\C)$ for some $k\in\mathbb{N}.$ Then $f$ is an eigenfunction of $L$ with eigenvalue $-(2k+n)|\lambda|$. Therefore by Lemma 
\textbf{\ref{eqn l7.11}}, $f=f\times\psi_k$. Since 
$\psi_{k}=\Sigma_{V_{\beta}\subset\mathcal{P}_{k}(\C)}\psi_{\beta},$ we get 
$f=\sum_{V_{\beta}\subset\mathcal{P}_{k}(\C)}f\times \psi_{\beta}$. But then by Lemma 
\textbf{\ref{eqn l7.10}}, $f=f\times\psi_{\alpha}.$ 
\end{proof}

As an immediate consequence of Theorem \textbf{\ref{eqn t7.9}} and Proposition 
\textbf{\ref{eqn p7.12}} we get the following Theorem.

\begin{thm}\label{eqn t7.13}
If $\Psi\in\mathcal{E}^{\delta,\lambda}(\C)$ is is a generalized $K$-spherical 
function of type $\delta$ corresponding to the eigenvalue $\mu_{\alpha}^{\lambda}$, 
then $\Psi=\Psi_{\alpha}^{\delta,\lambda}C$, where $C=C^{\delta,\lambda}_{\alpha}(\Psi)$ as defined by (\ref{eqn63}). 
\end{thm}

We conclude this section by giving another formulae for $\Psi_{\alpha}^{\delta,
\lambda}$ which will be used in the next section. Define $$\Phi_{\alpha}^{\delta,
\lambda}(z)=\big\langle \Pi^{\lambda}
(z),\mathcal{W}^{\lambda}(P^{\delta})\big\rangle^{\lambda}_{\alpha}=\sum_{\nu=1}^{d(\alpha)}\big\langle 
\Pi^{\lambda}
(z)e_{\alpha\nu}^{\lambda},\mathcal{W}^{\lambda}(P^{\delta})e_{\alpha\nu}^{\lambda} \big\rangle$$

\begin{prop}\label{eqn p7.14}
$\Psi_{\alpha}^{\delta,\lambda}=\pi^{-n}(2|\lambda|)^{n}\Phi_{\alpha}^{\delta,
\lambda}.$   Consequently $\theta^{\lambda}(p)\psi_{\alpha}^{\lambda}=\big\langle\Pi^{\lambda}(z),\mathcal{W}^{\lambda}(p)\big\rangle^{\lambda}_{\alpha}$ whenever $p\in H_{\stackrel{\vee}{\delta}}$. 
\end{prop}

\begin{proof}
Note that, on the one hand a direct calculation shows $$\big\langle\mathcal{G}^{\lambda}
(f),\mathcal{W}^{\lambda}(P^{\delta})\big\rangle_{\alpha}^{\lambda}=\langle f,
\Phi_{\alpha}^{\delta,\lambda}\rangle,$$ and on the other hand, by Lemma 
\textbf{\ref{eqn l7.5}} \textbf{(a)}, we have $$\big\langle\mathcal{G}^{\lambda}(f),
\mathcal{W}^{\lambda}(P^{\delta})\big\rangle_{\alpha}^{\lambda}=\pi^{n}(2|\lambda|)^{-
n}\langle f,\Psi_{\alpha}^{\delta,\lambda}\rangle,$$ for all $f\in\mathscr{S}
(\mathbb{C}^{n})$. Hence $\Psi_{\alpha}^{\delta,\lambda}=\pi^{-n}(2|
\lambda|)^{n}\Phi_{\alpha}^{\delta,\lambda}.$ This also can be proved directly 
using the inversion formulae for Weyl transform.
\end{proof}

\section{$K$-finite eigenfunctions}

Following the view point of Thangavelu in \cite{T} (see Theorem 3.3 there), we obtain a representation for $K$-finite joint eigenfunctions in $\mathcal{E}^{\lambda}(\C)$.  

\begin{thm}\label{eqn t8.1}
Let $f\in\mathcal{E}^{\lambda}(\C)$ be a $K$-finite joint eigenfunction for all 
$D\in\mathcal{L}_{K}^{\lambda}(\C)$ with eigenvalue $\mu_{\alpha}^{\lambda}$. Then 
$f(z)=\big\langle\Pi^{\lambda}
(z),\mathcal{W}^{\lambda}(P)\big\rangle_{\alpha}^{\lambda}$ for some $K$-harmonic polynomial $P$.  
\end{thm}

To prove the theorem we first prove the following lemma, which is an easy consequence 
of Theorem \textbf{\ref{eqn t7.13}} and Proposition \textbf{\ref{eqn p7.14}}.

\begin{lem}\label{eqn l8.2}
Suppose $F:\mathbb{C}^{n}\longrightarrow\mathcal{M}_{d(\delta)\times d(\delta)}$ is a 
smooth, square integrable joint eigenfunction for all $D\in\mathcal{L}^{\lambda}_{K}
(\mathbb{C}^{n})$ with eigenvalue $\mu^{\lambda}_{\alpha}$. Also assume $F(k\cdot 
z)=\delta(k)F(z)$, for some $\delta\in\widehat{K}_{M}$. Then, there exists a 
$l(\delta)\times d(\delta)$ constant matrix $C$ such that $F=\Phi_{\alpha}^{\delta,
\lambda}C$.
\end{lem}

\begin{proof}
We suppress the superscript $\lambda$. For each $j\in\{1,2,\cdots d(\delta)\}$, define 
$F^{j}:\mathbb{C}^{n}\longrightarrow\mathcal{M}_{d(\delta)\times l(\delta)}$ to be the matrix whose first column is precisely the $j$th column of $F(z)$ and else 
are zero. Then clearly each $F^{j}$ is square integrable generalized $K$-spherical 
function. Hence, by Theorem \textbf{\ref{eqn t7.13}} and Proposition \textbf{\ref{eqn 
p7.14}}, it follows that there exist $l(\delta)\times l(\delta)$ constant matrix 
$C^{j}$ such that $F^{j}=\Phi_{\alpha}^{\delta}C^{j}.$ Equating the entries in 
first column we get $$F_{ij}=F^{j}_{i1}=\sum_{k=1}^{l(\delta)}(\Phi^{\delta}_{\alpha})_ 
{ik} C^{j}_{k1},~1\leq i,j\leq d(\delta);$$ which in matrix form can be 
written as $F=\Phi_{\alpha}^{\delta}C$, where $C$ is the $l(\delta)\times 
d(\delta)$ constant matrix given by $C_{kj}=C_{k1}^{j}$. Hence the proof.
\end{proof}

Let $\widehat{K}$ denote the set of all inequivalent unitary irreducible representations 
of $K$. For $\delta\in\widehat{K}$, let $\chi_{\delta}(k)=\textup{tr}[\delta(k)]$.

\begin{proof}
$(Proof~ of~ Theorem$ \textbf{\ref{eqn t8.1}})~ Since $f$ is $K$-finite, by Lemma 1.7, 
Chapter IV of \cite{H1}, there is a finite subset $\widehat{K}(f)$ of $\widehat{K}$ 
such that $$f(z)=\sum_{\delta\in\widehat{K}(f)}d(\delta)\chi_{\delta}\ast 
f(z):=\sum_{\delta\in\widehat{K}(f)}d(\delta)\int_{K}\chi_{\delta}(k)f(k^{-1}\cdot 
z)dk=\sum_{\delta\in\widehat{K}(f)}d(\delta)\textup{tr}(f^{\delta}),$$ where $$f^{\delta}
(z)=\int_{K}f(k^{-1}\cdot z)\delta(k)dk.$$ Clearly $f^{\delta}\in\mathcal{E}^{\lambda}
(\C)$. Since any $D\in\mathcal{L}_{K}(\mathbb{C}^{n})$ commutes with the action of 
$K$, clearly each $f^{\delta}$ is also a joint eigenfunction for all 
$D\in\mathcal{L}_{K}(\mathbb{C}^{n})$ with eigenvalue $\mu_{\alpha}$. Also note that 
$f^{\delta}(k\cdot z)=\delta(k)f^{\delta}(z)$. Now, for $z=(r,k M)$ and $m\in M$,
$$f^{\delta}(z)=f^{\delta}
(r,kmM)=f^{\delta}\big(km\cdot(r,M)\big)=\delta(k)\delta(m)f^{\delta}(r,M).$$
Therefore if $\delta\notin\widehat{K}_{M}$, integrating both side of the above equation 
over $M$, we get $f^{\delta}(z)=0$.
So assume that $\delta\in\widehat{K}_{M}$. But then by the previous lemma, each 
$f^{\delta}_{ij}$ can be written as $f^{\delta}_{ij}
(z)=\big\langle\Pi^{\lambda}
(z),\mathcal{W}^{\lambda}{(\widetilde{P}^{\delta}_{ij}})\big\rangle_{\alpha}^{\lambda}$ for some $\widetilde{P}_{ij}^{\delta}\in H_{\delta}$. 
Hence the proof follows.
\end{proof}

Let $f(z,t)$ be a joint eigenfunction for all $D\in\mathcal{L}_{K}(\mathfrak{h}_{n})$ 
with eigenvalue $\widetilde{\mu}^{\lambda}_{\alpha}$. Since $\frac{\partial}{\partial t}
(\phi^{\lambda}_{\alpha})=i\lambda \phi_{\alpha},$  $\widetilde{\mu}^{\lambda}_{\alpha}
(\partial/\partial t)=i\lambda$, $f$ has the form $f(z,t)=e^{i\lambda 
t}g(z).$ Clearly $g$ is a joint eigenfunction for all $D\in\mathcal{L}_{K}^{\lambda}
(\mathbb{C}^{n})$ with eigenvalue $\mu^{-\lambda}_{\alpha}$. Therefore, Theorem 
\textbf{\ref{eqn t8.1}} implies the following theorem on the Heisenberg group.

\begin{thm}\label{eqn t8.3}
Let $f$ be a $K$-finite joint eigenfunction for all $D\in\mathcal{L}_{K}
(\mathfrak{h}_n)$ with eigenvalue $\widetilde{\mu}_{\alpha}^{\lambda}$ such that 
$f(z,0)\in\mathcal{E}^{\lambda}(\C)$. Then $f(z,t)=\big\langle\Pi^{\lambda}(z,t),\mathcal{W}^{\lambda}
(P)\big\rangle_{\alpha}^{\lambda}$ for some $K$-harmonic polynomial 
$P$.
\end{thm}

The following proposition says that $\mu^{\lambda}_{\alpha}$'s are the only possible 
eigenvalues for joint eigenfunctions of all $D\in\mathcal{L}_{K}^{\lambda}
(\mathbb{C}^{n})$, which belong to $\mathcal{E}^{\lambda}(\C)$. Hence Theorem 
\textbf{\ref{eqn t8.1}}, actually describes all $K$-finite joint eigenfunctions of all 
$D\in\mathcal{L}_{K}^{\lambda}(\mathbb{C}^{n})$, which belong to 
$\mathcal{E}^{\lambda}(\C)$. Consequently, Theorem \textbf{\ref{eqn t8.3}} actually 
describes all $K$-finite joint eigenfunctions $f(z,t)$ of all $D\in\mathcal{L}
(\mathfrak{h}_{n})$ with eigenvalue $\widetilde{\mu}$, such that $\widetilde{\mu}
(\frac{\partial}{\partial t})$ is a non zero real number and 
$f(z,0)\in\mathcal{E}^{\lambda}(\C)$.

\begin{prop}\label{eqn p8.4}
Let $f\in\mathcal{E}^{\lambda}(\C)$ be a joint eigenfunction of all 
$D\in\mathcal{L}_{K}^{\lambda}(\C)$ with eigenvalue $\mu$. Then 
$\mu=\mu^{\lambda}_{\alpha}$ for some $\alpha\in\Lambda$.
\end{prop}

\begin{proof}
From Remark \textbf{\ref{eqn r7.31}}, recall $\mathcal{S}_{pq}$, the space of bigraded spherical harmonics of degree $(p,q)$. Take an orthonormal basis $\{Y^{j}_{pq}(\omega):j=1,2,
\cdot,d(p,q)\}$ for $\mathcal{S}_{pq}$, so that $\{Y^{j}_{pq}(\omega):j=1,2,
\cdots,d(p,q);p+q=k;k=0,1,\cdots\infty\}$ form a basis for $L^{2}(S^{2n-1})$. 
Therefore for each $r>0$, $f(r\omega)$ has the bigraded spherical harmonic expansion 
$$f(r\omega)=\sum_{m=0}^{\infty}\sum_{p+q=m}\sum_{j=1}^{d(p,q)}f^{j}_{pq}(r)Y^{j}_{pq}
(\omega),~\omega\in S^{2n-1},$$ where $$f_{pq}^{j}
(r)=\int_{S^{2n-1}}f(r\omega)\overline{Y^{j}_{pq}}(\omega)d\omega,~r>0.$$ Since $f$ is 
smooth, clearly $f^{j}_{pq}(r)$ is bounded at zero. Now let $\mu(\mathcal{L})=-2|\lambda|(2a+n)$, 
$a\in\mathbb{C}$. i.e $f$ is an eigenfunction of $\mathcal{L}$ with eigenvalue $-2|\lambda|(2a+n)
$. Then it can be shown that (see the proof of Proposition 4.5 \cite{T}), each 
$f_{pq}^{j}(r)Y^{j}_{pq}(\omega)$ is also an eigenfunction of $\mathcal{L}$ with eigenvalue 
$-2|\lambda|(2a+n)$ i.e $$\mathcal{L}[f_{pq}^{j}(r)Y^{j}_{pq}(\omega)]=-2|\lambda|(2a+n)
[f_{pq}^{j}(r)Y^{j}_{pq}(\omega)].$$ Writing $\mathcal{L}$ in polar coordinate, using the fact 
that $Y^{j}_{pq}$ is an eigenfunction of the spherical Laplacian on $S^{2n-1}$, and 
then making a change of variable $$f^{j}_{pq}(r)=\\
r^{p+q} u(2|\lambda|r^{2})e^{-|\lambda|r^{2}},$$ we get (for details see the proof of 
Proposition 4.4 \cite{T}) that $u$ satisfies the following confluent hypergeometric 
equation \bea\label{eqn 8.4.1}tu^{\prime\prime}(t)+(d-t)u^{\prime}(t)-(p-a)u(t)=0,\eea 
where $d=n+p+q$. The equation (\ref{eqn 8.4.1}) has two linearly independent solutions 
$u_1$ and $u_2$, with the following asymptotic behaviour (see \cite{ABCP}, page-145):

\indent {\rm{\textbf{(i)}}} If $(p-a)\neq 0,-1,-2,\cdots$, $$u_{1}(t)\sim\frac{(d-1)!}
{\Gamma(p-a)}e^{t}t^{p-a-d},~u_{2}(t)\sim t^{-(p-a)}~\textup{as}~t\rightarrow 
+\infty$$ $$u_{1}(t)\sim 1,~u_{2}(t)\sim\begin{cases}\frac{-\textup{log} t}{\Gamma(p-
a)} ~\textup{if} ~d=1\\ \frac{c}{t^{d-1}} ~\textup{if}~ d\geq 2\,
\end{cases}~\textup{as}~t\rightarrow 0^{+},$$ where $c$ is a non zero constant. 

\indent {\rm{\textbf{(ii)}}} If $(p-a)$ is a non positive integer, $$u_{1}
(t)=L^{d-1}_{a-p}(t),~u_{2}(t)\sim e^{t}(-t)^{a-p-d}~\textup{as}~t\rightarrow +\infty.
$$ Therefore, under the conditions on $f$, the only possibility is $(p-a)$ is a non 
positive integer and consequently $$f^{j}_{pq}(r)=r^{p+q}L^{n+p+q-1}_{a-p}(2|\lambda|
r^{2})e^{-|\lambda|r^{2}}.$$ So there exists non-positive integer $k$ such that $a=k$. 
Hence $f$ is a eigenfunction of $\mathcal{L}$ with eigenvalue $-2|\lambda|(2k+n)$. Therefore by 
Lemma \textbf{\ref{eqn l7.11}}, 
$f=f\times\psi_{k}=\sum_{V_{\beta}\subset\mathcal{P}_{k}(\C)}f\times\psi_{\beta}$. 
Since $f$ is non zero, there exists $\alpha\in\Lambda$ with 
$V_{\alpha}\subset\mathcal{P}_{k}(\C)$ such that $f\times\psi_{\alpha}\neq 0$. Now let 
$D\in\mathcal{L}_{K}(\C)$. Then $$Df=\sum_{V_{\beta}\subset\mathcal{P}_{k}
(\C)}D[f\times\psi_{\beta}]=\sum_{V_{\beta}\subset\mathcal{P}_{k}(\C)}f\times 
D\psi_{\beta}=\sum_{V_{\beta}\subset\mathcal{P}_{k}(\C)}\mu_{\beta}
(D)f\times\psi_{\beta}.$$ Again $$Df=\mu(D)f=\sum_{V_{\beta}\subset\mathcal{P}_{k}
(\C)}\mu(D)f\times\psi_{\beta}.$$ So we get $$\sum_{V_{\beta}\subset\mathcal{P}_{k}
(\C)}[\mu(D)-\mu_{\beta}(D)]f\times\psi_{\beta}=0.$$ Hence 
$$\sum_{V_{\beta}\subset\mathcal{P}_{k}(\C)}[\mu(D)-\mu_{\beta}
(D)]f\times\psi_{\beta}\times\psi_{\alpha}=0,$$ which implies that 
$[\mu(D)-\mu_{\alpha}(D)]f\times\psi_{\alpha}=0$. Since $f\times\psi_{\alpha}\neq 0$, 
we get $\mu(D)=\mu_{\alpha}(D)$. But $D\in\mathcal{L}_{K}(\C)$ is arbitrary. Hence 
$\mu=\mu_{\alpha}$.
\end{proof}

\begin{rem}\label{eqn r8.5}

Theorem \textbf{\ref{eqn t8.1}} holds true even if we assume that $f$ is a distribution such that $e^{-(|\lambda|-\epsilon)|\cdot|^{2}}f$ defines a tempered distribution for some $\epsilon>0$.


\end{rem}

\section{square integrable eigenfunctions}
In this section we prove the following theorem characterizing square integrable joint 
eigenfunctions of all $D\in\mathcal{L}^{\lambda}_{K}(\C)$. This is analogues to Theorem 3.3 in \cite{T}.  


\begin{thm}\label{eqn t10.1}
The square integrable joint eigenfunctions of all the operators 
$D\in\mathcal{L}_{K}^{\lambda}(\C)$ with eigenvalue $\mu^{\lambda}_{\alpha}$ are 
precisely $f(z)=\big\langle\Pi^{\lambda}(z),S\big\rangle^{\lambda}_{\alpha}$, where 
$S\in\mathcal{O}(V_{\alpha})$. Moreover $||f||^{2}_{2}=\pi^{n}(2|
\lambda|)^{-n}||S||_{\alpha}^{2}.$
\end{thm}

\begin{proof}
Let $f\in L^{2}(\C)$ be a joint eigenfunction of all $D\in\mathcal{L}_{K}(\C)$ with 
eigenvalue $\mu_{\alpha}$. We have 
$$f=\sum_{\delta\in\widehat{K}_{M}}d(\delta)\chi_{\delta}*f=
\sum_{\delta\in\widehat{K}_{M}}d(\delta)\textup{tr}(f^{\delta}),$$ where the series converges in $L^2(\C)$. Clearly each 
$f^{\delta}:\C\rightarrow \mathcal{M}_{d(\delta)\times d(\delta)}$ is a joint 
eigenfunction of all $D\in\mathcal{L}_{K}(\C)$ with eigenvalue $\mu_{\alpha}$. also 
$f^{\delta}(k\cdot z)=\delta(k)f^{\delta}(z)$. Therefore by Lemma \textbf{\ref{eqn 
l8.2}}, there is a $\big(l(\delta)\times d(\delta)\big)$ constant matrix $C_{\delta}$ 
such that $$d(\delta)f^{\delta}=\Psi^{\delta}_{\alpha}C_{\delta}=\pi^{-n}(2|
\lambda|)^{n}\big\langle\Pi(z),\mathcal{W}(P^{\delta})\big\rangle_{\alpha} C_{\delta}.
$$ Hence \bea\label{eqn 10.1}f(z)=\pi^{-n}(2|
\lambda|)^{n}\sum_{\delta\in\widehat{K}_{M}}\textup{tr}\big[\big\langle\Pi(z),\mathcal{W}
(P^{\delta})\big\rangle_{\alpha} C_{\delta}\big],\eea and \beas||f||
_{2}^{2}=\sum_{\delta\in\widehat{K}_{M}}\big|\big|
\textup{tr}\big[\Psi^{\delta}_{\alpha}C_{\delta}\big]\big|\big|
_{2}^{2}&=&\sum_{\delta\in\widehat{K}_{M}}\big|\big|
\textup{tr}\big[\big(\theta(P^{\delta})\psi_{\alpha}\big)C_{\delta}\big]\big|\big|
_{2}^{2}\\&=&\pi^{-n}(2|\lambda|)^{n}\sum_{\delta\in\widehat{K}_{M}}\big|\big|
\textup{tr}\big[\mathcal{W}(P^{\delta})C_{\delta}\big]\big|\big|_{\alpha}^{2},\eeas 
where the last equality follows from Lemma \textbf{\ref{eqn l7.5}} \textbf{(b)}. 
Therefore $$S:=\pi^{-n}(2|
\lambda|)^{n}\sum_{\delta\in\widehat{K}_{M}}\textup{tr}\big[\mathcal{W}
(P^{\delta})\vert_{V_{\alpha}}C_{\delta}\big]$$ defines an element in $\mathcal{O}
(V_{\alpha}),$ and consequently from (\ref{eqn 10.1}) we get 
$f(z)=\big\langle\Pi(z),S\big\rangle_{\alpha}$. Conversely let 
$f(z)=\big\langle\Pi(z),S\big\rangle_{\alpha}$ for some $S\in\mathcal{O}(V_{\alpha})$. 
Let $\widehat{K}(\alpha)=\big\{\delta\in\widehat{K}_{M}:\mathcal{W}
(H_{\delta})\vert_{V_{\alpha}}\neq\{0\}\big\}$. For each $\delta\in\widehat{K}(\alpha)$, 
choose $p_{j}^{\delta}\in H_{\delta}$, $j=1,2,\cdots n_{\alpha}(\delta)$ so that 
$\{\mathcal{W}(p_{j}^{\delta})\vert_{V_{\alpha}}:j=1,2,\cdots,n_{\alpha}(\delta)\}$ 
forms an orthonormal basis for $\mathcal{W}(H_{\delta})\vert_{V_{\alpha}}$. Hence by 
Proposition \textbf{\ref{eqn p7.7}}, $\{\mathcal{W}
(p_{j}^{\delta})\vert_{V_{\alpha}}:j=1,2,\cdots,n_{\alpha}(\delta);\delta\in\widehat{K}
(\alpha)\}$ is an orthonormal basis for $\mathcal{O}(V_{\alpha}).$ Therefore we can 
write for each $z\in\C$, \bea\label{eqn 10.2} 
f(z)=\big\langle\Pi(z),S\big\rangle_{\alpha}&=&\nonumber\sum_{\delta\in\widehat{K}
(\alpha)}\sum_{j=1}^{n_{\alpha}(\delta)}\big\langle\Pi(z),\mathcal{W}
(p^{\delta}_{j})\big\rangle_{\alpha}\big\langle\mathcal{W}
(p^{\delta}_{j}),S\big\rangle_{\alpha}\\&=&\pi^{n}(2|\lambda|)^{-
n}\sum_{\delta\in\widehat{K}(\alpha)}\sum_{j=1}^{n_{\alpha}
(\delta)}\big[\theta(p^{\delta}_{j})\psi_{\alpha}\big](z)\big\langle\mathcal{W}
(p^{\delta}_{j}),S\big\rangle_{\alpha}\eea by Proposition \textbf{\ref{eqn p7.14}}. 
But $$\sum_{\delta\in\widehat{K}(\alpha)}\sum_{j=1}^{n_{\alpha}(\delta)}\big|
\big\langle\mathcal{W}(p^{\delta}_{j}),S\big\rangle_{\alpha}\big|^{2}=||S||
_{\alpha}^{2}\leq\infty,$$ and by Lemma \textbf{\ref{eqn l7.5}} \textbf{(b)}, 
\begin{eqnarray*} \big\langle\theta(p^{\delta}_{j})\psi_{\alpha},
\theta(p^{\delta^{\prime}}_{j^{\prime}})\psi_{\alpha}\big\rangle=\begin{cases}0 
~\textup{if} ~(\delta,j)\neq(\delta^{\prime},j^{\prime})\\ \pi^{-n}(2|\lambda|)^{n} 
~\textup{if}~ (\delta,j)=(\delta^{\prime},j^{\prime})\,.\end{cases}
 \end{eqnarray*} Therefore it follows that the series for $f$ defined by equation 
(\ref{eqn 10.2}) converges in $L^{2}(\C)$. In particular $f\in L^{2}(\C)$. Since any 
$D\in\mathcal{L}_{K}(\C)$ is a polynomial coefficient differential operator we have 
$$Df(z)=\pi^{n}(2|\lambda|)^{-n}\sum_{\delta\in\widehat{K}(\alpha)}\sum_{j=1}^{n_{\alpha}
(\delta)}D\big[\theta(p^{\delta}_{j})\psi_{\alpha}\big](z)\big\langle\mathcal{W}
(p^{\delta}_{j}),S\big\rangle_{\alpha}$$ in the distribution sense. But 
$D\big[\theta(p^{\delta}_{j})\psi_{\alpha}\big]=\mu_{\alpha}(D)
[\theta(p^{\delta}_{j})\psi_{\alpha}\big]$. Therefore we can conclude that 
$Df=\mu_{\alpha}(D)f$. Hence $f$ is a joint eigenfunction of all $D\in\mathcal{L}_{K}
(\C)$ with eigenvalue $\mu_{\alpha}$. Also note that $||f||^{2}_{2}=\pi^{n}(2|
\lambda|)^{-n}||S||_{\alpha}^{2}$. Thus the proof is complete. 
\end{proof}

\section{Integral representations of eigenfunctions when dim$V_{\delta}^{M}=1$}
As usual let $(K,\mathbb{H}^{n})$ $(K\subset U(n))$ be a Gelfand pair such that the 
$K$-action on $\C$ is polar. In this section we consider the special case when 
dim$V_{\delta}^{M}=1$ for all $\delta\in\widehat{K}_{M}$, and (under the usual growth 
condition) characterize any joint eigenfunction of all 
$D\in\mathcal{L}_{K}^{\lambda}(\C)$. We show that for this special case it is enough 
to consider subgroups of the type $K=U(n_{1})\times U(n_{2})\times\cdots\times U(n_{m})$, 
$n_{1}+n_{2}+\cdots+n_{m}=n.$ Then the generalized $K$-spherical functions are given 
in terms of certain Laguerre polynomials. We use the well-known asymptotic 
behaviour of Laguerre polynomials to characterize joint eigenfunctions. We give 
two such characterizations. The first one is a direct generalization of Theorem 
\textbf{\ref{eqn t8.1}} and  Theorem \textbf{\ref{eqn t10.1}}. We will see that this 
is actually analogues to Theorem 4.1 in \cite{T}, which gives an 
integral representation of eigenfunctions. Though the ideas behind the proof are 
similar to that in \cite{T}, we give the details here, since we will be dealing with $K=U(n_1)\times U(n_2)$ instead of $K=U(n)$. The second one gives 
a different integral representation of eigenfunctions with an explicit kernel. 

\begin{lem}\label{eqn l9.1}
Suppose \textup{dim}$V^{M}_{\delta}=1$ for all $\delta\in\widehat{K}_{M}$. Also assume 
that the decomposition of $\mathcal{P}_{1}(\C)$ into $K$-irreducible subspaces is as 
follows : $\mathcal{P}_{1}(\C)=\bigoplus_{j=1}^{m}V_{j}$, where 
$V_{1}=\textup{span}\{z_1,z_2,\cdots,z_{n_{1}}\}, 
V_{2}=\textup{span}\{z_{n_1+1},z_{n_1+2},\cdots,z_{n_1+n_2}\},\cdots$; $n_1+n_2+\cdots 
n_m=n$. Then $\mathcal{P}(\C_{\mathbb{R}})^{K}=\mathcal{P}(\C_{\mathbb{R}})^{K_{0}},$ 
where $K_{0}=U(n_1)\times U(n_2)\cdots\times U(n_m).$ 
\end{lem}

\begin{proof}
For simplicity of the proof we take $m=2$. Let $\{v_1,v_2,\cdots,v_{d(\alpha)}\}$ be 
an orthonormal (in $\mathcal{H}^{\lambda}$ for $\lambda=\frac{1}{2}$) basis for $V_{\alpha}$ and 
$p_{\alpha}=\Sigma_{i=1}^{d(\alpha)}v_i\bar{v_i}$. Then $\{p_{\alpha}\}_{\alpha\in\Lambda}$ is a vector 
space basis for $\mathcal{P}(\C_{\mathbb{R}})^{K}$ (see Proposition 3.9 in 
\cite{BJR}). In particular, any $K$-invariant second degree homogeneous polynomial has to be a linear combination of $\gamma_1$ and $\gamma_2$, where $\gamma_{1}(z)=|z_1|^2+|z_2|^2+\cdots |z_{n_{1}}|
^2$ and $\gamma_{2}(z)=|z_{n_1+1}|^2+|z_{n_1+2}|^2+\cdots |z_{n}|^2$. Since $U(n_1)\times U(n_2)$ invariant elements in $\mathcal{P}
(\C_{\mathbb{R}})$ are generated by $\gamma_{1}$ and $\gamma_{2}(z)$, to prove 
the theorem it is enough to show that any $K$-invariant homogeneous polynomial can be written as a polynomial 
in $\gamma_1(z)$ and $\gamma_2(z)$. We prove this by induction on the degree of $K$-invariant homogeneous polynomials. Note that degree of a $K$-invariant homogeneous polynomial is always even. Denote the representation of $K$ on $V_{j}$ by 
$\delta_{j}$. Since $l(\delta_{j})$=dim$V^{M}_{j}=1$, $H_{\delta_{j}}=V_{j}$ (see 
(\ref{eqn21})). Now if $V_{\alpha}\subset\mathcal{P}_{2}(\C)$ then degree of 
$p_{\alpha}$ is $4$. Since $\mathcal{F}^{\prime}\big(\frac{\partial p_{\alpha}}{\partial{\bar{z}_{1}}}\big)$ is equal to $z_1$ times a $K$-invariant distribution, which transform according to representation $\delta_1$ and $K$-action commutes with $\mathcal{F}^{\prime}$, the same is true for $\frac{\partial p_{\alpha}}{\partial\bar{z}_1}$. Hence $\frac{\partial p_{\alpha}}{\partial{\bar{z}_{1}}}\in IH_{\delta_{1}}=IV_{1}$, where 
$I=\mathcal{P}(\C_{\mathbb{R}})^{K}$. But $\frac{\partial p_{\alpha}}
{\partial{\bar{z}_{1}}}$ being a third degree homogeneous polynomial and any $K$-invariant second degree homogeneous polynomial being a linear combination of $\gamma_1$ and $\gamma_2$, $\frac{\partial p_{\alpha}}
{\partial{\bar{z}_{1}}}$ has to be of the following form : \bea\label{eqn 
9.1}\frac{\partial p_{\alpha}}{\partial{\bar{z}_{1}}}(z)=\sum_{j=1}^{n_{1}}z_{j}
[a_{j}\gamma_{1}(z)+b_{j}\gamma_{2}(z)]. \eea   Similarly, as $\frac{\partial p_{\alpha}}
{\partial{\bar{z}_{n_{1}+1}}}$ is a third degree homogeneous polynomial which 
belongs to $IH_{\delta_{2}}=IV_{2}$, it has the following form : \bea\label{eqn 
9.2}\frac{\partial p_{\alpha}}{\partial{\bar{z}_{n_1+1}}}
(z)=\sum_{j=1}^{n_{2}}z_{n_1+j}[a^{\prime}_{j}\gamma_{1}(z)+b^{\prime}_{j}\gamma_{2}
(z)]. \eea From the above two equations we get $$\frac{\partial^{2} p_{\alpha}}
{\partial{\bar{z}_{n_{1}+1}}\partial{\bar{z}_{1}}}=\sum_{j=1}^{n_{1}}b_{j}z_{j}z_{n_1+1}
=\sum_{j=1}^{n_2}a^{\prime}_{j}z_{n_1+j}z_{1},$$ which implies that $b_{j}=0$ if 
$j\neq 1.$ So (\ref{eqn 9.1}) becomes \bea\label{eqn 9.3}\frac{\partial p_{\alpha}}
{\partial{\bar{z}_{1}}}(z)=z_1[a_{1}\gamma_{1}(z)+b_{1}\gamma_{2}
(z)]+\sum_{j=2}^{n_{1}}z_ja_{j}\gamma_{1}(z).\eea Now let $$v_i(z)=z_{1}
(c_{i1}z_1+c_{i2}z_2+\cdots c_{in}z_n)+q_{1}(z),~ i=1,2,\cdots d(\alpha),$$ where 
$q_{1}(z)$ is a second degree homogeneous holomorphic polynomial in $z_2,z_3,\cdots 
z_n.$ Then \bea\label{eqn 9.4}\frac{\partial p_{\alpha}}{\partial{\bar{z}_{1}}}
(z)=\sum_{i=1}^{d(\alpha)}[z_{1}(c_{i1}z_1+c_{i2}z_2+\cdots c_{in}z_n)+q_{1}(z)]
[2\bar{c}_{i1}\bar{z}_1+\bar{c}_{i2}\bar{z}_2+\cdots \bar{c}_{in}\bar{z}_n].\eea 
Equating the coefficient of $z_2z_1\bar{z}_{1}$ from the right hand sides of (\ref{eqn 
9.3}) and (\ref{eqn 9.4}) we get $a_2=2\Sigma c_{i2}\bar{c}_{i1}.$ Again equating the 
coefficients of $z_{1}^{2}\bar{z}_{2}$ from the right hand sides of (\ref{eqn 9.3}) 
and (\ref{eqn 9.4}) we get $\Sigma c_{i1}\bar{c}_{i2}=0$. Hence $a_2=0$. similarly 
$a_{3}=a_{4}=\cdots =a_{n_{1}}=0.$) So from (\ref{eqn 9.3}) we get $$\frac{\partial 
p_{\alpha}}{\partial{\bar{z}_{1}}}(z)=z_1[a_{1}\gamma_{1}(z)+b_{1}\gamma_{2}(z)].$$ 
Similarly we get $$\frac{\partial p_{\alpha}}{\partial{\bar{z}_{j}}}
(z)=z_j[a_{j1}\gamma_{1}(z)+b_{j1}\gamma_{2}(z)],j=1,2,\cdots n.$$ Hence we can write 
$p_{\alpha}$ as \begin{eqnarray*} p_{\alpha}(z)=\begin{cases}a_{j}\gamma_{1}
(z)^{2}+b_{j}\gamma_{1}(z)\gamma_{2}(z)+r_{j}(z) ~\textup{if} ~j=1,2,\cdots,n_1\\ 
c_{j}\gamma_{1}(z)\gamma_{2}(z)+d_{j}\gamma_{2}(z)^{2}+r_{j}(z) ~\textup{if}~ 
j=n_1+1,n_1+2,\cdots,n_1+n_2=n\,.\end{cases}
 \end{eqnarray*}
Here $r_{j}(z)$ is a fourth degree homogeneous polynomial in $z,\bar{z}$, which is 
independent of $\bar{z}_{j}$. For $j=1,2$, equating the coefficients of $|z_{1}|^{2}|
z_{2}^{2}|$, we get $a_1=a_2.$ Similarly we can show that all $a_{j}$'s are same; and 
all $d_{j}'s$ are same. Again for $j=i$ and $j=n_1+k~ (i=1,2,\cdots,n_1;k=1,2,
\cdots,n_2)$, equating the coefficients of $|z_{i}|^{2}|z_{n_{1}+k}|^{2}$, we get 
$b_{i}=c_{n_{1}+k}$. Hence we can write $$p_{\alpha}(z)=a_{1}\gamma_{1}
(z)^{2}+b_{1}\gamma_{1}(z)\gamma_{2}(z)+d_{1}\gamma_{2}(z)^{2}+\widetilde{r}_{j}(z),
\forall j=1,2,\cdots,n,$$ where $\widetilde{r}_{j}(z)$ is a fourth degree homogeneous 
polynomial in $z,\bar{z}$, which is independent of $\bar{z}_{j}$. Therefore 
$$p_{\alpha}(z)=a_{1}\gamma_{1}(z)^{2}+b_{1}\gamma_{1}(z)\gamma_{2}(z)+d_{1}\gamma_{2}
(z)^{2}+r(z),$$ where $r(z)(=\widetilde{r}_{1}(z)=\widetilde{r}_{2}(z)=\cdots=\widetilde{r}_{n}
(z))$ is a fourth degree homogeneous polynomial in $z$ only. But since $p_{\alpha}$ 
has the form $p_{\alpha}=\Sigma_{i=1}^{d(\alpha)}v_i\bar{v_i}$, it follows that 
$r(z)\equiv 0.$ Hence $$p_{\alpha}(z)=a_{1}\gamma_{1}(z)^{2}+b_{1}\gamma_{1}
(z)\gamma_{2}(z)+d_{1}\gamma_{2}(z)^{2}.$$ So we have proved that if 
$V_{\alpha}\subset\mathcal{P}_{2}(\C)$, then $p_{\alpha}$ can be written as a 
polynomial in $\gamma_{1}$ and $\gamma_{2}$. Hence, it follows that any $K$-invariant, $4$th degree, homogeneous polynomial can be written as a polynomial in $\gamma_1$ and $\gamma_2$. Now, let any $K$-invariant homogeneous polynomial of degree $2N$ can be written as a polynomial in $\gamma_1$ and $\gamma_2$. We have to show that any $K$-invariant homogeneous polynomial of degree $2(N+1)$ can be written as a polynomial in $\gamma_1$ and $\gamma_2$. But for this, it is enough to show the following : If $V_{\alpha}\subset\mathcal{P}_{N+1}(\C)$, then $p_{\alpha}$ can be written as a polynomial in $\gamma_1$ and $\gamma_2$. So fix a $\alpha\in\Lambda$, such that $V_{\alpha}\subset\mathcal{P}_{N+1}(\C)$. Then $\frac{\partial p_{\alpha}}{\partial\bar{z}_1}$ is a $(2N+1)$th degree homogeneous polynomial which belongs to $IV_1$. Therefore by the induction hypothesis, $\frac{\partial p_{\alpha}}{\partial\bar{z}_1}(z)$ has the following form $$\frac{\partial p_{\alpha}}{\partial\bar{z}_1}(z)=\sum_{j=1}^{n_1}
z_j\bigg[\sum_{l=0}^{N}a^l_j\big(\gamma_1(z)\big)^{N-l}\big(\gamma_2(z)\big)^{l}\bigg].$$ Similarly, $\frac{\partial p_{\alpha}}
{\partial{\bar{z}_{n_{1}+1}}}$ has the following form $$\frac{\partial p_{\alpha}}
{\partial{\bar{z}_{n_{1}+1}}}(z)=\sum_{j=1}^{n_2}
z_{n_1+j}\bigg[\sum_{l=0}^{N}b^l_j\big(\gamma_1(z)\big)^{N-l}\big(\gamma_2(z)\big)^{l}\bigg].$$ Now using similar arguments as before, it is possible to show that $p_{\alpha}$ is generated by $\gamma_1$ and $\gamma_2$. Hence the proof.
\end{proof}

\begin{prop}\label{eqn p9.2}
Suppose \textup{dim}$V^{M}_{\delta}=1$ for all $\delta\in\widehat{K}_{M}$. Then there 
exist $
J\in U(n)$, and positive integers $n_1,n_2,\cdots n_m$ with $n_1+n_2+\cdots 
n_m=n$, such that $\mathcal{P}(\C_{\mathbb{R}})^{K}=\mathcal{P}
(\C_{\mathbb{R}})^{K_{0}},$ where $K_{0}=J\big[U(n_1)\times U(n_2)\cdots\times 
U(n_m)\big]J^{-1}.$ 
\end{prop}

\begin{proof}
Let $\mathcal{P}_{1}(\C)=\bigoplus_{l=1}^{m}V_{l}$ be the decomposition of 
$\mathcal{P}_{1}(\C)$ into $K$-irreducible subspaces. $V_{l}$'s are pairwise 
orthogonal in $\mathcal{H}^{\frac{1}{2}}$. Let dim$V_{l}=n_{l}$ so that 
$n_{1}+n_{2}+\cdots n_{m}=n$. Let $u_{i}(z)=z_{i}$. Choose an orthonormal (in 
$\mathcal{H}^{\frac{1}{2}}$) basis $\{v_{i}(z)=\Sigma_{j=1}^{n}c_{ij}u_{j}:i=1,2,
\cdots,n\}$ for $\mathcal{P}_{1}(\C)$ such that $\{v_{1},v_{2},\cdots,v_{n_{1}}\}$, 
$\{v_{n_{1}+1},v_{n_{1}+2},\cdots,v_{n_{1}+n_{2}}\}$, $\cdots$ form a basis for 
$V_{1},V_{2},\cdots$ respectively. Since $\{u_{1}(z),u_{2}(z),\cdots,u_{n}(z)\}$ is an 
orthonormal set in $\mathcal{H}^{\frac{1}{2}}$, we get $ \Sigma_{j=1}^{n}|c_{ij}|
^{2}=1;~\textup{and}~ \Sigma_{j=1}^{n}c_{ij}\overline{c_{i^{\prime}j}}=
0~\textup{if}~i\neq i^{\prime}.$ Hence the matrix  
$J:=\big(\bar{c}_{ij}\big)_{n\times n}$ is unitary. Therefore $J^{-1}=J^{*}=(c_{ji})_{n\times n}$, 
which implies that $(J\cdot u_{i})(z)=u_{i}(J^{-1}\cdot z)=v_{i}(z)$ or $J\cdot 
u_{i}=v_{i}.$ Consequently  the decomposition of $\mathcal{P}_{1}(\C)$ into 
$J^{-1}KJ$-irreducible subspaces is given by $\mathcal{P}_{1}
(\C)=\bigoplus_{l=1}^{m}V^{\prime}_{l}$, where  
$V_{1}^{\prime}=\textup{span}\{u_{1},u_{2},\cdots,u_{n_{1}}\},$
 $V_{2}^{\prime}=\textup{span}\{u_{n_{1}+1},u_{n_{1}+2},
\cdots,u_{n_{1}+n_{2}}\}\cdots$. Next, $M=K_{z_0}$ for some $K$-regular point $z_0$. 
Then clearly $J^{-1}\cdot z_0$ is a $J^{-1}KJ$-regular point, and 
$J^{-1}MJ=[J^{-1}KJ]_{J^{-1}\cdot z_0}$. Let $K^{\prime}=J^{-1}KJ$ and 
$M^{\prime}=J^{-1}MJ$. For each $\delta\in\widehat{K}_{M}$ define the irreducible unitary 
representation $\delta^{\prime}$ of $K^{\prime}$ on $V_{\delta^{\prime}}=V^{\delta}$ 
by $\delta^{\prime}(J^{-1}kJ)=\delta(k)$ for all $k\in K$. Then it is easy to see that 
the map $\delta\rightarrow\delta^{\prime}$ is a bijection from $\widehat{K}_{M}$ onto 
$\widehat{K^{\prime}}_{M^{\prime}}$, and dim$V_{\delta^{\prime}}^{M^{\prime}}=1$ for all 
$\delta^{\prime}\in\widehat{K^{\prime}}_{M^{\prime}}$. Therefore by the previous lemma 
$\mathcal{P}(\C_{\mathbb{R}})^{K^{\prime}}=\mathcal{P}
(\C_{\mathbb{R}})^{K^{\prime}_{0}},$ where $K^{\prime}_{0}=U(n_1)\times 
U(n_2)\cdots\times U(n_m).$ Hence the proof. 
\end{proof}

If dim$V_{\delta}^{M}=1$ for all $\delta\in\widehat{K}_{M}$, the above proposition says 
that with respect to a suitable coordinate system on $\C$, the $K$-invariant 
polynomials are same as that of $U(n_{1})\times U(n_{2})\times\cdots\times U(n_{m})  
$, $n_{1}+n_{2}+\cdots+n_{m}=n.$  Since $\mathcal{P}(\C_{\mathbb{R}})^{K}$ determines 
$\mathcal{L}_{K}(\textbf{h}_{n})$ (see \cite{BJR}, section-3), and hence 
$\mathcal{L}_{K}^{\lambda}(\C)$, to find joint eigenfunctions of all 
$D\in\mathcal{L}^{\lambda}_{K}(\C)$ for this special case, it is enough to consider 
the groups $U(n_{1})\times U(n_{2})\times\cdots\times U(n_{m})$, 
$n_{1}+n_{2}+\cdots+n_{m}=n.$ For simplicity of notation, here we only deal with the 
particular case : $m=2$. 

So, from now on $K$ always stands for $U(n_{1})\times 
U(n_{2})$, and $M$ the stabilizer of the $K$-regular point $e=(1,0,\cdots,1,0,
\cdots,0)\in\C$, where the second $1$ is at the $(n_{1}+1)$th position. Via the map 
$kM\rightarrow k\cdot e$, we have the identification $$ K/M=K\cdot e =\{z\in\C:
\sum_{j=1}^{n_{1}}|z_{j}|^{2}=1,\sum_{j=n_{1}+1}^{n}|z_{j}|^{2}=1\}.$$ If we identify 
$\C$ with $\mathbb{C}^{n_{1}}\times\mathbb{C}^{n_{2}}$ by the map 
$z\rightarrow\big((z_{1},z_{2},\cdots,z_{n_{1}}),(z_{n_{1}+1},z_{n_{1}+2},
\cdots,z_{n})\big),$ then $K/M=S^{2n_{1}-1}\times S^{2n_{2}-1}$, where $S^{2n_{1}-1}$ 
is the unit sphere in $\mathbb{C}^{n_{1}}$ and $S^{2n_{2}-1}$ is the unit sphere in 
$\mathbb{C}^{n_{2}}$. Now we explicitly describe the spaces $H_{\delta}$ and 
$\mathcal{E}_{\delta}(K/M)=\mathcal{E}_{\delta}(S^{2n_{1}-1}\times S^{2n_{2}-1})$. 
Since $I_{+}$, the set of polynomials in $I=\mathcal{P}(\C_{\mathbb{R}})^{K}$ without 
constant term, is generated by $\Sigma_{j=1}^{n_{1}}|z_{j}|^{2},
\Sigma_{j=n_{1}+1}^{n}|z_{j}|^{2},$ the set of $K$-harmonic polynomials is given by 
$$H=\{P\in\mathcal{P}(\C_{\mathbb{R}}):\triangle_{1}P=0,\triangle_{2}P=0\},$$ where 
$$\triangle_{1}=\sum_{j=1}^{n_{1}}\frac{\partial^{2}}{\partial 
z_{j}\partial\bar{z}_{j}},~\triangle_{2}=\sum_{j=n_{1}+1}^{n}\frac{\partial^{2}}
{\partial z_{j}\partial\bar{z}_{j}}.$$ For $z\in\C$, let $z^{1}=(z_{1},z_{2},\cdots 
z_{n_{1}})\in\mathbb{C}^{n_{i}}$, $z^{2}=(z_{n_{1}+1},z_{n_{1}+2},\cdots, z_{n})$. Let 
$i=1$ or $2$. For each pair of positive integer $(p,q)$, we define 
$\mathcal{P}^{i}_{pq}$ to be the subspace of $\mathcal{P}
(\mathbb{C}^{n_{i}}_{\mathbb{R}})$ consisting of all polynomials of the form 
$$P(z^{i})=\sum_{|\alpha_{i}|=p}\sum_{|\beta_{i}|=q}(z^{i})^{\alpha_{i}}
(\bar{z^{i}})^{\beta_{i}}.$$ Here $\alpha_{i}$ and $\beta_{i}$ are multi-indices of non 
negative integers of length $n_{i}$. We let $$H^{i}=\{P\in\mathcal{P}
(\mathbb{C}^{n_{i}}_{\mathbb{R}}):
\triangle_{i}P=0\};~H^{i}_{pq}=\{P\in\mathcal{P}^{i}_{pq}:\triangle_{i}P=0\}.$$ We 
have the identification $\mathcal{P}(\C_{\mathbb{R}})=\mathcal{P}
(\mathbb{C}_{\mathbb{R}}^{n_{1}})\otimes\mathcal{P}(\mathbb{C}_{\mathbb{R}}^{n_{2}})$, 
$H=H^{1}\otimes H^{2},$ and consequently $H$ has the algebraic direct sum 
decomposition : $$H=\bigoplus_{p_1,q_1,p_2,q_2\in\mathbb{Z}^{+}} H^{1}_{p_1q_1}\otimes 
H^{2}_{p_2q_2}.$$ Here $\mathbb{Z}^{+}$ denotes the set of non negative integers. Also 
note that each $P\in\mathcal{P}^{1}_{p_{1}q_{1}}\otimes\mathcal{P}^{2}_{p_{2}q_{2}}$ 
satisfy the homogeneity condition $$P(\lambda_{1}z^{1},
\lambda_{2}z^{2})=\lambda_{1}^{p_{1}}\bar{\lambda}_{1}^{q_{1}}
\lambda_{2}^{p_{2}}\bar{\lambda}_{2}^{q_{2}}P(z)$$ for all $\lambda_{1},
\lambda_{2}\in\mathbb{C}.$ Let $\mathcal{E}^{i}_{pq}$ stand for the restrictions of 
members of $H^{i}_{pq}$ to $S^{2n_{i}-1}$. The relation between $P\in H^{i}_{pq}$ and 
its restriction $Y^{i}_{pq}$ is given by $P^{i}_{pq}(z^{i})=|z^{i}|^{p+q}Y^{i}_{pq}
(\omega^{i})$, if $z^{i}=r^{i}\omega^{i}$, $r^{i}>0,~\omega^{i}\in S^{2n_{i}-1}$. The 
natural action of $U(n_{i})$ defines a unitary representation, $\delta^{i}_{pq}$ on each of these spaces 
$\mathcal{E}^{i}_{pq}$, considered as a Hilbert subspace of $L^{2}(S^{2n_{i}-1})$. 
Clearly the restriction of 
$H^{1}_{p_1q_1}\otimes H^{2}_{p_2q_2}$ to $S^{2n_1-1}\times S^{2n_2-1}$ is given by 
$\mathcal{E}^{1}_{p_1q_1}\otimes\mathcal{E}^{2}_{p_2q_2}$. If we consider this as a 
Hilbert subspace of $L^{2}(S^{2n_1-1}\times S^{2n_2-1})$, then the natural action of 
$K$ on each of these spaces 
$\mathcal{E}^{1}_{p_1q_1}\otimes\mathcal{E}^{2}_{p_2q_2}$ defines a unitary 
representation which is same as $\delta^{1}_{p_1q_1}\otimes\delta^{2}_{p_2q_2}$. Now 
for each fixed $i\in\{1,2\}$, we have the following well known facts about the class 
one representations of $U(n_{i})$ (see \cite{T2}, page: 64-69) : The representations 
$\delta^{i}_{pq}$ of $U(n_{i})$ on $\mathcal{E}^{i}_{pq}$ are irreducible. 
$\delta^{i}_{pq}$ and $\delta^{i}_{p^{\prime}q^{\prime}}$ are unitarily equivalent if 
and only if $(p,q)=(p^{\prime},q^{\prime})$. Let $M_{i}\subset U(n_{i})$ be the 
stabilizer of $(1,0,\cdots,0)\in\mathbb{C}^{n_{i}}$, so that $M=M_1\times M_2$. Any 
$\delta\in\widehat{U(n_{i})}_{M_{i}}$ is equivalent to some $\delta^{i}_{pq}$. $L^{2}
(S^{2n_{i}-1})$ has the orthogonal Hilbert space decomposition : $L^{2}
(S^{2n_{i}-1})=\bigoplus^{\bot}_{p,q\in\mathbb{Z}^{+}}\mathcal{E}^{i}_{pq}$. From these facts we can prove the following proposition.  

\begin{prop}\label{eqn p9.3}

\indent{\rm{\textbf{(a)}}} The representations 
$\delta^{1}_{p_1q_1}\otimes\delta^{2}_{p_2q_2}$ of $K$ on 
$\mathcal{E}^{1}_{p_1q_1}\otimes\mathcal{E}^{2}_{p_2q_2}$ are irreducible. 
$\delta^{1}_{p_1q_1}\otimes\delta^{2}_{p_2q_2}$ and 
$\delta^{1}_{p^{\prime}_1q^{\prime}_1}\otimes\delta^{2}_{p^{\prime}_2q^{\prime}_2}$ 
are unitarily equivalent if and only if $(p_1,q_1,p_2,q_2)
=(p_1^{\prime},q_1^{\prime},p_2^{\prime},q_2^{\prime}).$ Moreover any 
$\delta\in\hat{K}_{M}$ is equivalent to some 
$\delta^{1}_{p_1q_1}\otimes\delta^{2}_{p_2q_2}.$

\indent{\rm{\textbf{(b)}}} We have the orthogonal Hilbert space decomposition of 
$L^{2}(S^{2n_1-1}\times S^{2n_2-1})$ : $$L^{2}(S^{2n_1-1}\times 
S^{2n_2-1})=\bigoplus^{\bot}_{p_1,q_1,p_2,q_2
\in\mathbb{Z}^{+}}\mathcal{E}^{1}_{p_1q_1}\otimes\mathcal{E}^{2}_{p_2q_2}.$$  

\end{prop}

By the above proposition, the decomposition of $\mathcal{P}(\C)$ into $K$-irreducible 
subspaces is given by $$\mathcal{P}
(\C)=\bigoplus_{m_{1},m_{2}\in\mathbb{Z}^{+}}V_{m_{1}m_{2}},$$ where 
$$V_{m_{1}m_{2}}=\textup{span}\big\{(z^{1})^{\alpha_{1}}(z^{2})^{\alpha_{2}}:|
\alpha_{1}|=m_{1},|\alpha_{2}|=m_{2}\big\}.$$ Denote the corresponding bounded $K$-
spherical functions by $\phi^{\lambda}_{m_{1}m_{2}}(z,t)=e^{i\lambda 
t}\psi^{\lambda}_{m_{1}m_{2}}(z)$. Then $\psi^{\lambda}_{m_{1}m_{2}}(z)$ is a joint 
eigenfunction of all $D\in\mathcal{L}_{K}^{\lambda}(\C)$ with eigenvalue, say 
$\mu^{\lambda}_{m_{1}m_{2}}$. Note that here $\mathcal{L}^{\lambda}_{K}(\C)$ is 
generated by $$\mathcal{L}^{\lambda}_{1}:=\sum_{j=1}^{n_{1}}L^{\lambda}_{j}
\overline{L}_{j}^{\lambda}+\overline{L}^{\lambda}_{j}L^{\lambda}_{j},
~\textup{and}~\mathcal{L}^{\lambda}_{2}
:=\sum_{j=n_{1}+1}^{n_{2}}L^{\lambda}_{j}
\overline{L}_{j}^{\lambda}+\overline{L}^{\lambda}_{j}L^{\lambda}_{j}.$$ Let $L_{k}^{\alpha}$ be 
the $k$th  degree Laguerre polynomial of type $\alpha$. For any $\nu\in\mathbb{N}$, 
and any $\zeta\in\mathbb{C}^{\nu}$, define $$\varphi_{k,\lambda}^{\alpha}
(\zeta)=L_{k}^{\alpha}\big(2|\lambda||\zeta|^{2}\big)e^{-|\lambda||\zeta|^{2}}.$$

\begin{prop}\label{eqn p9.4}
$\mu^{\lambda}_{m_{1}m_{2}}(\mathcal{L}_{i}^{\lambda})=-2|\lambda|(2m_{i}+n_{i})$, $i=1,2$. 
$\psi^{\lambda}_{m_{1}m_{2}}$ has the following formulae in terms of Laguerre 
polynomials : $$\psi_{m_{1}m_{2}}^{\lambda}(z)=\pi^{-n}(2|
\lambda|)^{n}\prod_{i=1}^{2}\varphi_{m_{i},\lambda}^{n_{i}-1}(z^{i}).$$
\end{prop}

\begin{proof}
As usual we drop the superscript $\lambda$. Take $z_{1}^{m_{1}}z_{n_{1}+1}^{m_{2}}\in 
V_{m_{1}m_{2}}$. By Remark \textbf{\ref{eqn r5.9}}, $$\mathcal{G}(\mathcal{L}_{1})
[z_{1}^{m_{1}}z_{n_{1}+1}^{m_{2}}]=\mu_{m_{1}m_{2}}(\mathcal{L}_{1})
[z_{1}^{m_{1}}z_{n_{1}+1}^{m_{2}}],$$ which, by Proposition \textbf{\ref{eqn p5.3}}, 
reduces to $$\bigg(-\sum_{j=1}^{n_{1}}\overline{W}_{j}W_{j}+
W_{j}\overline{W}_{j}\bigg)
[z_{1}^{m_{1}}z_{n_{1}+1}^{m_{2}}]=\mu_{m_{1}m_{2}}(\mathcal{L}_{1})
[z_{1}^{m_{1}}z_{n_{1}+1}^{m_{2}}].$$ Using the definition of $W_{j}$ and 
$\overline{W}_{j}$, an easy calculation shows that $\mu_{m_{1}m_{2}}(\mathcal{L}_{1})=-2|
\lambda|(2m_{1}+n_{1})$. Similarly $\mu_{m_{1}m_{2}}(\mathcal{L}_{2})=-2|\lambda|
(2m_{2}+n_{2})$. Since $\varphi_{m_{i},\lambda}^{n_{i}-1}(z^{i})$ is an eigenfunction 
of $\mathcal{L}_{i}$ with eigenvalue $-2|\lambda|(2m_{i}+n_{i})$, $\Pi_{i=1}^{2}\varphi_{m_{i},
\lambda}^{n_{i}-1}(z^{i})$ is a joint eigenfunction of $\mathcal{L}_{K}(\C)$ with 
eigenvalue $\mu_{m_{1}m_{2}}$ . Hence $\psi_{m_{1}m_{2}}
(z)=c\Pi_{i=1}^{2}\varphi_{m_{i},\lambda}^{n_{i}-1}(z^{i}),$ for some constant $c$. To 
calculate the constant $c$, first note that by Proposition \textbf{\ref{eqn p6.4}}, 
$\psi_{m_{1}m_{2}}(z)=\psi_{m_{1}m_{2}}\times\psi_{m_{1}m_{2}}(z)$. In particular, 
putting $z=0$, we get $$cL_{m_{1}}^{n_{1}-1}(0)L_{m_{2}}^{n_{2}-1}
(0)=c^{2}\prod_{i=1}^{2}\int_{\mathbb{C}^{n_{i}}}\big[L_{m_{i}}^{n_{i}-1}\big(2|
\lambda||z^{i}|^{2}\big)\big]^{2}e^{-2|\lambda||z^{i}|^{2}}dz^{i}.$$ Using the well-
known facts  $$L^{\alpha}_{k}(0)=\frac{\Gamma(k+\alpha+1)}
{\Gamma(k+1)\Gamma(\alpha+1)},~\textup{and}~\int_{0}^{\infty}[L_{k}^{\alpha}
(r)]^{2}e^{-r}r^{\alpha}dr=\frac{\Gamma(k+\alpha+1)}{\Gamma(k+1)},$$ we can deduce 
that $c=\pi^{-n}(2|\lambda|)^{n}$. Hence the proof.
\end{proof}

\textbf{From now on we always assume that $\lambda>0$ and state our results only for 
$\lambda>0$. The corresponding results for $\lambda<0$ can be obtained by 
interchanging the role of $p_i$ and $q_i$.}

\begin{prop}\label{eqn p9.5}
Let $P\in H^{1}_{p_1q_1}\otimes H^{2}_{p_2q_2}$. Then $$\theta^{\lambda}
(P)\psi^{\lambda}_{m_{1}m_{2}}(z)=\pi^{-n}(2|\lambda|)^{n}P(z)\prod_{i=1}^{2}
(-1)^{q_{i}}(2|\lambda|)^{p_{i}+q_{i}}\varphi_{m_{i}-p_i,\lambda}^{n_{i}+p_i+q_i-1}
(z^{i})$$ if $p_{i}\leq m_{i}$ for all $i=1,2$; otherwise $\theta^{\lambda}
(P)\psi^{\lambda}_{m_{1}m_{2}}(z)=0$.
\end{prop}

\begin{proof}
Since $\delta^{1}_{p_1q_1}\otimes\delta^{2}_{p_2q_2}$ has a unique (upto a constant 
multiple) $M$-fixed vector, by Corollary \textbf{\ref{eqn c4.3}}, it is enough to 
prove the proposition for 
$P(z)=z_{1}^{p_1}\bar{z}_{2}^{q_{1}}z_{n_{1}+1}^{p_2}\bar{z}_{n_{1}+2}^{q_{2}}$, which 
clearly belongs to $H^{1}_{p_1q_{1}}\otimes H^{2}_{p_{2}q_{2}}.$ Since 
$\theta(P)=\overline{R}_{2}^{q_{1}}\overline{R}_{n_{1}+2}^{q_{2}}(-R_{1})^{p_{1}}
(R_{n_{1}+1})^{q_{2}}$, it follows that  
$$\theta(P)\psi_{m_{1}m_{2}}=\big[\theta(z_1^{p_{1}}
\bar{z}_{2}^{q_{1}})\psi_{m_{1}}
(z^{1})\big]\big[\theta(z_{n_1+1}^{p_{2}}\bar{z}_{n_2+2}^{q_{2}})
\psi_{m_{2}}(z^{2})\big],$$ where $$\psi_{m_{i}}(z^{i})=\pi^{-n_{i}}(2|
\lambda|)^{n_{i}}\varphi^{n_{i}-1}_{m_{i},\lambda}(z^{i}),~i=1,2.$$ Hence the proof follows by (\ref{eqn r7.5.4}). 

\end{proof}

\begin{cor}\label{eqn c9.6}

\indent{\rm{\textbf{(a)}}} Let $Pg\in\mathcal{E}^{\lambda}(\C)$ be a joint 
eigenfunction of all $D\in\mathcal{L}_{K}^{\lambda}(\C)$ with eigenvalue 
$\mu_{m_{1}m_{2}}^{\lambda}$, where $g$ is $K$-invariant and $P\in 
H^{1}_{p_{1}q_{1}}\otimes H^{2}_{p_2q_2}$. Then there is a constant $c_{p_1q_1p_2q_2}$ 
such that $$P(z)g(z)=c_{p_1q_1p_2q_2}P(z)\prod_{i=1}^{2}(-1)^{q_{i}}(2|
\lambda|)^{p_{i}+q_{i}}\varphi_{m_{i}-p_i,\lambda}^{n_{i}+p_i+q_i-1}(z^{i}),$$ when 
$p_{i}\leq m_{i}$ for all $i=1,2$; otherwise $Pg=0$. 
 
\indent{\rm{\textbf{(b)}}} Let $P\in H^{1}_{p_1q_1}\otimes H^{2}_{p_2q_2}$. Then 
$$\big\langle\Pi^{\lambda}(z),\mathcal{W}^{\lambda}(P)\big\rangle^{\lambda}_{m_1 
m_2}=P(z)\prod_{i=1}^{2}(-1)^{q_{i}}(2|\lambda|)^{p_{i}+q_{i}}\varphi_{m_{i}-p_i,
\lambda}^{n_{i}+p_i+q_i-1}(z^{i}),$$ when $p_{i}\leq m_{i}$ for all $i=1,2$; otherwise 
$\big\langle\Pi^{\lambda}(z),\mathcal{W}^{\lambda}
(P)\big\rangle^{\lambda}_{m_1m_2}=0$. 

\end{cor}

\begin{proof}
Since $\delta^{1}_{p_1q_1}\otimes\delta^{2}_{p_2q_2}$ has unique (upto a constant 
multiple) $M$-fixed vector, \textbf{(a)} follows from Theorem \textbf{\ref{eqn t7.13}} 
and Proposition \textbf{\ref{eqn p9.5}}; \textbf{(b)} follows from Proposition 
\textbf{\ref{eqn p7.14}} and Proposition \textbf{\ref{eqn p9.5}}. 
\end{proof}

\begin{lem}\label{eqn l9.7}
Let $\{Y^{j}_{p_1q_1p_2q_2}:j=1,2,\cdots d(p_1,q_1,p_2,q_2)\}$ be an orthonormal basis 
for $\mathcal{E}^{1}_{p_1q_1}\otimes\mathcal{E}^2_{p_2q_2}$ so that 
$\{Y^{j}_{p_1q_1p_2q_2}:j_i=1,2,\cdots d(p_1,q_1,p_2,q_2); 
p_1,q_1,p_2,q_2\in\mathbb{Z}^{+}\}$ forms an orthonormal basis for $L^{2}
(S^{2n_1-1}\times S^{2n_2-1})$. Let $Y^{j}_{p_1q_1p_2q_2}$ be the restriction of 
$\widetilde{P}^{j}_{p_1q_1p_2q_2}\in H^{1}_{p_1q_1}\otimes H^{2}_{p_2q_2}$ i.e 
$$\widetilde{P}^{j}_{p_1q_1p_2q_2}(z)=\widetilde{P}^{j}_{p_1q_1p_2q_2}
(z^{1},z^{2})=r_{1}^{p_1+q_1}r_{2}^{p_2+q_2}Y^{j}_{p_1q_1p_2q_2}(\omega^{1},
\omega^{2}),$$ where $z^{i}=r_i\omega^i,~ \omega^i\in S^{2n_i-1}.$  Define 
\bea\label{eqn  9.30}P^{j}_{p_1q_1p_2q_2}(z)=\sqrt{\prod_{i=1}^{2}\Gamma(n_i)(2|
\lambda|)^{-(p_i+q_i)}\frac{\Gamma(m_i-p_i+1)}{\Gamma(m_i+n_i+q_i)}} 
~\widetilde{P}^{j}_{p_1q_1p_2q_2}(z).\eea Then $$\{P^{j}_{p_1q_1p_2q_2}
(z):j=1,2,\cdots,d(p_1,q_1,p_2,q_2);p_i\leq m_i,q_i\in\mathbb{Z}^{+};i=1,2\}$$ forms 
an orthonormal basis for $\mathcal{O}^{\lambda}(V_{m_1m_2}).$ 
\end{lem}

\begin{proof}
Taking $p=q=P^{j}_{p_1q_1p_2q_2}$ in Lemma \textbf{\ref{eqn l7.5}} \textbf{(b)}, we 
get $$\big|\big|\mathcal{W}(P^{j}_{p_1q_1p_2q_2})\big|\big|_{m_1m_2}=\sqrt{\pi^{n}(2|
\lambda|)^{-n}}\big|\big|\theta(P)\psi_{m_1m_2}\big|\big|_2,$$ Therefore by 
Proposition \textbf{\ref{eqn p9.5}}, $\mathcal{W}(P^{j}_{p_1q_1p_2q_2})=0$, unless  
$p_1\leq m_1$, $p_2\leq m_2$; and if $p_1\leq m_1$, $p_2\leq m_2$, writing the right 
hand side of the above equation in polar coordinates and using the formulae 
$$\int_{0}^{\infty}[L_{k}^{\alpha}(r)]^{2}e^{-r}r^{\alpha}dr=\frac{\Gamma(k+\alpha+1)}
{\Gamma(k+1)},$$ we can deduce that $\big|\big|\mathcal{W}(P^{j}_{p_1q_1p_2q_2})\big|
\big|_{m_1m_2}=1$. But then, since each 
$\delta^{1}_{p_1q_1}\otimes\delta^{2}_{p_2q_2}$ has a unique (upto a constant multiple) 
$M$-fixed vector, the proof follows from Proposition \textbf{\ref{eqn p7.7}}. 
\end{proof}

\begin{lem}\label{eqn l9.8}
Let $\widetilde{P}^{j}_{p_1q_1p_2q_2}$ and $P^{j}_{p_1q_1p_2q_2}$  are as in the 
previous lemma. If $f$ is a joint eigenfunction of all $D\in\mathcal{L}_{K}^{\lambda}
(\C)$ with eigenvalue $\mu^{\lambda}_{m_1m_2}$ satisfying 
$\chi_{\delta}*f\in\mathcal{E}^{\lambda}(\C)$ for each $\delta\in\hat{K}_{M}$, then 
there exist constants $a^{j}_{p_1q_1p_2q_2}$ such that \bea\label{eqn 
9.31}f(z)=\mathop{\sum_{p_1,q_1,p_2,q_2\in\mathbb{Z}^{+}}}_{p_1\leq m_1,p_2\leq  
m_2}\sum_{j=1}^{d(p_1,q_1,p_2,q_2)}a^{j}_{p_1q_1p_2q_2}
\big\langle\Pi^{\lambda}(z),\mathcal{W}^{\lambda}
(P^{j}_{p_1q_1p_2q_2})\big\rangle^{\lambda}_{m_1 m_2},\eea where the series converges uniformly over compact subsets of $\C$. $a_{p_1q_1p_2q_2}$'s 
satisfy the following : \bea\label{eqn 
9.32}\mathop{\sum_{p_1,q_1,p_2,q_2\in\mathbb{Z}^{+}}}_{p_1\leq m_1,p_2\leq  
m_2}\sum_{j=1}^{d(p_1q_1p_2q_2)}|a^j_{p_1q_1p_2q_2}|
^{2}\prod_{i=1}^{2}\frac{k_i^{q_i}}{\Gamma(n_i+p_i+q_i)}<\infty,~\forall 
k_1,k_2\in\mathbb{N}.\eea
\end{lem}

\begin{proof}
By Proposition \textbf{\ref{eqn p9.3}} \textbf{(b)}, we have the expansion, for fixed 
$r_1,r_2>0$,  \bea\label{eqn 
9.33}f(z)=f(r_1\omega^1,r_2\omega^2)=\mathop{\sum_{p_1,q_1,p_2,q_2\in\mathbb{Z}^{+}}}
_{p_1\leq m_1,p_2\leq  m_2}\sum_{j=1}^{d(p_1,q_1,p_2,q_2)}f^j_{p_1q_1p_2q_2}
(r_1,r_2)Y^j_{p_1q_1p_2q_2}(\omega^{1},\omega^{2}),\eea where the right hand side 
converges in $L^{2}(S^{2n_1-1}\times S^{2n_2-1}).$ Here $$f^j_{p_1q_1p_2q_2}
(r_1,r_2)=\int_{S^{2n_1-1}}\int_{S^{2n_2-1}}f(r_1 
\omega^1,r_2\omega^2)\overline{Y^j_{p_1q_1p_2q_2}(\omega^1,\omega^2)}d\omega^{1}d\omega^2.$$ By a 
representation theoretic argument it can be shown that (see the proof of Proposition 
4.5 in \cite{T}) if $\delta=\delta^{1}_{p_1q_1}\otimes\delta^{2}_{p_2q_2}$, 
$z=(r_1\omega^1,r_2\omega^2)$, $$f^j_{p_1q_1p_2q_2}(r_1,r_2)Y^j_{p_1q_1p_2q_2}
(\omega^{1},\omega^{2})=d(\delta)\int_{K}f(k\cdot 
z)\big(\delta(k^{-1})Y^j_{p_1q_1p_2q_2},Y^j_{p_1q_1p_2q_2}\big)dk .$$ Hence we can 
conclude that each $f^j_{p_1q_1p_2q_2}(r_1,r_2)Y^j_{p_1q_1p_2q_2}(\omega^1,
\omega^2)\in\mathcal{E}^{\lambda}(\C)$ is a joint eigenfunction of all 
$D\in\mathcal{L}_{K}^{\lambda}(\C)$ with eigenvalue $\mu^{\lambda}_{m_1m_2}.$ But  
then by Corollary \textbf{\ref{eqn c9.6}} \textbf{(a)}, it follows that 
$\big(\textup{for}~z=(z^1,z^2)=(r_1\omega^1,r_2\omega^2)\big)$ $$f^j_{p_1q_1p_2q_2}
(r_1,r_2)Y^j_{p_1q_1p_2q_2}(\omega^1,\omega^2)=a^j_{p_1q_1p_2q_2}P^j_{p_1q_1p_2q_2}
(z)\prod_{i=1}^{2}(-1)^{q_{i}}(2|\lambda|)^{p_{i}+q_{i}}\varphi_{m_{i}-p_i,
\lambda}^{n_{i}+p_i+q_i-1}(z^{i}),$$ 
for some constant $a^j_{p_1q_1p_2q_2}$. Hence (\ref{eqn 9.31}) follows from Corollary 
\textbf{\ref{eqn c9.6}} \textbf{(b)}. Since $Y^{j}_{p_1q_1p_2q_2}$ and 
$P^j_{p_1q_1p_2q_2}$ are related by (\ref{eqn  9.30}), from the above equation we get 
$$f^j_{p_1q_1p_2q_2}
(r_1,r_2)=a^j_{p_1q_1p_2q_2}b^j_{p_1q_1p_2q_2}\prod_{i=1}^{2}r_i^{p_i+q_i}
L_{m_{i}-p_i}^{n_{i}+p_i+q_i-1}(2|\lambda|r_i^2)e^{-|\lambda|r_i^2},$$ where 
$$b^j_{p_1q_1p_2q_2}=\prod_{i=1}^{2}(-1)^{q_{i}}(2|
\lambda|)^{p_{i}+q_{i}}\sqrt{\Gamma(n_i)(2|\lambda|)^{-(p_i+q_i)}\frac{\Gamma(m_i-
p_i+1)}{\Gamma(m_i+n_i+q_i)}}.$$ Now fix $r_1,r_2>0$. Since 
$$\mathop{\sum_{p_1,q_1,p_2,q_2\in\mathbb{Z}^{+}}}_{p_1\leq m_1,p_2\leq  m_2}|
f^j_{p_1q_1p_2q_2}(r_1,r_2)|^2=\big|\big|f(r_1\omega^1,r_2\omega^2)\big|\big|
^{2}_{L^2(S^{2n_1-1}\times S^{2n_2-1})}<\infty,$$ 
$$\mathop{\sum_{p_1,q_1,p_2,q_2\in\mathbb{Z}^{+}}}_{p_1\leq m_1,p_2\leq  m_2}\bigg[|
a^j_{p_1q_1p_2q_2}|^2\prod_{i=1}^2\frac{(2|\lambda|r_i^2)^{q_i}}
{\Gamma(n_i+p_i+q_i)}\bigg]\prod_{i=1}^2\frac{\Gamma(n_i+p_i+q_i)}
{\Gamma(m_i+n_i+q_i)}\big(L_{m_i-p_i}^{n_i+p_i+q_i-1}(2|\lambda|r_i^2)\big)^2<\infty.
$$ Therefore to prove (\ref{eqn 9.32}), it is enough to show that for large $q_1,q_2$, 
\bea\label{eqn 9.34}\prod_{i=1}^2\frac{\Gamma(n_i+p_i+q_i)}
{\Gamma(m_i+n_i+q_i)}\big(L_{m_i-p_i}^{n_i+p_i+q_i-1}(2|\lambda|r_i^2)\big)^2>c\eea 
for all $p_i\leq m_i.$ Now if $\alpha+1>2kt$, then $$|L_{k}^{\alpha}(t)|\geq\frac{1}
{2}\frac{\Gamma(k+\alpha+1)}{\Gamma(k+1)\Gamma(\alpha+1)}.$$ So when $n_i+q_i>2m_i(2|
\lambda|r^2_i)$, then $n_i+p_i+q_i>2(m_i-p_i)(2|\lambda|r^2_i)$ for all $p_i\leq m_i$ 
and hence $$L_{m_i-p_i}^{n_i+p_i+q_i-1}(2|\lambda|r_i^2)\geq\frac{1}
{2}\frac{\Gamma(m_i+q_i+n_i)}{\Gamma(m_i-p_i+1)\Gamma(n_i+p_i+q_i)}\geq \frac{1}
{2}\frac{\Gamma(m_i+q_i+n_i)}{\Gamma(m_i+1)\Gamma(n_i+p_i+q_i)}.$$ Therefore for all 
$q_i>2m_i(2|\lambda|r^2_i)-n_i$ and $p_i\leq m_i$, we have $$\big(L_{m_i-
p_i}^{n_i+p_i+q_i-1}(2|\lambda|r_i^2)\big)^2\geq \frac{1}{4}\frac{\Gamma(m_i+q_i+n_i)}
{\Gamma(m_i+1)^2\Gamma(n_i+p_i+q_i)},$$ which implies (\ref{eqn 9.34}). Hence the 
proof is complete. 
\end{proof}

Following \cite{T}, for each positive integer $k$, we define $\mathcal{B}_{k}$ to be 
the subspace of operators $S\in\mathcal{O}^{\lambda}(V_{m_1m_2})$ for which 
$$\mathop{\sum_{p_1,q_1,p_2,q_2\in\mathbb{Z}^{+}}}_{p_1\leq m_1,p_2\leq  m_2}||
\mathcal{P}_{\mathcal{W}^{\lambda}(H^1_{p_1q_1}\otimes H^2_{p_2q_2})}S||
_{m_1m_2}^2\prod_{i=1}^2\frac{\Gamma(n_i+p_i+q_i)}{k^{q_i}}<\infty,$$ where 
$\mathcal{P}_{\mathcal{W}^{\lambda}(H^1_{p_1q_1}\otimes H^2_{p_2q_2})}$ is the 
projection on $\mathcal{W}^{\lambda}(H^1_{p_1q_1}\otimes H^2_{p_2q_2})$, that is 
$$\mathcal{P}_{\mathcal{W}^{\lambda}(H^1_{p_1q_1}\otimes 
H^2_{p_2q_2})}S=\sum_{j=1}^{d(p_1,q_1,p_2,q_2)}\big\langle S,\mathcal{W}^{\lambda}
(P^j_{p_1q_1p_2q_2})\big\rangle^{\lambda}_{m_1m_2}\mathcal{W}^{\lambda}
(P^j_{p_1q_1p_2q_2}).$$ Then $\mathcal{B}_{k}$ becomes a Hilbert space if we define 
the inner product as $$\langle 
S_1,S_2\rangle_{\mathcal{B}_{k}}=\mathop{\sum_{p_1,q_1,p_2,q_2\in\mathbb{Z}^{+}}}_{p_1\leq 
m_1,p_2\leq  m_2}\big\langle\mathcal{P}_{\mathcal{W}^{\lambda}(H^1_{p_1q_1}\otimes 
H^2_{p_2q_2})}S_1,\mathcal{P}_{\mathcal{W}^{\lambda}(H^1_{p_1q_1}\otimes 
H^2_{p_2q_2})}S_2\big\rangle^{\lambda}_{m_1m_2}\prod_{i=1}^2\frac{\Gamma(n_i+p_i+q_i)}
{k^{q_i}}.$$ Note that for each $k\in\mathbb{N}$, 
$\mathcal{B}_{k}\subset\mathcal{B}_{k+1}$ and the inclusion 
$\mathcal{B}_{k}\hookrightarrow\mathcal{B}_{k+1}$ is continuous. We define 
$\mathcal{B}=\cup_{k\in\mathbb{N}}\mathcal{B}_{k}$ and equip it with the inductive 
limit topology.

\begin{lem}\label{eqn l9.9}
For each fixed $z\in\C$, $\Pi^{\lambda}(z)\in\mathcal{B}.$
\end{lem}

\begin{proof}
Fix $z\in\C$. \beas||\mathcal{P}_{\mathcal{W}(H^1_{p_1q_1}\otimes 
H^2_{p_2q_2})}\Pi(z)||^2_{m_1m_2}&=& \sum_{j=1}^{d(p_1,q_1,p_2,q_2)}\big|
\big\langle\Pi(z),\mathcal{W}(P^j_{p_1q_1p_2q_2})\big\rangle_{m_1m_2}\big|
^2\\&=&\sum_{j=1}^{d(p_1,q_1,p_2,q_2)}\bigg|P^j_{p_1q_1p_2q_2}(z)\prod_{i=1}^{2}
(-1)^{q_{i}}(2|\lambda|)^{p_{i}+q_{i}}\varphi_{m_{i}-p_i,\lambda}^{n_{i}+p_i+q_i-1}
(z^{i})\bigg|^2.\eeas Since $|Y^j_{p_1q_1p_2q_2}(\omega^1,\omega^2)|\leq 
c_1\prod_{i=1}^2(p_i+q_i)^{n_i-1}$ and $$\big|\varphi_{m_{i}-p_i,
\lambda}^{n_{i}+p_i+q_i-1}(z^{i})\big|\leq c_2\frac{\Gamma(m_i+n_i+q_i)}{\Gamma(m_i-
p_i+1)\Gamma(n_i+p_i+q_i)},$$ we get \beas||\mathcal{P}_{\mathcal{W}
(H^1_{p_1q_1}\otimes H^2_{p_2q_2})}\Pi(z)||^2_{m_1m_2}&\leq & c_3 
\prod_{i=1}^2(p_i+q_i)^{2n_i-2}\frac{\Gamma(m_i+n_i+q_i)(2|\lambda|r_i^2)^{p_i+q_i}}
{\Gamma(m_i-p_i+1)\big(\Gamma(n_i+p_i+q_i)\big)^2}\\&\leq & c_3 
\prod_{i=1}^2(p_i+q_i)^{2n_i}\frac{(m_i+n_i+q_i)^{m_i-p_i}(2|\lambda|r_i^2)^{p_i+q_i}}
{\Gamma(m_i-p_i+1)\Gamma(n_i+p_i+q_i)}\\&\leq & c^{\prime}_3 
\prod_{i=1}^2\frac{q_i^{2n_i+m_i}(2|\lambda|r_i^2)^{q_i}}
{\Gamma(n_i+p_i+q_i)},~\textup{for all}~q_i\in\mathbb{Z}^{+},~p_i\leq m_i.\eeas Now 
choose $k$ such that $2|\lambda|r_i^2\leq \frac{k}{2}.$ Then $$||
\mathcal{P}_{\mathcal{W}(H^1_{p_1q_1}\otimes H^2_{p_2q_2})}\Pi(z)||^2_{m_1m_2}\leq 
c^{\prime}_3 \prod_{i=1}^2\frac{q_i^{2n_i+m_i}}{2^{q_i}}\frac{k^{q_i}}
{\Gamma(n_i+p_i+q_i)},$$ which implies $$\big|\big|\Pi(z)\big|\big|
^2_{\mathcal{B}_{k}}=\leq c^{\prime}_3 
\mathop{\sum_{p_1,q_1,p_2,q_2\in\mathbb{Z}^{+}}}_{p_1\leq m_1,p_2\leq  
m_2}\prod_{i=1}^2\frac{q_i^{2n_i+m_i}}{2^{q_i}}<\infty.$$ Therefore 
$\Pi(z)\in\mathcal{B}_{k}$. Hence the proof. 
\end{proof}

\begin{lem}
Let $\mathcal{B}^{*}$ be the dual of $\mathcal{B}$. If $\upsilon\in\mathcal{B}^{*}$ 
then {\small\bea\label{eqn 9.35}||\upsilon||
_{k}^2:=\mathop{\sum_{p_1,q_1,p_2,q_2\in\mathbb{Z}^{+}}}_{p_1\leq m_1,p_2\leq  
m_2}\sum_{j=1}^{d(p_1,q_1,p_2,q_2)}\big|\upsilon\big(\mathcal{W}^{\lambda}
(P^{j}_{p_1q_1p_2q_2})\big)\big|^{2}\prod_{i=1}^{2}\frac{k^{q_i}}
{\Gamma(n_i+p_i+q_i)}<\infty,\eea} for all $k\in\mathbb{N}.$ Conversely if the 
constants $a^j_{p_1q_1p_2q_2}$'s satisfy \bea\label{eqn 
9.36}\mathop{\sum_{p_1,q_1,p_2,q_2\in\mathbb{Z}^{+}}}_{p_1\leq m_1,p_2\leq  
m_2}\sum_{j=1}^{d(p_1,q_1,p_2,q_2)}|a^j_{p_1q_1p_2q_2}|
^{2}\prod_{i=1}^{2}\frac{k^{q_i}}{\Gamma(n_i+p_i+q_i)}<\infty\eea for all 
$k\in\mathbb{N}$, then there is a unique $\upsilon\in\mathcal{B}^{*}$ such that 
$\upsilon\big(\mathcal{W}^{\lambda}(P^{j}_{p_1q_1p_2q_2})\big)=a^j_{p_1q_1p_2q_2}.$
\end{lem}

\begin{proof}
Since the topology on $\mathcal{B}$ is the inductive limit topology, 
$\upsilon\in\mathcal{B}^{*}$ if and only if $\upsilon\in \mathcal{B}_{k}^{*}$ for all 
$k$. Fix a $k$. 
Then as $\mathcal{B}_{k}$ is a Hilbert 
space, there exists $S_k\in\mathcal{B}_{k}$ such that $\upsilon(S)=\big\langle 
S,S_k\big\rangle_{\mathcal{B}_{k}}$ for all $S\in\mathcal{B}_{k}.$ Taking 
$S=\mathcal{W}(P^j_{p_1q_1p_2q_2})$, we get $$\upsilon\big(\mathcal{W}
(P^j_{p_1q_1p_2q_2})\big)=\big\langle \mathcal{W}
(P^j_{p_1q_1p_2q_2}),S_k\big\rangle_{\mathcal{B}_{k}}=\big\langle \mathcal{W}
(P^j_{p_1q_1p_2q_2}),S_k\big\rangle_{m_1m_2}\prod_{i=1}^{2}\frac{\Gamma(n_i+p_i+q_i)}
{k^{q_i}}.$$ Since $S_k\in\mathcal{B}_{k}$, 
$$\mathop{\sum_{p_1,q_1,p_2,q_2\in\mathbb{Z}^{+}}}_{p_1\leq m_1,p_2\leq  
m_2}\sum_{j=1}^{d(p_1,q_1,p_2,q_2)}\big|\big\langle S_k,\mathcal{W}
(P^j_{p_1q_1p_2q_2})\big\rangle_{m_1m_2}\big|^2\prod_{i=1}^2\frac{\Gamma(n_i+p_i+q_i)}
{k^{q_i}}<\infty .$$ Hence (\ref{eqn 9.35}) follows. Conversely, let the constants 
$a^{j}_{p_1q_1p_2q_2}$'s satisfy (\ref{eqn 9.36}). Then we can define an operator 
$S_{k}\in\mathcal{B}_{k}$ by $$\big\langle \mathcal{W}
(P^j_{p_1q_1p_2q_2}),S_k\big\rangle_{m_1m_2}=a^j_{p_1q_1p_2q_2}\prod_{i=1}^2\frac{k^{q_i}}
{\Gamma(n_i+p_i+q_i)}.$$ For each $k\in\mathbb{N}$, define $\upsilon_{k}\in\mathcal{B}_{k}^{*}$, by $\upsilon_{k}
(S)=\big\langle S,S_k\big\rangle_{\mathcal{B}_{k}}$ for all $S\in\mathcal{B}_{k}$.
Note that $$\upsilon_{k}(S)=\mathop{\sum_{p_1,q_1,p_2,q_2\in\mathbb{Z}^{+}}}_{p_1\leq 
m_1,p_2\leq  m_2}\sum_{j=1}^{d(p_1,q_1,p_2,q_2)}a^j_{p_1q_1p_2q_2}\big\langle S,
\mathcal{W}(P^j_{p_1q_1p_2q_2})\big\rangle_{m_1m_2},~s\in\mathcal{B}_k.$$ Therefore 
for any $S\in\mathcal{B}$, if we define $\upsilon(S)$ to be equal to the right hand 
side of the above equation then 
$\upsilon\mid_{B_k}=\upsilon_{k}\in\mathcal{B}_{k}^{*}$. Hence 
$\upsilon\in\mathcal{B}^{*}$. Also note that $\upsilon(\mathcal{W}
(P^j_{p_1q_1p_2q_2}))=a^j_{p_1q_1p_2q_2}$. Uniqueness of $\upsilon$ follows from the 
fact that $$\bigg\{\sqrt{\prod_{i=1}^2\frac{k^{q_i}}{\Gamma(n_i+p_i+q_i)}}~\mathcal{W}
(P^j_{p_1q_1p_2q_2}):j=1,2,\cdots,d(p_1,q_1,p_2,q_2);p_i\leq 
m_i,q_i\in\mathbb{Z}^{+}\bigg\}$$ forms an orthonormal basis for $\mathcal{B}_k$. 
Hence the proof is complete.   
\end{proof}

\begin{thm}\label{eqn t9.11}
Let $f$ be a joint eigenfunction of all $D\in\mathcal{L}_{K}^{\lambda}(\C)$ with eigenvalue $\mu^{\lambda}_{m_1m_2}$ such that $\chi_{\delta}* f\in\mathcal{E}^{\lambda}(\C)$ for all $\delta\in\widehat{K}_{M}$. Then $f(z)=\upsilon\big(\Pi^{\lambda}(z)\big)$ for a unique $\nu\in\mathcal{B}^{*}$
Conversely, if $f(z)=\upsilon\big(\Pi^{\lambda}(z)\big)$ for some  $\nu\in\mathcal{B}^{*}$, then $f$ is a joint eigenfunction of all $D\in\mathcal{L}_{K}^{\lambda}(\C)$ with eigenvalue $\mu^{\lambda}_{m_1m_2}$ and $\chi_{\delta}* f\in\mathcal{E}^{\lambda}(\C)$ for all $\delta\in\widehat{K}_{M}$.

\end{thm}

\begin{proof}
Let $\upsilon\in\mathcal{B}^{*}$ and $f(z)=\upsilon\big(\Pi(z)\big)$. We claim that 
$$f(z)=\mathop{\sum_{p_1,q_1,p_2,q_2\in\mathbb{Z}^{+}}}_{p_1\leq m_1,p_2\leq  
m_2}\sum_{j=1}^{d(p_1,q_1,p_2,q_2)}\upsilon\big(\mathcal{W}
(P^{j}_{p_1q_1p_2q_2})\big)\big\langle\Pi(z),\mathcal{W}
(P^{j}_{p_1q_1p_2q_2})\big\rangle_{m_1m_2},$$ where the right hand side converges 
absolutely and uniformly over every compact subset of $\C$. To prove the claim fix $r_i>0$. 
Then the proof of Lemma \textbf{\ref{eqn l9.9}} shows that there exist 
$k\in\mathbb{N}$ (depending on $r_i$) such that $\Pi(z)\in\mathcal{B}_{k}$ and $||
\Pi(z)||_{\mathcal{B}_{k}}<c$ for all $z\in\C$ with $|z^i|\leq r_i$. Since 
$\Pi(z)\in\mathcal{B}_{k}$, it follows that 
$$\mathop{\sum_{p_1,q_1,p_2,q_2\in\mathbb{Z}^{+}}}_{p_1\leq m_1,p_2\leq  
m_2}\sum_{j=1}^{d(p_1,q_1,p_2,q_2)}\big\langle\Pi(z),\mathcal{W}
(P^{j}_{p_1q_1p_2q_2})\big\rangle_{m_1m_2}\mathcal{W}(P^{j}_{p_1q_1p_2q_2})$$ 
converges to $\Pi(z)$ in the Hilbert space $\mathcal{B}_{k}$. Since 
$\upsilon\in\mathcal{B}^{*}_{k}$, we get \bea\label{eqn 
9.37}\upsilon\big(\Pi(z)\big)=\mathop{\sum_{p_1,q_1,p_2,q_2\in\mathbb{Z}^{+}}}_{p_1\leq 
m_1,p_2\leq  m_2}\sum_{j=1}^{d(p_1,q_1,p_2,q_2)}\upsilon\big(\mathcal{W}
(P^{j}_{p_1q_1p_2q_2})\big)\big\langle\Pi(z),\mathcal{W}
(P^{j}_{p_1q_1p_2q_2})\big\rangle_{m_1m_2}.\eea 
Multiply 
$\upsilon\big(\mathcal{W}_{m_1m_2}(P^{j}_{p_1q_1p_2q_2})\big)$ by 
$\Pi_{i=1}^2k^{q_i}\big(\Gamma(n_i+p_i+q_i)\big)^{-1}$, $\big\langle\Pi(z),\mathcal{W}
(P^{j}_{p_1q_1p_2q_2})\big\rangle_{m_1m_2}$ by $\Pi_{i=1}^2k^{-
q_i}\Gamma(n_i+p_i+q_i)$ and then use the Cauchy-Schwarz inequality to get \beas 
&&\mathop{\sum_{p_1,q_1,p_2,q_2\in\mathbb{Z}^{+}}}_{p_1\leq m_1,p_2\leq  
m_2}\sum_{j=1}^{d(p_1,q_1,p_2,q_2)}\bigg|\upsilon\big(\mathcal{W}
(P^{j}_{p_1q_1p_2q_2})\big)\big\langle\Pi(z),\mathcal{W}
(P^{j}_{p_1q_1p_2q_2})\big\rangle_{m_1m_2}\bigg|\\&=&\leq ||\upsilon||_{k}||\Pi(z)||
_{\mathcal{B}_{k}}\leq c ||\upsilon||_{k}\eeas for all $z\in\C$ such that 
$|z^i|\leq r_i$. Since $r_i>0$ was arbitrary, the claim follows. In particular $f$ is a 
smooth function. Since any $D\in\mathcal{L}_{K}(\C)$ is a polynomial coefficient 
differential operator we have 
$$Df(z)=\mathop{\sum_{p_1,q_1,p_2,q_2\in\mathbb{Z}^{+}}}_{p_1\leq m_1,p_2\leq  
m_2}\sum_{j=1}^{d(p_1,q_1,p_2,q_2)}\upsilon\big(\mathcal{W}
(P^{j}_{p_1q_1p_2q_2})\big)D\bigg[\big\langle\Pi(z),\mathcal{W}
(P^{j}_{p_1q_1p_2q_2})\big\rangle_{m_1m_2}\bigg]$$ in the distribution sense. But 
$$D\bigg[\big\langle\Pi(z),\mathcal{W}
(P^{j}_{p_1q_1p_2q_2})\big\rangle_{m_1m_2}\bigg]=\mu_{m_1m_2}(D)\big\langle\Pi(z),
\mathcal{W}(P^{j}_{p_1q_1p_2q_2})\big\rangle_{m_1m_2}.$$ Therefore we can conclude 
that $Df=\mu_{m_1m_2}(D)f$. Hence $f$ is a joint eigenfunction of all 
$D\in\mathcal{L}_{K}(\C)$ with eigenvalue $\mu_{m_1m_2}$. Now, if $\stackrel{\vee}{\delta}=\delta^1_{p_1q_1}\otimes\delta^2_{p_2q_2}$, equation (\ref{eqn 9.37}) implies that, $\chi_{\delta}*f=0$ if $p_1>m_1$ or $p_2>m_2$; and when $p_i\leq m_i$ for $i=1,2$, $$\chi_{\delta}*f=\frac{1}{d(p_1,q_1,p_2,q_2)}\sum_{j=1}^{d(p_1,q_1,p_2,q_2)}\upsilon\big(\mathcal{W}
(P^{j}_{p_1q_1p_2q_2})\big)\big\langle\Pi(z),\mathcal{W}
(P^{j}_{p_1q_1p_2q_2})\big\rangle_{m_1m_2}.$$ But, by Proposition \textbf{\ref{eqn p7.14}}, $\big\langle\Pi(z),\mathcal{W}
(P^{j}_{p_1q_1p_2q_2})\big\rangle_{m_1m_2}=\theta(P^j_{p_1q_1p_2q_2})\psi_{m_1m_2}$ which clearly equals to $e^{-|\lambda||z|^2}$ times a polynomial. Hence it follows that $\chi_{\delta}* f\in\mathcal{E}^{\lambda}(\C)$. 

Conversely let $f$ be a joint eigenfunction of all $D\in\mathcal{L}_{K}(\C)$ with 
eigenvalue $\mu_{m_1m_2}$ such that $\chi_{\delta}*f\in\mathcal{E}^{\lambda}(\C)$ for 
each $\delta\in\widehat{K}_{M}$. By Lemma \textbf{\ref{eqn l9.8}}, there exist constants 
$a^{j}_{p_1q_1p_2q_2}$ such that 
$$f(z)=\mathop{\sum_{p_1,q_1,p_2,q_2\in\mathbb{Z}^{+}}}_{p_1\leq m_1,p_2\leq  
m_2}\sum_{j=1}^{d(p_1,q_1,p_2,q_2)}a^{j}_{p_1q_1p_2q_2}
\big\langle\Pi(z),\mathcal{W}(P^{j}_{p_1q_1p_2q_2})\big\rangle_{m_1 m_2},$$ and 
$a^j_{p_1q_1p_2q_2}$'s satisfy the following : 
$$\mathop{\sum_{p_1,q_1,p_2,q_2\in\mathbb{Z}^{+}}}_{p_1\leq m_1,p_2\leq  
m_2}\sum_{j=1}^{d(p_1q_1p_2q_2)}|a^j_{p_1q_1p_2q_2}|^{2}\prod_{i=1}^{2}\frac{k^{q_i}}
{\Gamma(n_i+p_i+q_i)}<\infty,~\forall k\in\mathbb{N}.$$ Then by the previous lemma 
there exists $\upsilon\in\mathcal{B}^{*}$ such that $\upsilon\big(\mathcal{W}
(P^j_{p_1q_1p_2q_2})\big)=a^j_{p_1q_1p_2q_2},$ and consequently by (\ref{eqn 9.37}), 
$f(z)=\upsilon\big(\Pi(z)\big)$. 

Now we prove the uniqueness of $\upsilon$ which will complete the proof of the 
theorem. So let $\upsilon\in\mathcal{B}^{*}$ and $\upsilon\big(\Pi(z)\big)=0$ for all 
$z\in\C$. We must prove that $\upsilon=0$. It is enough to show that 
$\upsilon\big(\mathcal{W}(P^j_{p_1q_1p_2q_2})\big)=0$ for all $j=1,2,
\cdots,d(p_1,q_1,p_2,q_2); p_i\leq m_i,q_i\in\mathbb{Z}^{+}$. But this follows, since 
(\ref{eqn 9.37}) and Corollary \textbf{\ref{eqn c9.6}} \textbf{(b)} imply that for 
each fixed $r_1,r_2>0$, \beas  
&&\big\langle\upsilon\big(\Pi(r_1\cdot,r_2\cdot)\big),Y^j_{p_1q_1p_2q_2}\big\rangle_
{L^2(S^{2n_1-1}\times S^{2n_2-1})}
\\&=& b^j_{p_1q_1p_2q_2}\upsilon\big(\mathcal{W}
(P^{j}_{p_1q_1p_2q_2})\big)\prod_{i=1}^2r_i^{p_i+q_i}L^{n_i+p_i+q_i-1}_{m_i-p_i,
\lambda}\big(2|\lambda|r_i^2\big)e^{-|\lambda|r_i^2}\eeas for some non zero constants 
$b^j_{p_1q_1p_2q_2}.$           
\end{proof}

We have already mentioned that the above characterization is analogues to the view 
point of Thangavelu \cite{T} (see Theorem 4.1 there). Now we make this analogy clear 
by showing that the above theorem can be reformulated (Theorem \textbf{\ref{eqn 
t9.12}} below), which is similar to Theorem 4.1 in \cite{T}. Consider 
$$L^{2}_{m_1,m_2}(S^{2n_1-1}\times 
S^{2n_2-1}):=\overline{\textup{span}}\big\{Y^j_{p_1q_1p_2q_2}:j=1,2,
\cdots,d(p_1,q_1,p_2,q_2);p_i\leq m_i,q_i\in\mathbb{Z}^{+}\big\}$$ as Hilbert subspace 
of $L^{2}(S^{2n_1-1}\times S^{2n_2-1})$. Then the map $$\mathcal{I}:
\mathcal{O}^{\lambda}(V_{m_1m_2})\rightarrow L^{2}_{m_1,m_2}(S^{2n_1-1}\times 
S^{2n_2-1})$$ defined by $$\mathcal{I}\big(\mathcal{W}^{\lambda}
(P^{j}_{p_1q_1p_2q_2})\big)=Y^j_{p_1q_1p_2q_2}$$ is an Hilbert space isomorphism. Note 
that $\mathcal{I}(\mathcal{B}_k)$ is the subspace of all functions $\phi$ in 
$L^{2}_{m_1,m_2}(S^{2n_1-1}\times S^{2n_2-1})$ such that $$ 
\mathop{\sum_{p_1,q_1,p_2,q_2\in\mathbb{Z}^{+}}}_{p_1\leq m_1,p_2\leq  m_2}\big|\big|
\phi_{\delta_{p_1q_2}\otimes \delta_{p_2q_2}}\big|\big|
^{2}\prod_{i=1}^2\frac{\Gamma(n_i+p_i+q_i)}{k^{q_i}}<\infty,$$ where, for $\omega\in 
S^{2n_1-1}\times S^{2n_2-1}$, \beas\phi_{\delta_{p_1q_2}\otimes \delta_{p_2q_2}}
(\omega)&:=&d(p_1,q_1,p_2,q_2)[\chi_{\delta_{p_1q_1}\otimes \delta_{p_2q_2}}*\phi]
(\omega)\\&=&d(p_1,q_1,p_2,q_2)\int_{K}\chi_{\delta_{p_1q_1}\otimes \delta_{p_2q_2}}
(k)\phi(k^{-1}\cdot\omega)dk\\&=&\sum_{j=1}^{d(p_1,q_1,p_2,q_2)}\langle\phi,
Y^j_{p_1q_1p_2q_2}\rangle Y^j_{p_1q_1p_2q_2}(\omega).\eeas Each $\mathcal{I}
(\mathcal{B}_{k})$ becomes a Hilbert space with the inner product $$\langle\phi_1,
\phi_2\rangle_{\mathcal{I}(\mathcal{B}_{k})}=\big\langle\mathcal{I}^{-1}\phi_1,
\mathcal{I}^{-1}\phi_2\big\rangle_{\mathcal{B}_k}.$$ Explicitly  $$\langle\phi_1,
\phi_2\rangle_{\mathcal{I}
(\mathcal{B}_{k})}=\mathop{\sum_{p_1,q_1,p_2,q_2\in\mathbb{Z}^{+}}}_{p_1\leq 
m_1,p_2\leq  m_2}\big\langle(\phi_1)_{\delta_{p_1q_2}\otimes \delta_{p_2q_2}},
(\phi_2)_{\delta_{p_1q_2}\otimes 
\delta_{p_2q_2}}\big\rangle\prod_{i=1}^2\frac{\Gamma(n_i+p_i+q_i)}{k^{q_i}}.$$ 
Consider $\mathcal{I}(\mathcal{B})=\cup_{k\in\mathbb{N}}\mathcal{I}(\mathcal{B}_k)$ 
and equip this space with the inductive limit topology. Let {\small\bea\label{eqn 
9.38}\mathcal{P}^{\lambda}_{m_1m_2}(z,
\omega)=\mathop{\sum_{p_1,q_1,p_2,q_2\in\mathbb{Z}^{+}}}_{p_1\leq m_1,p_2\leq  
m_2}\sum_{j=1}^{d(p_1,q_1,p_2,q_2)}\big\langle\Pi^{\lambda}(z),\mathcal{W}^{\lambda}
(P^j_{p_1q_1p_2q_2})\big\rangle^{\lambda}_{m_1m_2}Y^j_{p_1q_1p_2q_2}(\omega),\eea} 
$\omega\in S^{2n_1-1}\times S^{2n_2-1}.$ It is easy to see that $\mathcal{I}
(\Pi^{\lambda}(z))=\mathcal{P}^{\lambda}_{m_1m_2}(z,\cdot).$ Then one can show that 
Theorem \textbf{\ref{eqn t9.11}} is equivalent to the following theorem :

\begin{thm}\label{eqn t9.12}
Let $f$ be a joint eigenfunction of all $D\in\mathcal{L}_{K}^{\lambda}(\C)$ with eigenvalue $\mu^{\lambda}_{m_1m_2}$ such that $\chi_{\delta}* f\in\mathcal{E}^{\lambda}(\C)$ for all $\delta\in\widehat{K}_{M}$. Then $$f(z)=\int_{S^{2n_1-1}\times 
S^{2n_2-1}}\mathcal{P}^{\lambda}_{m_1m_2}(z,\omega)d\nu(\omega),$$ for a unique $\nu\in\mathcal{B}^{*}$
Conversely, if $$f(z)=\int_{S^{2n_1-1}\times 
S^{2n_2-1}}\mathcal{P}^{\lambda}_{m_1m_2}(z,\omega)d\nu(\omega),$$ for some  $\nu\in\mathcal{B}^{*}$, then $f$ is a joint eigenfunction of all $D\in\mathcal{L}_{K}^{\lambda}(\C)$ with eigenvalue $\mu^{\lambda}_{m_1m_2}$ and $\chi_{\delta}* f\in\mathcal{E}^{\lambda}(\C)$ for all $\delta\in\widehat{K}_{M}$.

\end{thm}  

The above theorem gives an integral representation of joint eigenfunctions, where the 
kernel $\mathcal{P}^{\lambda}_{m_1m_2}(z,\omega)$ is given by the series in (\ref{eqn 
9.38}). Now we shall give another integral representation, where the kernel can be 
given explicitly. Fix $a_1,a_2>0$ so that $$L^{n_i+p_i+q_i-1}_{m_i-p_i}\big(2|\lambda|
a_i^2\big)\neq 0$$ for all $p_i\leq m_i,q\in\mathbb{Z}^{+};i=1,2$. Define 
\beas\mathcal{Q}^{\lambda}_{m_1m_2}(z,\omega)&=& 
e^{-2i\lambda~\textup{Im}\big(z\cdot\overline{(a_1\omega^1,a_2\omega^2)}\big)}\psi^{\lambda}_{m_1m_2}\big(z-(a_1\omega^1,a_2\omega^2)
\big)\\&=&\pi^{-n}(2|\lambda|)^{n}\prod_{i=1}^2 e^{2i\lambda a_i~\textup{Im}
(z^i\cdot\bar{\omega^i})}\varphi_{m_i,\lambda}^{n_i-1}\big(z^i-a_i\omega^i\big),\eeas 
where $z=(z^1,z^2)\in\mathbb{C}^{n_1}\times \mathbb{C}^{n_2}$ and $\omega=(\omega^1,
\omega^2)\in S^{2n_1-1}\times S^{2n_2-1}.$ For each positive integer $k$, define 
$\mathcal{A}_{k}$ to be the subspace of functions $\phi$ in $L^{2}_{m_1m_2}
(S^{2n_1-1}\times S^{2n_2-1})$ for which 
$$\mathop{\sum_{p_1,q_1,p_2,q_2\in\mathbb{Z}^{+}}}_{p_1\leq m_1,p_2\leq  m_2}\big|
\big|\phi_{\delta_{p_1q_2}\otimes \delta_{p_2q_2}}\big|\big|
^{2}\prod_{i=1}^2\bigg[\frac{\Gamma(n_i+p_i+q_i)}{k^{q_i}}\bigg]^2<\infty.$$ Each 
$\mathcal{A}_{k}$ becomes a Hilbert space with the following inner product : 
$$\langle\phi_1,
\phi_2\rangle_{\mathcal{A}_{k}}=\mathop{\sum_{p_1,q_1,p_2,q_2\in\mathbb{Z}^{+}}}_{p_1\leq 
m_1,p_2\leq  m_2}\big\langle(\phi_1)_{\delta_{p_1q_2}\otimes \delta_{p_2q_2}},
(\phi_2)_{\delta_{p_1q_2}\otimes 
\delta_{p_2q_2}}\big\rangle\prod_{i=1}^2\bigg[\frac{\Gamma(n_i+p_i+q_i)}
{k^{q_i}}\bigg]^2.$$ We take $\mathcal{A}=\cup_{k\in\mathbb{N}}\mathcal{A}_k$ and 
equip it with the inductive limit topology. Let $\mathcal{A}^{*}$ be the dual of 
$\mathcal{A}$ with respect to this topology. Then we have the following integral 
representation of joint eigenfunctions of all $D\in\mathcal{L}^{\lambda}_{K}(\C).$ 

\begin{thm}
Let $f$ be a joint eigenfunction of all $D\in\mathcal{L}_{K}^{\lambda}(\C)$ with eigenvalue $\mu^{\lambda}_{m_1m_2}$ such that $\chi_{\delta}* f\in\mathcal{E}^{\lambda}(\C)$ for all $\delta\in\widehat{K}_{M}$. Then $$f(z)=\int_{S^{2n_1-1}\times 
S^{2n_2-1}}\mathcal{Q}^{\lambda}_{m_1m_2}(z,\omega)d\nu(\omega),$$ for a unique $\nu\in\mathcal{A}^{*}$
Conversely, if $$f(z)=\int_{S^{2n_1-1}\times 
S^{2n_2-1}}\mathcal{Q}^{\lambda}_{m_1m_2}(z,\omega)d\nu(\omega),$$ for some  $\nu\in\mathcal{A}^{*}$, then $f$ is a joint eigenfunction of all $D\in\mathcal{L}_{K}^{\lambda}(\C)$ with eigenvalue $\mu^{\lambda}_{m_1m_2}$ and $\chi_{\delta}* f\in\mathcal{E}^{\lambda}(\C)$ for all $\delta\in\widehat{K}_{M}$.

\end{thm} 

\begin{proof}
The theorem can be proved using arguments similar to the proof of Theorem 
\textbf{\ref{eqn t9.11}}, once we have the following claim : $$\int_{S^{2n_1-1}\times 
S^{2n_2-1}}\mathcal{Q}^{\lambda}_{m_1m_2}(z,\omega)Y^j_{p_1q_1p_2q_2}
(\omega)d\omega=c_{p_1q_1p_2q_2}\big(P^j_{p_1q_1p_2q_2}\big)^{\prime}
(z)\prod_{i=1}^2\varphi^{n_i+p_i+q_i-1}_{m_i-p_i,\lambda}(z^i),$$ where 
$c_{p_1q_1p_2q_2}=0$ if either $p_1>m_1$ or $p_2>m_2$ , and for $p_i\leq m_i$, it is 
given by \beas c_{p_1q_1p_2q_2}&=& \pi^{-n}(2|\lambda|)^{n}\prod_{i=1}^{2}(2|
\lambda|)^{p_i+q_i}\frac{\Gamma(n_i)\Gamma(m_i-p_i+1)}
{\Gamma(m_i+n_i+q_i)}\frac{a_i^{(p_i+q_i)}}{a_i^{2n_i-1}}L^{n_i+p_i+q_i-1}_{m_i-
p_i}\big(2|\lambda|a_i^2\big)e^{-|\lambda|a_i^2}.\eeas To prove the claim, first note 
that we can write \beas &&\int_{S^{2n_1-1}\times 
S^{2n_2-1}}\mathcal{Q}^{\lambda}_{m_1m_2}(z,\omega)Y^j_{p_1q_1p_2q_2}
(\omega)d\omega\\&=&\bigg[\prod_{i=1}^2\frac{1}
{a_i^{2n_i+p_i+q_i-1}}\bigg]\big(P^j_{p_1q_1p_2q_2}\big)^{\prime}d\mu_{a_1,a_2}
\times^{\lambda}\psi^{\lambda}_{m_1m_2}(z)\\&=& \left[\sqrt{\prod_{i=1}^{2}(2|
\lambda|)^{-(p_i+q_i)}\frac{\Gamma(n_i)\Gamma(m_i-p_i+1)}
{\Gamma(m_i+n_i+q_i)}}\bigg(a_i^{2n_i+p_i+q_i-1}\bigg)\right]^{-1}\\&\times& 
\left[P^j_{p_1q_1p_2q_2}
d\mu_{a_1,a_2}
\times^{\lambda}\psi^{\lambda}_{m_1m_2}(z)\right],\eeas where $d\mu_{a_1,a_2}$ is the 
surface measure on $a_1 S^{2n_1-1}\times a_2S^{2n_2-1}.$ But then the claim follows, 
if we can prove the following lemma.
\end{proof}

\begin{lem}
Let $a_1,a_2>0$ and $d\mu_{a_1,a_2}$ be the surface measure on $a_1S^{2n_1-1}\times 
a_2 S^{2n_2-1}$. Let $P\in H^{1}_{p_1q_1}\otimes H^{2}_{p_2q_2}$. Then 
$$Pd\mu_{a_1,a_2}\times^{\lambda}\psi^{\lambda}_{m_1m_2}=b_{p_1q_1p_2q_2}\pi^{-n}(2|
\lambda|)^{n}P(z)\prod_{i=1}^{2}(-1)^{q_{i}}(2|\lambda|)^{p_{i}+q_{i}}\varphi_{m_{i}-
p_i,\lambda}^{n_{i}+p_i+q_i-1}(z^{i}),$$ $$b_{p_1q_1p_2q_2}=\prod_{i=1}^{2}
(-1)^{q_i}\frac{\Gamma(n_i)\Gamma(m_i-p_i+1)}{\Gamma(m_i+n_i+q_i)} 
a_i^{2(p_i+q_i)}L^{n_i+p_i+q_i-1}_{m_i-p_i}\big(2|\lambda|a_i^2\big)e^{-|\lambda|
a_i^{2}},$$ if $p_{i}\leq m_{i}$ for all $i=1,2$; otherwise 
$Pd\mu_{a_1,a_2}\times^{\lambda}\psi^{\lambda}_{m_1m_2}=0$.  
\end{lem}

\begin{proof}
Let $\stackrel{\vee}{\delta}=\delta^1_{p_1q_1}\otimes \delta^2_{p_2q_2}$. Take 
$P^j_{p_1q_1p_2q_2}$, $j=1,2,\cdots,d(p_1,q_1,p_2,q_2)$, as Lemma \textbf{\ref{eqn 
l9.7}}, so that $\{P^j_{p_1q_1p_2q_2}:j=1,2,\cdots,d(p_1,q_1,p_2,q_2)\}$ forms a basis 
for $H_{\stackrel{\vee}{\delta}}$. Also we have $\big|\big|
\mathcal{W}\big(P^j_{p_1q_1p_2q_2}\big)\big|\big|_{m_1m_2}^{2}=1$ if $p_i\leq 
m_i,i=1,2.$ We can choose suitable bases $\textbf{b}$ for $V_{\stackrel{\vee}{\delta}}$ and 
$\textbf{e}$ for $F_{\delta}=\textup{Hom}_{K}(V_{\delta},H_{\delta})$ so that with 
respect to these bases $P^{\stackrel{\vee}{\delta}}:\C\rightarrow \mathcal{M}_{d(\delta)\times 
1}$ can be given as follows : $P^{\stackrel{\vee}{\delta}}_{j1}=P^j_{p_1q_1p_2q_2}$. Since 
$\Psi^{\stackrel{\vee}{\delta}}_{m_1m_2}=\theta(P^{\stackrel{\vee}{\delta}})\psi_{m_1m_2}$, by 
Proposition \textbf{\ref{eqn p9.5}}, we can say that, 
$\Psi^{\stackrel{\vee}{\delta}}_{m_1m_2}=\widetilde{\Psi}^{\stackrel{\vee}{\delta}}_{m_1m_2}$ if $p_i\leq 
m_i$ for all $i=1,2$; otherwise $\Psi^{\stackrel{\vee}{\delta}}_{m_1m_2}=0.$ Now let $p_i\leq 
m_i$ for $i=1,2$. Then \beas 
\widetilde{A}^{\stackrel{\vee}{\delta}}_{m_1m_2}=A^{\stackrel{\vee}{\delta}}_{m_1m_2}&=&\int_{\C}
[\Psi^{\stackrel{\vee}{\delta}}_{m_1m_2}(z)]^{\star}[\Psi^{\stackrel{\vee}{\delta}}_{m_1m_2}
(z)]dz\\&=&\sum_{j=1}^{d(p_1,q_1,p_2,q_2)}\big|\big|
\theta(P^j_{p_1q_1p_2q_2})\psi_{m_1m_2}\big|\big|^2_2\\&=& \pi^{-n}(2|
\lambda|)^{n}d(p_1,q_1,p_2,q_2),~\textup{by Lemma \textbf{\ref{eqn l7.5}} 
\textbf{(b)}},\eeas and $$\widetilde{L}^{\stackrel{\vee}{\delta}}_{m_1m_2}
(z)=\widetilde{L}^{\stackrel{\vee}{\delta}}_{m_1m_2}(z)=\pi^{-n}(2|\lambda|)^{n}\prod_{i=1}^{2}
(-1)^{q_{i}}(2|\lambda|)^{p_{i}+q_{i}}L_{m_{i}-p_i}^{n_{i}+p_i+q_i-1}(2|\lambda||z^i|
^2).$$ Also we have \beas\Upsilon_{\stackrel{\vee}{\delta}}(z)&=&[P^{\stackrel{\vee}{\delta}}
(z)]^{\star}[P^{\stackrel{\vee}{\delta}}(z)]\\&=&\sum_{j=1}^{d(p_1,q_1,p_2,q_2)}\big|
P^{j}_{p_1q_1p_2q_2}(z)\big|^2\\&=&\prod_{i=1}^{2}(2|
\lambda|)^{-(p_i+q_i)}\frac{\Gamma(n_i)\Gamma(m_i-p_i+1)}
{\Gamma(m_i+n_i+q_i)}r_i^{2(p_i+q_i)}\sum_{j=1}^{d(p_1,q_1,p_2,q_2)}\big|
Y^j_{p_1q_1p_2q_2}(\omega)\big|^2\\&=& \frac{d(p_1,q_1,p_2,q_2)}{\big|
a_1S^{2n_1-1}\times a_2S^{2n_2-1}\big|}\prod_{i=1}^{2}(2|
\lambda|)^{-(p_i+q_i)}\frac{\Gamma(n_i)\Gamma(m_i-p_i+1)}
{\Gamma(m_i+n_i+q_i)}r_i^{2(p_i+q_i)}.\eeas Therefore from Theorem \textbf{\ref{eqn 
t7.9}}, we can show that, for $p_i\leq m_i,i=1,2,$ 
$$P^{\stackrel{\vee}{\delta}}d\mu_{a_1,a_2}\times^{\lambda}
\psi^{\lambda}_{m_1m_2}=b_{p_1q_1p_2q_2}\Psi^{\stackrel{\vee}{\delta}}_{m_1m_2},$$ where 
$b_{p_1q_1p_2q_2}$ is given by $$  b_{p_1q_1p_2q_2}= \prod_{i=1}^{2}
(-1)^{q_i}\frac{\Gamma(n_i)\Gamma(m_i-p_i+1)}{\Gamma(m_i+n_i+q_i)} 
a_i^{2(p_i+q_i)}L^{n_i+p_i+q_i-1}_{m_i-p_i}\big(2|\lambda|a_i^2\big)e^{-|\lambda|
a_i^{2}}.$$ Hence the proof follows.
\end{proof}

\textbf{Acknowledgement.} The author wishes to thank Prof. E. K. Narayanan for suggesting this problem and for many useful discussions. The author is also grateful to Prof. S. Thangavelu and Prof. F. Ricci for several helpful discussions.

\end{document}